\documentclass[12pt,reqno]{amsart}
\usepackage[activeacute,english]{babel}

\usepackage[utf8]{inputenc}

\usepackage[english]{babel}
\usepackage{amsfonts,amsmath,amsthm,bm,amssymb}
\usepackage{graphics,tikz,float}

\usepackage{hyperref,cleveref}
\usepackage{autonum}

\DeclareRobustCommand{\SkipTocEntry}[5]{}

\usepackage[margin=0.9in]{geometry}
\newcommand{\EE}{{\mathcal E}}
\newcommand{\FF}{{\mathcal F}}

\newcommand{\LL}{{\mathcal L}}

\newcommand{\C}{{\mathbb C}}
\newcommand{\N}{{\mathbb N}}
\newcommand{\R}{{\mathbb R}}
\newcommand{\Z}{{\mathbb Z}}

\newcommand{\Sph}{{\mathbb S}}

\newcommand{\intL}{\int_{-L}^L}

\newcommand{\eper}{\operatorname{per}}

\newcommand{\intp}{\int_{-\pi}^{\pi}}

\newtheorem{theorem}{Theorem}[section]
\newtheorem{lemma}[theorem]{Lemma}
\newtheorem{corollary}[theorem]{Corollary}
\newtheorem{proposition}[theorem]{Proposition}

\newtheorem{remark}[theorem]{Remark}

\newtheorem{definition}[theorem]{Definition}

\newtheorem{Open problem}[theorem]{Open problem}
\numberwithin{equation}{section}

\allowdisplaybreaks[1]

\subjclass[2010]{35J61, 35B10, 35A15, 35S05.}
\keywords{}

\begin{document}

\title[Periodic solutions to integro-differential equations]
{Periodic solutions to integro-differential equations:
variational formulation, symmetry, and regularity}

\author[X. Cabré]{Xavier Cabr\'e}
\address{X. Cabr\'e \textsuperscript{1,2,3}
\newline
\textsuperscript{1} 
ICREA, Pg. Lluis Companys 23, 08010 Barcelona, Spain
\newline
\textsuperscript{2} 
Universitat Polit\`ecnica de Catalunya, Departament de Matem\`{a}tiques and IMTech, Diagonal 647, 08028 Barcelona, Spain
\newline
\textsuperscript{3} 
Centre de Recerca Matem\`atica, Edifici C, Campus Bellaterra, 08193 Bellaterra, Spain.}
\email{xavier.cabre@upc.edu}

\author[G. Csat\'o]{Gyula Csat\'o}
\address{G. Csat\'o \textsuperscript{1,2}
\newline
\textsuperscript{1}
Departament de Matem\`{a}tiques i Inform\`{a}tica,
Universitat de Barcelona,
Gran Via 585,
08007 Barcelona, Spain
\newline
\textsuperscript{2}
Centre de Recerca Matem\`atica, Edifici C, Campus Bellaterra, 08193 Bellaterra, Spain.}
\email{gyula.csato@ub.edu}

\author[A. Mas]{Albert Mas}
\address{A. Mas \textsuperscript{1,2}
\newline
\textsuperscript{1}
Departament de Matem\`atiques,
Universitat Polit\`ecnica de Catalunya,
Campus Diagonal Bes\`os, Edifici A (EEBE), Av. Eduard Maristany 16, 08019
Barcelona, Spain
\newline
\textsuperscript{2}
Centre de Recerca Matem\`atica, Edifici C, Campus Bellaterra, 08193 Bellaterra, Spain.}
\email{albert.mas.blesa@upc.edu}

\date{\today}
\thanks{The three authors are supported by the Spanish grants PID2021-123903NB-I00 and RED2018-102650-T funded by MCIN/AEI/10.13039/501100011033 and by ERDF ``A way of making Europe'', and by the Catalan grant 
2021-SGR-00087. The second author is in addition supported by the Spanish grant PID2021-125021NA-I00.
This work is supported by the Spanish State Research Agency, through the Severo Ochoa and Mar\'ia de Maeztu Program for Centers and Units of
Excellence in R\&D (CEX2020-001084-M)}

\begin{abstract}
We consider nonconstant periodic constrained minimizers of semilinear elliptic equations for integro-differential operators in $\mathbb{R}$. We prove that, after an appropriate translation, each of them is necessarily an even function which is decreasing in half its period. In particular, it has only two critical points in half its period, the absolute maximum and minimum. If these statements hold for all nonconstant periodic solutions, and not only for constrained minimizers, remains as an open problem. 

Our results apply to operators with kernels in two different classes: kernels $K$ which are convex and kernels for which
$K(\tau^{1/2})$ is a completely monotonic function of $\tau$. This last new class arose in our previous work on
nonlocal Delaunay surfaces in $\mathbb{R}^n$. Due to their symmetry of revolution, it gave rise to a 1d problem involving an operator with a nonconvex kernel. Our proofs are based on a not so well-known Riesz rearrangement inequality on the
circle~$\mathbb{S}^1$ established in 1976.

We also put in evidence a new regularity fact which is a truly nonlocal-semilinear effect and also occurs in the nonperiodic setting. Namely, for nonlinearities in $C^\beta$ and when $2s+\beta <1$ ($2s$ being the order of the operator), the solution is not always $C^{2s+\beta-\epsilon}$ for all $\epsilon >0$.

\bigskip

\end{abstract}

\maketitle



\section{Introduction}

\subsection{A new symmetry result for periodic solutions}
\label{ss subsection1.1}

It is well known that bounded solutions to the semilinear second order ODE $-u''=f(u)$ in $\R$ are ---up to a multiplicative factor~$\pm 1$ and up to translations--- either increasing in $\R$, or even with respect to $0$ and decreasing in $(0,+\infty)$, or periodic. This follows immediately from the fact that $u$ must be even with respect to any of its critical points\footnote{This is a consequence of the uniqueness theorem for ODEs: if $u'$ vanishes at $x=0$, say, then $v(x):=u(-x)$ solves the same equation and has same initial position and derivative at $x=0$ as $u$.} after considering the cases when $u'$ vanishes at none, only one, or at least two points. 

It is then natural to ask whether this classification  also holds true for the fractional Laplacian, that is, if, for every 
$0<s<1$, any bounded solution $u$ to 
\begin{equation}
 \label{eq:frac Lapl intro}
  (-\Delta)^s u=f(u) \quad\text{ in }\R
\end{equation}
belongs to one of the above mentioned three categories. Recall that
\begin{equation}\label{defi laplacian_}
(-\Delta)^s u(x):=c_s\, \text{P.V.}\!\int_\R dy\,
\frac{u(x)-u(y)}{|x-y|^{1+2s}},\qquad
c_s:=\frac{s4^s\Gamma(1/2+s)}{\sqrt\pi\,\Gamma(1-s)}.
\end{equation} 
Due to the lack of an analogue uniqueness result for  \eqref{eq:frac Lapl intro}, the aforementioned argument can not be carried out in the nonlocal case.  

We learned this question from Sol\`a-Morales when working on our joint paper \cite{Cabre Sola bdry reaction}. The paper dealt with equation \eqref{eq:frac Lapl intro} for $s=1/2$. To the best of our knowledge the conjecture is still wide open. It is only known to be true for the equation
\begin{equation}\label{eq:intro1.1:BenOno}
   (-\Delta)^{1/2}u=-u+u^2\quad\text{ in }\R
\end{equation} 
modelling traveling wave solutions to the Benjamin-Ono equation in hydrodynamics \cite{Benj,Ono} and for the Peierls-Nabarro equation 
\begin{equation}\label{eq:intro1.1:Peierls}
   (-\Delta)^{1/2}u=\sin(\pi u)\quad\text{ in }\R
\end{equation} 
from continuum modeling of dislocations in crystals \cite{Nab}. These are two really special equations since they are ``completely integrable''. More precisely,  Amick and Toland \cite{AmickToland} were able to find, with explicit expressions, all bounded solutions to 
\eqref{eq:intro1.1:BenOno}. This required $f(u)$ to be exactly $-u+u^2$ and $s=1/2$.\footnote{Their method consisted of identifying the equation as the boundary relation with respect to a local one in the upper half-plane which can be solved using complex analysis techniques.}
Later, in \cite{Toland} Toland also found all bounded solutions to the Peierls-Nabarro equation using the fact that a derivative of any solution is the difference of two solutions to the Benjamin-Ono equation. 

Both equations \eqref{eq:intro1.1:BenOno} and \eqref{eq:intro1.1:Peierls} admit infinitely many periodic solutions of different amplitude. Equation \eqref{eq:intro1.1:BenOno} admits, on the other hand, a ``ground state'' solution; that is, an even solution which is decreasing in $(0,+\infty)$. Instead, \eqref{eq:intro1.1:Peierls} admits a ``layer solution'', that is, a solution that is increasing in all of $\R$.

In \Cref{intro:fourier mult} we will describe other dispersive equations (besides the Benjamin-Ono equation) that lead, when looking for standing or traveling wave solutions, to periodic solutions of semilinear elliptic integro-differential equations fitting in our general framework below.

Back to the fractional Laplacian, for nonlinearities different than the two above there are few works that classify solutions or prove symmetry and monotonicity properties. Let us mention the main ones, including also some of those that establish Sturm-Liouville type results for integro-differential equations ---and thus the results are obtained in the absence of an ODE existence and uniqueness theorem. We focus mainly on results on the evenness and monotonicity of certain solutions, since the current paper concerns these properties in the periodic framework.

Chen, Li, and Ou \cite{Chen Li Ou} adapted the moving planes method to the fractional Laplacian and proved the even symmetry of positive solutions tending to zero at infinity (sometimes called ``ground states''). 

Still regarding ``ground state'' solutions, an important 2013 work of Frank and Lenzmann~\cite{Frank Lenzmann 1} established the uniqueness and nondegeneracy of a constrained minimizer in $H^s(\R)$ for subcritical equations of the form
\begin{equation}
 \label{eq:frac Lapl gBO}
(-\Delta)^{s}u=-u+u^p\quad\text{ in }\R.
\end{equation}  
For this, they first prove a nondegeneracy theorem for the linearized equation $(-\Delta)^{s}v+v-pu^{p-1}v=0$ at a ground state $u$. Here, the essential idea relies on a suitable substitute for Sturm-Liouville theory of ODEs to study the sign-changes of the second eigenfunction of an operator $(-\Delta)^s+V(x)$. Then, by means of an implicit function theorem, the nondegeneracy result is used to construct a unique branch of solutions depending on the parameter $s$. To conclude uniqueness of the solution, they extend the branch up to $s=1$ and use the uniqueness result for ODEs when $s=1$. 

For radial solutions to the same equation $(-\Delta)^su=-u+u^p$ in higher dimensions, Frank, Lenzmann, and Silvestre \cite{Frank Lenzmann Silvestre} extended this uniqueness and nondegeneracy result. To do so, they proved Sturmian properties  for the radial eigenfunctions of $(-\Delta)^s+V(|x|)$  with $V$ nondecreasing in $(0,+\infty)$.\footnote{Very recently,  Fall and Weth  \cite{Fall Weth 2023} improved the uniqueness result of \cite{Frank Lenzmann Silvestre} by  removing a hypothesis which required the ``ground state'' to be a constrained minimizer. They still count with the information, however, that ``ground states'' are radial, and in particular even.} Their argument uses a Hamiltonian for the fractional Laplacian that had been discovered by the first author, Sol\`a-Morales, and Sire~\cite{Cabre Sola bdry reaction, Cabre Sire A}.\footnote{In a forthcoming article~\cite{Cabre Mas Hamilt}, we will make an in depth study of the Hamiltonian for periodic solutions, also in the case of more general integro-differential operators.}

In the following result we establish the even symmetry and monotonicity of periodic constrained minimizers for a nonlocal Lagrangian whose first variation is a semilinear equation $(-\Delta)^s u=f(u)$. 
Its proof is based on a not so well-known Riesz rearrangement inequality on~$\Sph^1$ from 1976. It is an open problem to know if evenness holds also for all periodic solutions ---not only for constrained minimizers. 

\begin{theorem}
\label{thm: frac Laplacian Intro}
Let $L>0$, $0<s<1$, $G\in C^1(\R)$, $\widetilde{G}\in C^1(\R)$, $c\in\R$, and set $g:=G'$, $\widetilde{g}:=\widetilde{G}'$. Assume that $u\in L^\infty(\R)$ is $2L$-periodic and minimizes the functional
\begin{equation}\label{energy_fac_lap}
\begin{split}
E(v):= \frac{c_s}{4}\int_{-L}^L dx\int_\R dy\,\frac{|v(x)-v(y)|^2}{|x-y|^{1+2s}}-\int_{-L}^L G(v)
\end{split}
\end{equation}
among all $2L$-periodic functions $v\in L^\infty(\R)$ satisfying the constraint
\begin{equation}
\label{thm1_constr}
\int_{-L}^L \widetilde{G}(v)=c.
\end{equation}
We then have:

\begin{itemize}
\item[$(a)$] Up to a translation in the variable $x\in\R$, $u$ is an even function in $\R$ which is nonincreasing in $(0,L)$.\footnote{\label{footpm}
In other words, there exists $z\in \R$, such that $y\mapsto u(z+y)$ is an even function in $\R$ which is nonincreasing for $y\in (0,L)$. A comment here is in order. If a similar result, based on a rearrangement, were proved for ``ground state'’ solutions in $\R$, its conclusion should include also the case in which $-u(z+\cdot)$ (with the minus sign) is even and nonincreasing in $(0,L)$. However, this case is not relevant for periodic solutions. Indeed, such a periodic solution (with its minimum at the origin) will be even and nonincreasing in $(0,L)$ after a further translation that places its maximum at the origin.
}
\;
\item[$(b)$] 
Assume that $g$ and $\widetilde{g}$ belong to $C^{\delta}(\R)$  for some $\delta>0$. Then, $u\in C^{2s+\epsilon}(\R)$ for some $\epsilon>0$ and 
solves the equation $(-\Delta)^su =g(u)+\lambda \widetilde{g}(u)$ in $\R$, for some $\lambda\in\R$.

\;
\item[$(c)$] Assume that $g$ and $\widetilde{g}$ belong to $C^{1+\delta}(\R)$ for some $\delta>0$. Then, $u'$ is a continuous function in $\R$ and, up to a translation in the variable $x\in\R$, $u$ is even and $u'<0$ in $(0,L)$, unless $u$ is constant. In particular,  $u$  is an even function in $\R$ which is decreasing in $(0,L)$. Thus, $u$ has only two critical points in~$[0,L]$, its maximum at $x=0$ and its minimum at $x=L$.
\end{itemize} 

\end{theorem} 
 
\begin{Open problem}
{\rm
Do the conclusions of Theorem~\ref{thm: frac Laplacian Intro} hold also for all (smooth) periodic solutions to \eqref{eq:frac Lapl intro}---and not only for constrained minimizers?
}
\end{Open problem}

It is worth pointing out that some semilinear equations $(-\Delta)^s u= f(u)$ do admit periodic constrained minimizers
which are nonconstant. Indeed, in the forthcoming work~\cite{Cabre Mas Sola AC-BO} we will show their
existence for the equation
$(-\Delta)^s u=-u+u^p$ in $\R$, where $u>0$ and $p>1$ is subcritical, whenever the period $2L$ is large
enough.\footnote{With the notation of the theorem, here we will have $G(u)=-u^2/2$, $\widetilde{G}(u)=|u|^{p+1}/(
p+1)$, and the constant~$c$ appropriately chosen so that the Lagrange multiplier $\lambda$ is $1$.}
On the other hand, in \Cref{section:appendix:minimizers const} we show that, for $f$ regular enough, every
bounded periodic solution to $(-\Delta)^s u= f(u)$ in $\R$ which is a stable solution (here we impose no
constraint)\footnote{Stability means that the second variation of $E$ at the $2L$-periodic solution is nonnegative
definite when acting on $2L$-periodic functions.} must be identically constant. In particular, under no constraint,
there are no nonconstant bounded periodic minimizers ---not even local minimizers, meaning minimizers among
small periodic perturbations.

For more general integro-differential operators, below we will introduce some classes of kernels for which the analogue of \Cref{thm: frac Laplacian Intro} also holds; see \Cref{coro constrained minim lagrangian2}.

\subsection{The Lagrangian for nonlocal periodic problems}\label{ss:var_intro}

To find solutions to the Dirichlet problem
$(-\Delta)^s u=g(u)$ in $(-L,L)$ and $u=\varphi$ in 
$\R\setminus (-L,L)$ by variational methods, one looks for critical points of the well-known Lagrangian
\begin{equation}
\begin{split}
 \label{intro:eq:Lagran frac lap nonper}
E(v)&+\frac{c_s}{4}\int_{\R\setminus (-L,L)}\! dx \int_{(-L,L)}\! dy\,\frac{|v(x)-v(y)|^2}{|x-y|^{1+2s}}\\
&= \frac{c_s}{4}
\iint_{\left(\R\times\R\right) \setminus \left( (-L,L)^c \times (-L,L)^c \right)} dxdy\,\frac{|v(x)-v(y)|^2}{|x-y|^{1+2s}}-\int_{-L}^L G(v)
\end{split}
\end{equation}
among functions with a prescribed exterior Dirichlet datum $\varphi$. Recall that $E$ is the functional in \eqref{energy_fac_lap}. However, in the periodic setting, one cannot prescribe any exterior datum, only periodicity. It turns out, as we see from \eqref{intro:eq:Lagran frac lap nonper}, that the Lagrangian $E$ of \Cref{thm: frac Laplacian Intro} for the periodic framework differs from the Lagrangian for the Dirichlet problem.

To our knowledge, a nonlocal Lagrangian of the same type as $E$ (that is, with $(-L,L)\times \R$ as set of integration), and modeled to find periodic solutions, appeared for first time in 2015 in the arXiv version of \cite{Davila et alt Delaunay}. That work concerned the construction, by variational methods, of some nonlocal analogues of Delaunay surfaces. When \cite{Davila et alt Delaunay} became available, we communicated to its authors that the first variation of their functional was the well-known nonlocal mean curvature.  \cite{Davila et alt Delaunay} had not addressed the first variation ---that was believed to be a kind of nonlocal mean curvature, but not exactly it. At that time, we were already working on the periodic Lagrangian \eqref{energy_fac_lap} for the fractional Laplacian.\footnote{That the nonlocal mean curvature is the first variation of the periodic fractional perimeter introduced in~\cite{Davila et alt Delaunay} already appeared in the Master's Thesis from 2017 of M. Alviny\`a \cite{Alvinya}, directed within our group; see its Proposition~5.0.3. On the other hand, Section 4.2 of the Thesis already introduces the periodic semilinear Lagrangian \eqref{energy_fac_lap}, which also appeared, in independent works from ours, in DelaTorre, del Pino, Gonz\'alez, and Wei~\cite{DDGW} and in Barrios, Garc\'ia-Meli\'an, and Quaas \cite{BGQ}.} 

The variational setting for general integro-differential operators that the current paper undertakes serves as foundation for existence results of periodic solutions to integro-differential equations, and not only in dimension one. Indeed, in \cite{Cabre Csato Mas delaunay} we established for first time the existence (for any given volume constraint within a period) of nonlocal Delaunay cylinders (i.e., periodic sets with constant nonlocal mean curvature) by variational methods ---something that was left as an open problem in \cite{Davila et alt Delaunay}.
The results in \cite{Cabre Csato Mas delaunay}  are strongly related to the developments of the present paper. 

Another instance is our forthcoming work \cite{Cabre Mas Sola AC-BO}, where we strengthen the existence results of Gui, Zhang, and Du~\cite{GZD} on periodic solutions for the fractional Laplacian with Allen-Cahn type nonlinearities, and those of Barrios, Garc\'{\i}a-Meli\'an, and Quaas~\cite{BGQ} for Benjamin-Ono type nonlinearities.

\subsection{The symmetry result for general nonlocal operators}\label{ss:symmetry_intro}

We also study more general equations, of the form
\begin{equation}\label{intro semilinear eq}
\LL_K u=f(u)\quad\text{in $\R$},
\end{equation}
where $u:\R\to\R$ is periodic and $\LL_K$ is an integro-differential operator defined by
\begin{equation}\label{defi laplacian_K}
\LL_K u(x):=\lim_{\epsilon\downarrow0}\int_{\{y\in\R:\,|x-y|>\epsilon\}}dy\,
\big(u(x)-u(y)\big)K(|x-y|)
\end{equation}
and $K:(0,+\infty)\to[0,+\infty)$.

In some results, we will assume one or both of the standard conditions on the growth of $K$:
\begin{equation}\label{defi laplacian_K growth1}
K(t)\leq \frac{\Lambda}{t^{1+2s}}\quad\text{for all $t>0$}
\end{equation}
and
\begin{equation}\label{defi laplacian_K growth2}
K(t)\geq\frac{\lambda}{t^{1+2s}}\quad\text{for all $t>0$,}
\end{equation}
for some constants $0<s<1$ and $0<\lambda\leq\Lambda$. 
 
For semilinear equations involving the fractional Laplacian, the Lagrangian can be constructed using a suitable extension to one more real variable, converting the problem into a local one. This extension technique has been used in \cite{Cui Wang 2021,DG,GZD} to show the existence of periodic solutions to certain classes of semilinear equations of the type $(-\Delta)^su=f(u)$ in~$\R.$  
However, for general integro-differential operators $\LL_K$ no extension exists a priori, and the Lagrangian must be defined directly ``downstairs''. We will easily see that the appropriate Lagrangian for $2L$-periodic solutions to $\LL_K u=g(u)$ in $\R$ is, in analogy to the case of the fractional Laplacian, 
\begin{equation}\label{lagrangian rearrange mod u}
\begin{split}
\EE_{\LL_K}(u):=\frac{1}{2}[u]_K^2-\int_{-L}^L G(u),
\end{split}
\end{equation}
where $g=G'$ and
\begin{equation}
\label{eq:intro per Gagl seminorm 1}
\begin{split}
[u]_K:=\Big(\frac{1}{2}\int_{-L}^L\! dx\int_\R dy\,|u(x)-u(y)|^2 K(|x-y|)\Big)^{1/2}.
\end{split}
\end{equation}
Note that if $\LL_K=(-\Delta)^s$ then $\EE_{\LL_K}$ is the Lagrangian $E$ in \eqref{energy_fac_lap}. 

To prove the evenness and monotonicity  of $2L$-periodic constrained minimizers, we use a periodic rearrangement. For a $2L$-periodic function $u:\R\to\R$, the periodic symmetric decreasing rearrangement of~$u$ is defined as the $2L$-periodic function $u^{*\eper}:\R\to\R$ such that $u^{*\eper}\chi_{(-L,L)}$ is the Schwarz rearrangement of $|u|\chi_{(-L,L)}$ 
on $\R$. Recall that the Schwarz rearrangement of $v:(-L,L)\to \R$, denoted by $v^*$, is the nonnegative function which is equimeasurable with $|v|,$ nonincreasing in $(0,L)$, and even with respect to $0$ ---that is, $v^*$ is obtained by replacing the superlevel-sets of $|v|$ by symmetric intervals.\footnote{Most rearrangements are defined by rearranging the superlevel-sets of $|u|$ and not of $u$. One might define the rearrangement by considering the superlevel-sets of $u$, but then one needs an additional assumption on $u$, such as, for instance, being bounded from below.}

The classical P\'olya-Szeg\H o inequality for functions vanishing on the boundary of $(-L,L)$ states that
\begin{equation}\label{gy:intro:local Polya Szeg}
    \int_{-L}^L |(v^*)'|^2\leq \int_{-L}^L |v'|^2, \quad\text{ if }v(-L)=v(L)=0.
\end{equation}
This inequality remains true in the local periodic setting, that is, if we assume $v$ to be $2L$-periodic (not necessarily with $v(\pm L)=0$) and we replace $v^*$ by $v^{*\eper}$. To see this, simply replace $v$ by $v-\min_{\R} v$, consider a translate $v(\cdot+ z)$ of $v$ so that $v(\cdot+ z)-\min_{\R} v$ vanishes at $-L$ and $L$, and then apply \eqref{gy:intro:local Polya Szeg}.  

Instead, in the nonlocal case, the study of periodic versions of \eqref{gy:intro:local Polya Szeg} becomes nontrivial. Indeed, the nonlocal semi-norm $[u]_K$ feels the changes after rearrangement in all other periods besides $(-L,L)$. This will have a nontrivial impact when localizing the problem to $(-L,L)$. Indeed, in \cite{Li Wang} it was shown that the Gagliardo $W^{s,2}(-L,L)$ semi-norm, 
\begin{equation}
\label{eq:intro exact Gagl}
[u]_{W^{s,2}(-L,L)}:=\Big(\int_{-L}^L\! dx\int_{-L}^L\!  dy\,|u(x)-u(y)|^2 |x-y|^{-1-2s}\Big)^{1/2},
\end{equation}
may increase, after symmetric decreasing rearrangement, for some smooth functions $u$ with compact support in $(-L,L)$. Note, however,  that our semi-norm $[u]_K$ differs from this last semi-norm in its set of integration $(-L,L)\times\R$.

A possible analogue to the rearrangement inequality \eqref{gy:intro:local Polya Szeg} for the fractional perimeter functional on periodic sets was raised as an open problem by D\'avila, del Pino, Dipierro, and Valdinoci in~\cite{Davila et alt Delaunay}. We solved it affirmatively in~\cite{Cabre Csato Mas delaunay}. Our proof was based on a Riesz rearrangement inequality on the circle $\Sph^1,$ established independently in 1976 by Baernstein and Taylor~\cite{Baernstein Taylor} and by Friedberg and Luttinger \cite{FriedLutt}. It is stated below in Theorem~\ref{thm:gy:sharp Friedberg Luttinger}.\footnote{Theorem~\ref{thm:gy:sharp Friedberg Luttinger} will lead to a periodic rearrangement inequality for quantities integrated in $(-L,L)\times(-L,L)$, as in \eqref{eq:intro exact Gagl}, but with $ |x-y|^{-1-2s}$ replaced by some kernels $\overline K (|x-y|)$ which are $2L$-periodic in $x$ for all given~$y$.
} 
We then realized that similar ideas would work for equations involving the fractional Laplacian instead of the more involved nonlocal mean curvature operator, giving rise to the current paper.

The following \Cref{thm rearrangement} establishes the analogue to the rearrangement inequality \eqref{gy:intro:local Polya Szeg} for the nonlocal semi-norm $[\cdot]_K$ on $2L$-periodic functions and with respect to the periodic symmetric decreasing rearrangement. 
Our result applies to three different classes of kernels $K$. The first one is given by those kernels $K$ which are convex. This is an interesting class that contains the kernel of the fractional Laplacian $K(t)=c_s t^{-1-2s}$. However, this class does not allow to conclude symmetry for the nonlocal Delaunay cylinders in~\cite{Cabre Csato Mas delaunay}, where we had to deal with the nonconvex kernel 
\begin{equation}
\label{eq:intro Del kernel}
K(t)=(t^2+a^2)^{-(n+s)/2}  \qquad\text{  for some $a> 0$ and $n\geq 2$}. 
\end{equation}
Our second class\footnote{In~\Cref{section:appendix rearrangement.kernels} we will show that none of these two classes contains the other.
} ---which includes \eqref{eq:intro Del kernel} and also the kernel of the fractional Laplacian--- are those kernels $K$ for which $\tau>0\mapsto K(\tau^{1/2})$ 
is a completely monotonic function. Let us recall this notion.

\begin{definition}\label{def compl mon rearrangement}
{\rm
A function $J:(0,+\infty)\to\R$ is {\em completely monotonic} if
$J$ is infinitely differentiable and satisfies
$$(-1)^k\frac{d^k}{d\tau^k}J(\tau)\geq0
\qquad\text{for all $k\geq0$ and all $\tau>0$.}
$$
}
\end{definition}

Clearly $K(\tau^{1/2})$ is completely monotonic in $\tau>0$ when $K$ is either the kernel \eqref{eq:intro Del kernel} or the one of the fractional Laplacian, since we have, respectively, $J(\tau)=(\tau+a^2)^{-(n+s)/2}$ and $J(\tau)=c_s \tau^{-(1+2s)/2}$  in \Cref{def compl mon rearrangement}.

We can now state our main result.

\begin{theorem}\label{thm rearrangement}
Let $0<s<1$, $K:(0,+\infty)\to[0,+\infty)$ satisfy the upper bound  \eqref{defi laplacian_K growth1}, $L>0$, and $[\cdot]_K$ be defined by \eqref{eq:intro per Gagl seminorm 1}. 

\begin{itemize}
\item[$(a)$] Assume that one of the following holds:
\begin{itemize}
\item[$(i)$] $K$ is convex.
\item[$(ii)$] The function $\tau>0\mapsto K(\tau^{1/2})$ 
is completely monotonic.  
\item[$(iii)$] $K$ is nonincreasing in $(0,L)$ and vanishes in $[L,+\infty)$.
\end{itemize}
Then, 
$[u^{*\eper}]_K\leq [u]_K$ for every $2L$-periodic function 
$u:\R\to\R$ such that $u\in L^2(-L,L)$.

\item[$(b)$] Assume that  $u$ is $2L$-periodic, belongs to $L^2(-L,L)$, satisfies
$[u^{*\eper}]_K= [u]_K<+\infty$, and that one of the following holds:

\begin{itemize}
\item[$(i)'$] $K$ is convex in $(0,+\infty)$ and strictly convex\footnote{\label{foot:str-convex}By $K$ being strictly convex in an interval $I$ we mean that $K(\lambda t_1 + (1-\lambda) t_2) < \lambda K(t_1) + (1-\lambda) K(t_2)$ for all $t_1\not = t_2$ in $I$ and all $\lambda\in (0,1)$. Thus, in case $K$ were $C^2$, this is more general than the condition $K''>0$ at all points in $I$. 
}
in $(c,+\infty)$ for some $c\geq 0$.
\item[$(ii)'$] $K$ is as in $(ii)$ and $K$ is not identically zero. 
\item[$(iii)'$]$K$ is decreasing in $(0,L)$ and vanishes in $[L,+\infty)$.
\end{itemize}
Then, for some $z\in\R$, we have that either $u= u^{*\eper}(\cdot+z)$ or $u=-u^{*\eper}(\cdot+z)$.
\end{itemize} 
\end{theorem}

Case $(iii)$ will apply if the kernel $K$ is nonincreasing and has compact support and, at the same time, the period is large enough.

Regarding the class $(ii)$, a characterization of complete monotonicity is given by Bernstein's theorem; see \cite{Ber_monotonic}, 
\cite[Theorem 12b in Chapter IV]{Wid_monotonic}, and \cite[Section 14.3]{Lax}. Roughly speaking, it states that a function is completely monotonic if and only if it is the Laplace transform of a nonnegative measure. In the context of \Cref{thm rearrangement}, this characterization reads as follows: condition $(ii)$ is equivalent to the existence of a nondecreasing function $h:[0,+\infty)\to\R$ such that, if $\mu$ denotes the Borel measure defined by $\mu((a,b)):=h(b)-h(a)$ for all $0\leq a\leq b<+\infty$, the kernel $K$ is of the form
\begin{equation}\label{kernel Laplace transform}
K(t)=\int_0^{+\infty}d\mu(r)\,e^{-t^2r}\quad \text{for all $t>0$,}
\end{equation}  
and the integral in \eqref{kernel Laplace transform} converges for all $t>0$.\footnote{In all our results, and also examples in \Cref{section:appendix rearrangement.kernels}, we will only use a special case of \eqref{kernel Laplace transform}, namely, when the measure is of the form $\mu(A)=\int_A dr\,\kappa(r)$ for some $\kappa:(0,+\infty)\to[0,+\infty)$ measurable (and not identically zero). That is, $\mu$ will be absolutely continuous with respect to the Lebesgue measure.} 
In \Cref{remark:corrv15 remark 6.2} we will compute the measure~$\mu$ for the kernel $K$ in \eqref{eq:intro Del kernel}
and for the kernel of  the fractional Laplacian.

This representation of $K$ by means of the Laplace transform is the key ingredient to prove \Cref{thm rearrangement} for the class $(ii)$. 
For this, we write the Lagrangian  $[\cdot]_K$ in \eqref{eq:intro per Gagl seminorm 1} as a double integral over $(-L,L)\times (-L,L)$ with respect to the $2L$-periodic kernel
$$
\overline K(t):=\sum_{k\in\Z}K(|t+2kL|) \qquad \text{for $t\in \R$}.
$$
In order to apply the Riesz rearrangement inequality on the circle, $\overline K$ must be nonincreasing in $t\in (0,L)$. This is a nontrivial question, for instance, for the nonlocal mean curvature kernel \eqref{eq:intro Del kernel} which appeared in~\cite{Cabre Csato Mas delaunay}. In that paper, we had the idea of expressing it in terms of a Laplace transform, and this reduced the problem to verify the monotonicity of the classical periodic heat kernel. We then realized that this method works for all kernels which arise from the Laplace transform of a nonnegative measure. This led us to a new class for us: kernels such that $K(\sqrt{\cdot})$ is completely monotonic. Then, with the keywords ``completely monotonic'' and ``fractional'' at hand, we found that in independent works from ours, Claassen and Johnson \cite{Claassen Johnson} and Bruell and Pei \cite{BruellPei} had also used the Laplace transform to prove symmetry and monotonicity properties of periodic solutions associated to standing or travelling waves of some nonlocal dispersive equations.\footnote{In~\Cref{intro:fourier mult} we will explain the relation between our semilinear elliptic equations and some dispersive equations.}

Indeed, for the particular case of the fractional Laplacian, part $(a)$ of \Cref{thm rearrangement} is also proved by Claassen and Johnson \cite{Claassen Johnson}. That is, they showed that if $K(t)=t^{-1-2s}$ then $[u^{*\eper}]_K\leq [u]_K$ for 
$2L$-periodic functions $u$. Their proof follows the same lines as ours for \Cref{thm rearrangement} $(ii)$. However, they do not cover the case of equality described in part $(b)$ of \Cref{thm rearrangement}, which is the key point to ensure that all constrained minimizers must be symmetric and monotone. Also, their arguments are more involved than ours since they consider the kernel for the fractional heat equation. In addition, for the case of the fractional Laplacian, our proof of the P\'olya-Szeg\H o type inequality for  convex kernels ---i.e., our proof of \Cref{thm rearrangement} $(i)$--- applies and is much simpler.

A class of kernels which contains the cases $(i)$, $(ii)$, and $(iii)$ is the one of nonincreasing kernels ---note that, in case $(i)$, convexity and the decay estimate \eqref{defi laplacian_K growth1} lead to nonincreasing monotonicity. Thus, it is tempting to think that Theorem \ref{thm rearrangement} could hold for all nonincreasing kernels. This is however not true, as we show in \Cref{section:appendix rearrangement.kernels} with a simple example.

\Cref{thm rearrangement} leads to the evenness and monotonicity of every constrained minimizer of the Lagrangian functional $\EE_{\LL_K}$. This is the content of the following result.

\begin{corollary}\label{coro constrained minim lagrangian2}

Let $L>0$, $0<s<1$, $K$ satisfy the upper bound \eqref{defi laplacian_K growth1}, 
$G\in C^1(\R)$, $\widetilde{G}\in C^1(\R)$, $c\in\R$, and set $g:=G'$, $\widetilde{g}:=\widetilde{G}'$.
Let $\EE_{\LL_K}$ be as in \eqref{lagrangian rearrange mod u}. Assume that $K$ satisfies one of the three conditions $(i)',$ $(ii)'$, or $(iii)'$ of \Cref{thm rearrangement} $(b)$.

If $u\in L^\infty(\R)$ is $2L$-periodic and minimizes $\EE_{\LL_K}$ among $2L$-periodic functions $v\in L^\infty(\R)$ satisfying
$\int_{-L}^L \widetilde{G}(v)=c$, then, up to a translation in the variable $x\in\R$, $u$ is even and nonincreasing in $(0,L)$.

Assume, in addition, that $K$ satisfies the lower bound \eqref{defi laplacian_K growth2} and that $g$ and $\widetilde g$ belong to $C^{1+\delta}(\R)$ for some $\delta>0$. Then, $u'$ is a continuous function in $\R$ and, up to a translation in the variable $x\in\R$, $u$ is even and $u'<0$ in $(0,L)$,  unless $u$ is constant. In particular,  $u$  is an even function in $\R$ which is decreasing in $(0,L)$.
Thus, $u$ has only two critical points in~$[0,L]$, its maximum at $x=0$ and its minimum at $x=L$.
\end{corollary}

Within the proof of  the corollary, we will see that if we further assume $u\geq 0$, then the condition $u\in L^{\infty}(\R)$ is not necessary in the first conclusion of the result, i.e., $u$ is even and nonincreasing in $(0,L)$ after a translation. Thus, we may consider constrained minimizers in classes of functions which are not necessarily bounded, such as the class $L^2(-L,L)$ for instance. 

In Remark~\ref{remark: gen const} we will point out that \Cref{coro constrained minim lagrangian2} also holds for minimizers with other type of constraints. For instance, a constraint on the quasi-norm of $u$ in a Lorentz space.

To prove the last statements of \Cref{coro constrained minim lagrangian2} and of \Cref{thm: frac Laplacian Intro} on strict monotonicity, we need to develop a strong maximum principle for periodic solutions to integro-differential equations. To our knowledge, this is the first time that such a result is proven, even for the fractional Laplacian. 

\begin{theorem}
\label{theorem:max principle}
Let $0<s<1$ and $L>0$. Assume that
$K:(0,+\infty)\to[0,+\infty)$ satisfies the upper bound  \eqref{defi laplacian_K growth1} and that $K$ satisfies one of the three conditions  $(i)'$, $ (ii)'$, or $ (iii)'$ of \Cref{thm rearrangement} $(b)$. Suppose that $c\in L^{\infty}(\R)$ and that $v\in C^{2s+\gamma}(\R)$ for some $\gamma>0$. 

If $v$ is odd, $2L$-periodic, and solves
\begin{equation}
  \left\{\begin{array}{ll}\LL_K v\leq c(x)v & \quad\text{ in }(0,L),
  \\
  v\leq 0 & \quad\text{ in }(0,L),
  \end{array}
  \right.
\end{equation}
then either $v<0$ in $(0,L)$ or $v$ is identically equal to $0$. 
\end{theorem}

We have a simple direct proof of this lemma, as that of the maximum principle for the fractional Laplacian without a periodicity condition, when the kernel $K$ is convex. Instead, for the class of nonconvex kernels with $K(\tau^{1/2})$ being completely monotonic in $\tau$, the proof uses the Laplace transform as when proving \Cref{thm rearrangement}.


\subsection{Regularity}

In this paper we also study the H\"older regularity of bounded  $2L$-periodic distributional solutions to 
$\LL_K u=f(u)$ in $\R$ when $f\in C^{\beta}(\R)$ for some $\beta>0$. Our regularity results are employed to establish some parts of the results stated above. By {\em  $2L$-periodic distributional solution} we mean that $u$ is $2L$-periodic, belongs to 
$L^1(-L,L)$, and satisfies 
\begin{equation}\label{def distrib.sol integro}
\intL u\,\LL_K\psi=\intL f(u) \psi
\end{equation}
for every $2L$-periodic and smooth function $\psi$.

The boundedness of solutions $u\in L^2(-L,L)$ such that $[u]_K$ is finite is guaranteed either if ~$s>1/2$~ or under standard subcritical-type assumptions on the nonlinearity $f$; see \Cref{L infty estimate}.

Recall that, with $f\in C^{\beta}(\R)$ as above, bounded solutions to $-\Delta u=f(u)$ in $B_1\subset\R^n$ satisfy  $u\in C^{2+\beta}(B_{1/2})$ (if $\beta>0$ is not an integer), a result which is obtained by simply combining Calder\'on-Zygmund and Schauder's estimates. If $\beta=k$ is an integer and $n>1$, then $u$ is $C^{2+k-\epsilon}$ for all $\epsilon>0$, but to prove or disprove that $u$ is always of class $C^{2+k}$ still seems to be an unsettled problem.\footnote{In the case $k=0,$ and thus $f$ is merely continuous, it has been posed as an open question in \cite{Shahgholian 2015} whether all solutions have continuous second derivatives.}

In view of the local result, one would expect that bounded solutions to $\LL_Ku=f(u)$ with $f\in C^{\beta}(\R)$ for some $\beta>0$ are of class $C^{2s+\beta-\epsilon}$ for every $\epsilon>0$ (or even perhaps $C^{2s+\beta}$). 
In the following result we prove that this claim is indeed true when $s\geq(1-\beta)/2$. However, if $0<\beta<1$ and $0<s<(1-\beta)/2$, the best regularity that one can expect for bounded solutions is to be of class $C^{2s/(1-\beta)}$, as explained in the following theorem and subsequent comments. This exponent $2s/(1-\beta)$ is smaller than $2s+\beta$ when $0<s<(1-\beta)/2$, and seems to be new on the literature of fractional semilinear elliptic equations.

We emphasize that the next result remains true also in the nonperiodic setting (i.e., for the interior regularity in the Dirichlet problem) with an additional assumption on the H\"older semi-norm of $K$; see  
\Cref{gy:remark:nonperiodic reg Holder} and the observations right after it.


\begin{theorem}\label{fthm:Holder reg periodic}
Let $K$ satisfy the bounds \eqref{defi laplacian_K growth1} and \eqref{defi laplacian_K growth2} for some $0<s<1$, $f$ belong to $C^\beta(\R)$ for some $\beta>0$, and $u\in L^\infty(\R)$ be a  $2L$-periodic distributional solution to $\LL_K u=f(u)$ in $\R$. 

Then, the following holds:
\begin{itemize}
\item[$(i)$] If $\beta < 1$ and $s<\frac{1-\beta}{2}$, then $u\in C^{\frac{2s}{1-\beta}-\epsilon}(\R)$ for all $\epsilon>0$. Moreover, $\|u\|_{C^{{2s}/(1-\beta)-\epsilon}(\R)}$ is bounded by a constant which depends only on 
$s$, $\beta$, $\epsilon$, $L$, $\lambda$, $\Lambda$,  $\|f\|_{C^\beta(\R)}$, and $\|u\|_{L^\infty(\R)}$.
\item[$(ii)$] If $s\geq\frac{1-\beta}{2}$ (this holds, in particular, if $\beta\geq 1$), then $u\in C^{\beta+2s-\epsilon}(\R)$ for all $\epsilon>0$. Moreover, $\|u\|_{C^{\beta+2s-\epsilon}(\R)}$ is bounded by a constant which depends only on 
$s$, $\beta$, $\epsilon$, $L$, $\lambda$, $\Lambda$, $\|f\|_{C^\beta(\R)}$, and $\|u\|_{L^\infty(\R)}$.
\end{itemize}
\end{theorem}

In our forthcoming work \cite{Csato Mas Holder opt}, it will be shown with an example that in case $(i)$ (and also in the nonperiodic setting) one cannot go beyond $C^{2s/(1-\beta)}$ regularity. This a truly semilinear-fractional effect that does not occur for the linear equation $\LL_K u=h(x)$, and neither for local semilinear equations.


\subsection{Fourier multiplier operators and dispersive equations}
\label{intro:fourier mult}

In this work we also address the variational structure of  equations of the form 
$\LL u = f(u)$, where $\LL$ is a Fourier multiplier operator defined by
\begin{equation}\label{intro:energy dispersive symbol}
\LL u(x):=\sum_{k\in\Z}\ell({\textstyle \frac{\pi k}{L}})u_ke^{\frac{i\pi k}{L}x}
\end{equation}
and $u(x)=\sum_{k\in\Z}u_ke^{\frac{i\pi k}{L}x}$ is the Fourier series expansion of a $2L$-periodic function $u$. Here, $\ell:\R\to\R$ is called the symbol (or multiplier) of $\LL$.
As we will see, this class of operators includes all $\LL_K$ from \eqref{defi laplacian_K} with $K$ satisfying the upper bound \eqref{defi laplacian_K growth1} ---the case $\ell(\xi)=|\xi|^{2s}$ being the fractional Laplacian $(-\Delta)^{s}$.

Equations of the form $\LL u = f(u)$ naturally arise in many physical models. For instance, in the context of wave propagation, the general nonlinear equation
\begin{equation}\label{dispersive wave1}
v_t-\LL (v_x)+g(v)_x=0\qquad\text{for }x\in\R,\,t\geq0
\end{equation}
was considered in \cite{BBM,CB} to describe long-crested, long-wavelength disturbances of small
amplitude propagating in one direction in a dispersive media. Here, 
$v\equiv v(x,t)$ represents and amplitude or a velocity, $x$ is usually associated with distance to some given point in the spatial domain of propagation, $t$ is proportional to the time variable, and the operator $\LL$ acts on the spatial variable. 
When looking for periodic traveling wave solutions of the form
\begin{equation}\label{dispersive wave2}
v(x,t)\equiv u(x-ct)=\sum_{k\in\Z}u_ke^{\frac{i\pi k}{L}(x-ct)},
\end{equation}
where $2L$ is the spatial period and $c$ the velocity of propagation, the dispersive equation \eqref{dispersive wave1} rewrites in terms of $u$ as
$-cu'-\LL (u')+g(u)'=0$.
Then, commuting the derivative with $\LL$, one arrives to the  equation $\LL u = f(u)$ with $f(u)=-cu+g(u)+A$,
where $A$ is a constant of integration. 
The Benjamin-Ono equation  
$(-\Delta)^{s}u=-u+u^p$ mentioned in \Cref{ss subsection1.1} is a concrete example coming from \eqref{dispersive wave1}, as well as
the Benjamin equation $\beta(-\Delta)u-\gamma(-\Delta)^{1/2}u=-u+u^2$ with $\beta>0$ and $\gamma>0$, which is a further generalization of Benjamin-Ono where short and long range diffusion compete against each other. For the generalized Korteweg-de Vries equation 
$v_t+v_{xxx}+g(v)_x=0$, which corresponds to \eqref{dispersive wave1} taking $\LL=-\partial_{xx}$, one gets $-u''=-cu+g(u)+A$. An example coming from a model not covered by \eqref{dispersive wave1} is the one for the Benjamin-Bona-Mahony equation $v_t+v_x+vv_x-v_{xxt}=0$, which leads to $-cu''=(c-1)u-u^2+A$.

The integro-differential equation 
$\LL_K u =f(u)$ considered in this work also falls within the above Fourier framework $\LL u = f(u)$. Indeed, the operators $\LL_K$ defined in \eqref{defi laplacian_K}, with $K$ satisfying the upper bound \eqref{defi laplacian_K growth1}, are multiplier operators with a symbol $\ell=\ell_K$ which can be explicitly written in terms of $K$; see \Cref{eigenfunctions nonlocal}. This follows from the simple fact that the standard Fourier basis are eigenfunctions of $\LL_K$. In turn, the equation 
$\LL_K u =f(u)$ covers some of the examples from wave propagation presented before. In particular, for those examples, \Cref{coro constrained minim lagrangian2} leads to symmetry properties for traveling waves found by a constrained minimization of the associated Lagrangian $\EE_{\LL_K}$ given in \eqref{lagrangian rearrange mod u}. 
This applies, for instance, to the nontrivial constrained minimizers for large enough periods $2L$ found by Chen and Bona~\cite{CB} when the symbol comes from a kernel as the ones described in \Cref{thm rearrangement}~$(b)$. Moreover, our \Cref{coro constrained minim lagrangian2} is, in a sense, the counterpart of the symmetry results found by Bruell and Pei~\cite{BruellPei} for dispersive equations with weak dispersion to the case of strong dispersion, once we restrict ourselves to constrained minimizers. Recall here that in 
\cite{BruellPei}, apart from complete monotonicity and/or homogeneity, the symbol is assumed to have a power-like decay at infinity ---this last property is the one which refers to the words ``weak dispersion''. Hence, their results do not cover integro-differential operators like the fractional Laplacian, for which the symbol has instead power growth at infinity.

The proofs of symmetry and monotonicity presented in this paper, which are based on periodic rearrangement, only work for operators of the form $\LL_K$ and some kernels $K$, but not for general multiplier operators $\LL$. However, the description of the variational structure of $\LL_K u=f(u)$ 
---briefly commented in \Cref{ss:symmetry_intro} and fully developed in \Cref{ss energy semilinear frac lap}--- easily extends to 
$\LL u=f(u)$. This is useful since there are models of interest in the literature (such as the Benjamin equation) which cannot be described solely by using an integro-differential operator.  
That is why we will show in \Cref{ss Energy dispersive} that the Lagrangian expression \eqref{energy defi half oper} generalizes to multiplier operators $\LL$ as in \eqref{intro:energy dispersive symbol}. More precisely, we will show that 
$\LL u=g(u)$ is the Euler-Lagrange equation of the functional
\begin{equation}\label{energy bona}
   \EE_{\LL}(u):=\int_{-L}^L \Big(\frac{1}{2}\,u\,\LL u-G(u)\Big);
\end{equation}
see \eqref{s2 sobol norm statement LK}, \eqref{s2 sobol norm statement}, and \Cref{critical point energy dispersive} for more details. The reader may also look at \Cref{ss Energy dispersive} for the precise assumptions required on the symbol $\ell$ of $\LL$ in order to construct the Lagrangian \eqref{energy bona}.

To finish this subsection, let us mention that the Fourier series approach also leads to a nice relation between the integro-differential expressions for the half-Laplacian on $\R$ and on $\Sph^1$. We give the details of this in~\Cref{section:appendix half.harm}.


\subsection*{Plan of the paper}
\begin{itemize}

\item In \Cref{ss energy semilinear frac lap} we describe the variational formulation for linear and semilinear equations. We define here the notions of periodic weak and distributional solutions.

\item In \Cref{section:Fourier} we generalize the variational formulation of \Cref{ss energy semilinear frac lap} to multiplier operators for Fourier series, introduced above in \Cref{intro:fourier mult}.

\item In \Cref{ss:symmetry} we show a nonlocal P\'olya-Szeg\H o inequality for the periodic symmetric decreasing rearrangement, \Cref{thm rearrangement}.

\item In \Cref{section:max principle} we establish a strong maximum principle for periodic solutions of integro-differential equations,  \Cref{theorem:max principle}.

\item In \Cref{section:proof of thm 1.1 and corollary constr. min}
we provide the proofs of \Cref{thm: frac Laplacian Intro} and \Cref{coro constrained minim lagrangian2}.

\item In \Cref{s:Linfty_estimates} we prove $L^\infty(\R)$ estimates for weak periodic solutions.

\item In \Cref{ss_Calpha_reg} we establish the $C^\alpha(\R)$ estimates of \Cref{fthm:Holder reg periodic} for bounded periodic distributional solutions.

\item In \Cref{section:appendix half.harm} we investigate the connection between the half-Laplacian acting on periodic functions in $\R$ and on functions defined in $\Sph^1$.

\item In \Cref{section:appendix rearrangement.kernels} we discuss on the classes of kernels considered in \Cref{thm rearrangement}.

\item In \Cref{section:appendix:minimizers const} we establish that every stable periodic solution is identically constant.

\end{itemize}



\section{Variational formulation for nonlocal periodic problems}\label{ss energy semilinear frac lap}

In this section we address the variational approach for periodic solutions to the linear equation $\LL_K u=h$ in $\R$ and also to the semilinear equation $\LL_K u=f(u)$ in $\R$. In the linear case we assume the function $h:\R\to\R$ to be periodic.

As mentioned in 
 \Cref{ss:symmetry_intro} of the Introduction, the well-known Lagrangian  \eqref{intro:eq:Lagran frac lap nonper} for the Dirichlet problem  for the fractional Laplacian differs from the Lagrangian \eqref{energy_fac_lap} for the periodic problem. The same situation occurs for the general integro-differential operator $\LL_Ku$ defined in \eqref{defi laplacian_K}.
In this section we will show that $\LL_K u=g(u)$ in $\R$ is the Euler-Lagrange equation for critical points of the functional (recall that $g=G'$)
\begin{equation}\label{energy dirichlet periodic}
\EE_{\LL_K}(u)
  =\frac{1}{4}\int_{-L}^{L}\!dx\int_{\R}dy\, {|u(x)-u(y)|^2}K(|x-y|)
- \int_{-L}^{L} G(u)
\end{equation}
acting on $2L$-periodic functions $u$. To show this, we will express the functional explicitly in terms of $\LL_K$. By a basic integration by parts formula for $\LL_K$ acting on periodic functions, we will indeed see that \eqref{energy dirichlet periodic} rewrites as 
\begin{equation}\label{energy defi half oper}
\EE_{\LL_K}(u)=\int_{-L}^L \Big(\frac{1}{2}\,{u}\,\LL_K u-G(u)\Big).
\end{equation}

In view of \eqref{energy dirichlet periodic}  it is now natural to define an associated semi-norm and inner product as follows.
Given a $2L$-periodic function $u:\R\to\R$,  we set
\begin{equation}\label{intro:def seminorm K}
\begin{split}
[u]_K:=\Big(\frac{1}{2}\int_{-L}^L\! dx\int_\R dy\,|u(x)-u(y)|^2 K(|x-y|)\Big)^{1/2}.
\end{split}
\end{equation}
In addition, given another $2L$-periodic function $v$, assuming that $[u]_K<+\infty$ and $[v]_K<+\infty$, we define the symmetric bilinear form
\begin{equation}
\begin{split}
\langle u,v\rangle_K:=\frac{1}{2}\intL dx\!\int_\R dy\,
{\big(u(x)-u(y)\big)\big(v(x)-v(y)\big)}K(|x-y|).
\end{split}
\end{equation}

This motivates the following weak formulation for periodic solutions to the  equation $\LL_{K} u=h$ in $\R$. Given a $2L$-periodic function $h$, with
$h\in L^2(-L,L)$, we say that a $2L$-periodic function 
$u:\R\to\R$ is a {\em  periodic weak solution to $\LL_K u=h$ in 
$\R$} if $u\in L^2(-L,L)$, $[u]_K<+\infty$, and
\begin{equation}\label{truly weak form eq}
  \langle u,v\rangle_K=\intL\!h v
\end{equation}
for every $2L$-periodic function $v$ with $v\in L^2(-L,L)$ and $[v]_K<+\infty$. In \Cref{fl l1} we will show that, if $u$ is smooth enough and satisfies \eqref{truly weak form eq} for all such functions $v$, then
\begin{equation}
\label{eq:after def of period weak sol}\intL (\LL_K u-h)\psi
  =\langle u,\psi\rangle_K-\intL\!h \psi=0
\end{equation}
for every $2L$-periodic and smooth function $\psi$. From this, by density and periodicity we conclude that $\LL_K u(x)=h(x)$ for almost every $x\in\R$.

In addition, in \Cref{prop:gy:weak per and nonper sol} we will state a consistency-type result establishing that periodic weak solutions are also weak solutions in the Dirichlet sense in every interval under their own exterior data. This result will be needed later in the paper.

A final comment on another notion of periodic solution to 
$\LL_Ku= h$ in $\R$ is in order. Based on our integration by parts formula  (\Cref{fl l1}) we may also introduce the concept of {\em periodic distributional solution} simply by asking 
that $u\in L^1(-L,L)$ is $2L$-periodic and that 
\begin{equation}\label{def distrib.sol integro}
\intL u\,\LL_K\psi=\intL h \psi
\end{equation}
for every $2L$-periodic and smooth function $\psi$. This notion of solution is convenient when considering more general right-hand sides $h$, such as distributions compactly supported in $(-L,L)$, with the understanding that $\int\!h \psi$ means the action of the distribution $h$ on the test function $\psi$.

Be aware that in 
\cite{DRSV,RS,S1}, whose regularity results will be used in \Cref{ss_Calpha_reg}, the authors call ``weak solution'' to what we refer as ``distributional solution''. 
The two notions coincide in the usual Dirichlet case if $u$ belongs to $L^2(\R)$ and has finite $W^{2,s}$ semi-norm. The same holds true in the periodic setting. More precisely, as a consequence of \Cref{fl l1}, we will see that every periodic distributional solution $u$ satisfying 
$u\in L^2(-L,L)$ and $[u]_K<+\infty$, is also a periodic weak solution (and vice-versa).

The following lemma is the basic integration by parts formula associated to 
$\LL_K$ when acting on periodic functions. It will justify the appropriate notion of periodic solution in the weak sense, and it will give the suitable kinetic term of the Lagrangian functional associated to $\LL_K$. 

\begin{lemma}\label{fl l1}
Let $K$ satisfy the upper bound \eqref{defi laplacian_K growth1}. If $u$ and 
$\psi$ are $2L$-periodic functions, $u\in L^2(-L,L),$ $[u]_K<+\infty,$ and 
$\psi\in C^{\alpha}(\R)$ for some $\alpha>2s$,\footnote{Here, $\alpha>1$ is allowed as well.}
 then
\begin{equation}\label{fl eq0}
\begin{split}
\intL\! u\,\LL_{K}\psi
=\frac{1}{2}\intL dx\!\int_\R dy\,
{\big(u(x)-u(y)\big)\big(\psi(x)-\psi(y)\big)}K(|x-y|).
\end{split}
\end{equation}
If in addition 
$u\in C^{\alpha}(\R)$ for some $\alpha>2s$, 
then
\begin{equation}\label{fl eq0_full integration}
\intL\! u\,\LL_{K}\psi=\intL\! (\LL_{K}u)\psi\,.
\end{equation}
\end{lemma}

\begin{remark}\label{rmk:pvLK:bilinear.form}{\em 
Before proving the lemma, a few remarks on \eqref{fl eq0} are in order. Recall that by definition 
$$
  \LL_K\psi(x)=\lim_{\epsilon\downarrow 0}\int_{\R}dy\,(\psi(x)-\psi(y))K_{\epsilon}(|x-y|)
  =:\lim_{\epsilon\downarrow 0}\LL_{K_{\epsilon}}\psi(x),\quad\text{where }K_{\epsilon}:=K\chi_{(\epsilon,+\infty)}.
$$
If 
$\psi\in C^{\alpha}(\R)$ 
($\psi$ not being necessarily periodic) and $\alpha>2s$, then 
$\LL_{K_{\epsilon}}\psi(x)$ is uniformly bounded for all 
$\epsilon>0$ and all $x\in\R$, and the pointwise limit as $\epsilon\downarrow0$ is well defined. This can be checked as follows. 
First, note that we can assume without loss of generality $2s<\alpha<2$, since $s$ is smaller than $1$. 
Using the upper bound \eqref{defi laplacian_K growth1} we see that 
(the second line is needed only if $2s\geq 1$ and hence $1<\alpha< 2$)
\begin{equation}\label{limit inside integral fractional laplacian prev}
\begin{split}
|\LL_{K_\epsilon}\psi(x)|&=\Big|\int_{\R}dz\,\big(\psi(x)-\psi(x-z)\big)
K_\epsilon(|z|)\Big|\\
&=\frac{1}{2}\Big|\int_{\R}dz\,\big(2\psi(x)-\psi(x-z)-\psi(x+z)\big)
K_\epsilon(|z|)\Big|\\
&\leq 4\Lambda\|\psi\|_{L^\infty(\R)}\int_1^{+\infty}\frac{dt}{t^{1+2s}}+
2\Lambda \|\psi\|_{C^{\alpha}(\R)}\int_{0}^1dt\,
\frac{t^{\alpha}}{t^{1+2s}}
\leq C
\end{split}
\end{equation}
for all $\epsilon>0$ and all $x\in\R$,
where $C$ is a constant depending only on 
$\Lambda, s,{\alpha},$ and $\|\psi\|_{C^{\alpha}(\R)}$.
Similar arguments also show that $|\LL_{K_\epsilon}\psi(x)-\LL_{K_\delta}\psi(x)|$ tends to zero as $\epsilon$ and 
$\delta$ tend to $0$. This proves the well-definiteness of the left-hand side of \eqref{fl eq0}, since we are assuming that $u\in L^2(-L,L)\subset L^1(-L,L)$.

On the other hand, observe that no principal value is required on the right-hand side of \eqref{fl eq0}. To see this, simply write 
${K}={K}^{1/2}{K}^{1/2}$ and use the Cauchy-Schwarz inequality to get
\begin{equation} \label{Holder for fl eq0_aux} 
  \frac{1}{2}\int_{-L}^Ldx\!\int_{\R}dy\,|u(x)-u(y)||\psi(x)-\psi(y)|K(|x-y|)
  \leq [u]_K [\psi]_K\,.
\end{equation}
Hence, 
$|u(x)-u(y)||\psi(x)-\psi(y)|{K}(|x-y|)$ is integrable in $(-L,L)\times\R$.
Here we have also used that $[\psi]_K<+\infty$ if $\alpha>s$.
}\end{remark} 

\begin{proof}[Proof of \Cref{fl l1}]
 Set
 $K_\epsilon=K\chi_{(\epsilon,+\infty)}$ for $\epsilon>0$, as before. We claim that it suffices to prove the lemma for $K_{\epsilon}$ instead of $K.$ Indeed, by definition 
 $\LL_K\psi(x)=\lim_{\epsilon\downarrow 0}\LL_{K_{\epsilon}}\psi(x)$ and, therefore, from \eqref{limit inside integral fractional laplacian prev} and dominated convergence it follows that $\lim_{\epsilon\downarrow 0}\int_{-L}^L u\,\LL_{K_{\epsilon}}\psi=
\int_{-L}^L u\,\LL_K\psi.$ On the other hand, using the fact that $0\leq K_{\epsilon}\leq K$ for all $\epsilon>0$ and 
\eqref{Holder for fl eq0_aux}, by dominated convergence we get that
\begin{equation}
\begin{split}
  \lim_{\epsilon\downarrow 0}\intL dx\!\int_\R dy\,&
{\big(u(x)-u(y)\big)\big(\psi(x)-\psi(y)\big)}K_{\epsilon}(|x-y|)
\\
  =& \intL dx\!\int_\R dy\,
{\big(u(x)-u(y)\big)\big(\psi(x)-\psi(y)\big)}K(|x-y|).
\end{split} 
\end{equation}
This shows the claim. Thus, we will work with $K_{\epsilon}$ from now on.

In the next step we check the integrability properties of certain integrals. This is required to justify the calculations proving the lemma.
Writing ${K_\epsilon}={K_\epsilon}^{1/2}{K_\epsilon}^{1/2}$ and using the Cauchy-Schwarz inequality and the upper bound \eqref{defi laplacian_K growth1} on the growth of  $K$, we see that
\begin{equation}
\begin{split}
\intL dx \int_{\R} dy\,
&|\psi(x)-\psi(y)|{K_\epsilon}(|x-y|)\,|u(x)|\\
&\leq\sqrt{2}
[\psi]_K\Big(\intL dx\, |u(x)|^2
\int_{\R\setminus (x-\epsilon,x+\epsilon)}dy\,K(|x-y|)\Big)^{1/2}
\leq \frac{C}{\epsilon^{s}}\,[\psi]_K\|u\|_{L^2(-L,L)},
\end{split}
\end{equation}
where $C$ only depends on $\Lambda$ and $s$.  Therefore, 
$({\psi(x)-\psi(y))}K_\epsilon(|x-y|)u(x)$ is absolutely integrable in $(-L,L)\times\R$. This also proves that $(\psi(x)-\psi(y))K_{\epsilon}(|x-y|)u(y)$ is absolutely integrable in 
$(-L,L)\times\R$, in view of \eqref{Holder for fl eq0_aux}  and the fact that
$|u(y)|\leq|u(y)-u(x)|+|u(x)|.$
These integrability properties  justify the forthcoming computations.

We now address the proof of \eqref{fl eq0} for $K_\epsilon$. Interchanging the name of the variables $x$ and $y$, 
we express the first integral as follows:
\begin{equation}\label{fl eq1}
\begin{split}
\intL u\,\LL_{K_\epsilon} \psi
&=\intL dx \int_\R dy\,\big({\psi(x)-\psi(y)\big)}{K_\epsilon}(|x-y|)\,u(x)\\
&=\frac{1}{2}\intL dx\int_\R dy\,\big({\psi(x)-\psi(y)\big)}{K_\epsilon}(|x-y|)\,u(x)\\
&\quad+\frac{1}{2}\intL dy \int_\R dx\,\big({\psi(y)-\psi(x)\big)}{K_\epsilon}(|x-y|)\,u(y).
\end{split}
\end{equation}
Since $u$ and $\psi$ are $2L$-periodic, the change of variables 
$\overline x=x-2kL$, $\overline y=y-2kL$ yields
\begin{equation}\label{fl eq2}
\begin{split}
\intL dy \int_\R dx\,\big(&\psi(y)-\psi(x)\big){K_\epsilon}(|x-y|)\,u(y)\\
&=\sum_{k\in\Z}\intL dy\int_{(2k-1)L}^{(2k+1)L}dx\,
\big({\psi(y)-\psi(x)}\big){K_\epsilon}(|x-y|)\,u(y)\\
&=\sum_{k\in\Z}\int_{-(2k+1)L}^{-(2k-1)L}d\overline y\intL d\overline x\,
\big({\psi(\overline y)-\psi(\overline x)}\big){K_\epsilon}(|\overline x-\overline y|)\,u(\overline y)
\\
&=-\intL d\overline x \int_\R d\overline y\,\big({\psi(\overline x)-\psi(\overline y)}\big){K_\epsilon}(|\overline x-\overline y|)
\,u(\overline y).
\end{split}
\end{equation}
From \eqref{fl eq1} and \eqref{fl eq2}, we conclude that
\begin{equation}
\intL u\,\LL_{K_\epsilon} \psi
=\frac{1}{2}\intL dx\int_\R dy\,
{\big(\psi(x)-\psi(y)\big)\big(u(x)-u(y)\big)}{K_\epsilon}(|x-y|).
\end{equation}
This shows \eqref{fl eq0} for $K_{\epsilon}$ in place of $K$, and finishes the proof of the lemma.
\end{proof}

Given $2L$-periodic functions $u$ and $v$ such that $[u]_K<+\infty$ and $[v]_K<+\infty$, we have considered the symmetric bilinear form
\begin{equation}\label{fl eq0_bis}
\begin{split}
\langle u,v\rangle_K:=\frac{1}{2}\intL dx\!\int_\R dy\,
{\big(u(x)-u(y)\big)\big(v(x)-v(y)\big)}K(|x-y|),
\end{split}
\end{equation}
which is well defined by \eqref{Holder for fl eq0_aux}. In particular, $[u]_K^2=\langle u,u\rangle_K$. Moreover, if in addition  
$u\in C^2(\R)$ and $v\in L^2(-L,L)$, \Cref{fl l1} yields
$$\langle u,v\rangle_K=\intL\! (\LL_{K}u)\,v.$$

Let us now turn our attention to the Lagrangian associated to the linear and semilinear equations. To compare it with the nonperiodic scenario, 
we begin by considering  the Dirichlet problem associated to the linear equation in a bounded open interval $I\subseteq\R$, namely,
\begin{equation}\label{semilinear dirichlet}
\left\{
\begin{aligned}
\LL_K u&=h &&\text{ in }I, \\
u&=g &&\text{ in }\R\setminus I,
\end{aligned}\right.
\end{equation}
where $h$ is the given right-hand side and $g$ corresponds to the exterior data. 
It is well known that  \eqref{semilinear dirichlet} is the Euler-Lagrange equation for critical points of the Lagrangian
\begin{equation}\label{energy dirichlet}
\begin{split}
\EE_D(u)&:
=\frac{1}{4}\Big\{\int_I dx\int_Idy\, |u(x)-u(y)|^2 K(|x-y|)\\
&\quad\qquad+2\int_I dx\int_{I^c}dy\, |u(x)-u(y)|^2 K(|x-y|)\Big\}
- \int_I hu
\end{split}
\end{equation}
among functions satisfying $u=g$ in $\R\setminus I$ ---see also \eqref{intro:eq:Lagran frac lap nonper} to write this as a single double integral in $(\R\times\R)\setminus (I^c\times I^c)$.

In contrast, consider now the periodic scenario. Let $L>0$ and $I=(-L,L)\subset\R$. We want to study the variational framework of $2L$-periodic solutions to  $\LL_K u=h$ in $\R$, where $h:\R\to\R$ is a $2L$-periodic function.
Here, no boundary data can be imposed (and hence the class of competitors above for the Dirichlet problem makes no sense here), but the solution is asked to be periodic and the equation to hold in the whole real line.

In the following \Cref{critical points semilinear periodic} we show that, as announced, $\LL_K u=h$ is the Euler-Lagrange equation among $2L$-periodic functions of a different functional than the one in the Dirichlet case. It is given by
\begin{equation}\label{def energy linear periodic case}
\begin{split}
\EE_{\LL_K}(u)&=\frac{1}{4}\int_I dx\int_\R dy\,
|u(x)-u(y)|^2K(|x-y|)- \int_I h u\\
&=\frac{1}{2}\langle u, u\rangle_K-\int_I h u=\frac{1}{2}[u]_K^2-\int_I h u.
\end{split}
\end{equation}
Using the periodicity of $u$, sometimes it will be useful to write 
$\EE_{\LL_K}(u)=\frac{1}{2}[u]_K^2-\int_I h u$ with
\begin{equation}
\label{gy:expr of K by periodicity}
\begin{split}
[u]_K=\Big(\frac{1}{2}\int_I dx\int_Idy\,
|u(x)-u(y)|^2\sum_{k\in\Z}K(|x-y+2kL|)\Big)^{1/2}.
\end{split}
\end{equation}
This will be useful, for instance, in the proof of \Cref{thm rearrangement} to get our results on the periodic symmetric decreasing rearrangement.

\begin{lemma}\label{critical points semilinear periodic}
Let $K$ satisfy the upper bound \eqref{defi laplacian_K growth1}. Assume that $u$ and $h$ are $2L$-periodic functions such that 
$u\in L^2(-L,L)$, $[u]_{K}<+\infty$, and $h\in L^1(-L,L)$. 
Then,
\begin{equation}\label{eq critical point linear weak}
\frac{d}{dt}\biggr\rvert_{t=0}\EE_{\LL_K}(u+t\psi)
=\langle u, \psi\rangle_K-\intL h \psi
=\intL \big(u\,\LL_K\psi -h\psi\big)
\end{equation}
for every $2L$-periodic smooth function $\psi$. In particular, if $h\in L^2(-L,L)$ and $u$ is a critical point of $\EE_{\LL_K}$ among $2L$-periodic smooth perturbations, then 
$\LL_K u=h$ in $\R$ in the periodic weak sense.
\end{lemma}

\begin{proof}
\Cref{fl l1} directly gives \eqref{eq critical point linear weak}.
Now, if $u$ is a critical point of $\EE_{\LL_K}$, from \eqref{eq critical point linear weak} we deduce that 
$$
\langle u, \psi\rangle_K=\intL h \psi
$$ 
for every $2L$-periodic smooth $\psi$. By density, this equality extends to every $2L$-periodic function $\psi$ such that 
$\psi\in L^2(-L,L)$ and 
$[\psi]_K<+\infty$. Thus, $u$ is a periodic weak solution to $\LL_K u=h$ in $\R$.
\end{proof}

Note that, once the variational framework for the linear periodic case is developed, we easily get that the semilinear equation $\LL_K u=g(u)$ in $\R$ is the Euler-Lagrange equation of the functional $\EE_{\LL_K}$ given in \eqref{energy dirichlet periodic}. That is, for the semilinear case we simply replace $\int_I h u$ by 
$\int_I  G(u)$ on the right-hand side of \eqref{def energy linear periodic case}, where $G'=g$.

To finish this section, we establish a consistency-type result that will be needed within the proof of \Cref{fthm:Holder reg periodic}. Recall that \Cref{fthm:Holder reg periodic} establishes the Hölder regularity 
 of bounded periodic distributional solutions to the equation $\LL_K u =f(u)$. The next proposition shows that periodic distributional solutions, in the linear case, are also distributional solutions in the usual Dirichlet sense.  It follows easily from the proposition  that periodic weak solutions are also weak solutions in the usual Dirichlet sense, provided one uses the standard (non-periodic) integration by parts formula and the corresponding (non-periodic) semi-norm to define weak solutions in the Dirichlet sense.

\begin{proposition}
\label{prop:gy:weak per and nonper sol}
Let $K$ satisfy the upper bound \eqref{defi laplacian_K growth1}. Assume that $u$ and $h$ belong to $L^1(-L,L)$, $u$ is $2L$-periodic, and 
\begin{equation}
\int_{-L}^L u\,\LL_K \psi=\int_{-L}^L h\psi
\end{equation}
for every $2L$-periodic smooth function $\psi$.

Then,
\begin{equation}\label{weak dirichlet sol defi}
\int_\R u\,\LL_K \varphi=\int_\R h\varphi
\end{equation}
for every $\varphi\in C^\infty_c(-L,L)$. In particular, $u$ is a distributional solution of $\LL_K u=h$ in the usual Dirichlet sense.
\end{proposition} 

\begin{proof} Let $I=(-L,L)$,  
$\varphi\in C^\infty_c(I)$, and let $\varphi_e\in C^{\infty}(\R)$ denote its $2L$-periodic extension to~$\R$. We claim that
\begin{equation}\label{eq:gyula:period non period Keps}
\int_\R u\,\LL_{K} \varphi
=\int_I u\,\LL_{K} \varphi_e\,.
\end{equation}
Recall that $u\in L^1(I)$ is $2L$-periodic.
To prove the claim let us write
\begin{equation}\label{regularity eq prev1}
\begin{split}
\int_\R u\,\LL_{K} \varphi
&=\int_\R dx\int_\R dy\, u(x)\big(\varphi(x)-\varphi(y)\big){K}(|x-y|)\\
&=\int_I dx\int_I dy\, u(x)\big(\varphi_e(x)-\varphi_e(y)\big){K}(|x-y|)\\
&\quad+\int_I dx\int_{\R\setminus I} dy\,u(x)\varphi_e(x){K}(|x-y|)
-\int_{\R\setminus I} dx\int_I dy\,u(x)\varphi_e(y){K}(|x-y|).
\end{split}
\end{equation}
Observe that the first two integrals on the last expression in  \eqref{regularity eq prev1} are absolutely convergent since $u\in L^1(I)$, $K$ satisfies \eqref{defi laplacian_K growth1}, $\varphi_e\in C^2(I)$, and $\varphi_e|_{I}$ has compact support in $I$. For the third integral, using that $\varphi$ has support in $(-L+\delta,L-\delta)=:I_{\delta}$ for some $\delta>0$, and also that $u$ is $2L$-periodic, we see that
\begin{equation}
\begin{split}
\int_{\R\setminus I} dx\int_I dy\, |u(x)||\varphi_e(y)|K(|x-y|)
&\leq\|\varphi\|_{L^\infty(\R)}\sum_{k\in\Z\setminus\{0\}}\int_{I+2kL} dx
\int_{I_{\delta}} dy\,|u(x)|K(|x-y|)\\
&=
\|\varphi\|_{L^\infty(\R)}\sum_{k\in\Z\setminus\{0\}}\int_{I} d\bar{x}
\int_{I_{\delta}-2kL} d\bar{y}\,|u(\bar{x})|K(|\bar{x}-\bar{y}|)\\
&\leq \Lambda\|\varphi\|_{L^\infty(\R)}\|u\|_{L^1(I)}
\int_{\R\setminus(-\delta,\delta)} dt\, |t|^{-1-2s}<+\infty.
\end{split}
\end{equation}
This means that $u(x)\varphi_e(y){K}(|x-y|)$ is absolutely integrable in $(\R\setminus I)\times I$ and therefore, arguing as in \eqref{fl eq2} and replacing there $k\in\Z$ by $k\in\Z\setminus\{0\}$, we see that
\begin{equation}
\begin{split}
\int_{\R\setminus I} dx\int_I dy\, u(x)\varphi_e(y)K(|x-y|)
=\int_I dx\int_{\R\setminus I} dy\, u(x)\varphi_e(y)K(|x-y|).
\end{split}
\end{equation}
Thus, \eqref{regularity eq prev1} leads to
\begin{equation}\label{final link periodic dirichlet}
\begin{split}
\int_\R u\,\LL_{K} \varphi
&=\int_I dx\int_I dy\, u(x)\big(\varphi_e(x)-\varphi_e(y)\big){K}(|x-y|)\\
&\quad+\int_I dx\int_{\R\setminus I} dy\,
u(x)\big(\varphi_e(x)-\varphi_e(y)\big){K}(|x-y|)
=\int_I u\,\LL_{K} \varphi_e.
\end{split}
\end{equation}
This proves the claim \eqref{eq:gyula:period non period Keps}.

From \eqref{eq:gyula:period non period Keps}, using that $u$ is a distributional $2L$-periodic solution to 
$\LL_K u=h$ in $\R$, we deduce that
\begin{equation}
\begin{split}
\int_\R u\,\LL_K \varphi=\int_I u\,\LL_K \varphi_e=\int_I h\varphi_e=\int_\R h\varphi,
\end{split}
\end{equation}
since $\varphi$ vanishes outside $I$.
This establishes \eqref{weak dirichlet sol defi}.
\end{proof}

\section{Fourier multiplier operators}\label{ss Energy dispersive}
\label{section:Fourier}

In this section we describe the variational structure of periodic solutions to semilinear equations of the form
$\LL_K u=f(u)$ from the Fourier side point of view,  and also for more general 
multiplier operators. This has applications to travelling waves of some dispersive equations as we have explained in \Cref{intro:fourier mult}.

The first key observation is that the standard Fourier basis of unitary complex exponentials  (and, thus, also the sines and cosines) are eigenfunctions of 
$\LL_K$; see \Cref{eigenfunctions nonlocal} below. 
All what is needed for this is the kernel to be even.
Therefore, our integro-differential operators have a simple representation in the Fourier side by means of multipliers. 
More precisely, if we set 
\begin{equation}\label{definition lK(k)}
\ell_K(\xi):=\frac{1}{|\xi|}\int_{\R}dz\,
\big({1-\cos(z)}\big)K\Big(\frac{|z|}{|\xi|}\Big)
\quad\text{for $\xi\in\R\setminus\{0\}$,\quad and }\ell_K(0):=0,
\end{equation}
then, for every $2L$-periodic function $u:\R\to\C$ written as
\begin{equation}\label{eq: complex u_k}
u(x)=\sum_{k\in\Z}u_ke^{\frac{i\pi k}{L}x}
\quad\text{with}\quad
u_k:=\frac{1}{2L}\intL dx\, u(x)e^{-\frac{i\pi k}{L}\,x},
\end{equation}
we will see in \Cref{eigenfunctions nonlocal} that, formally,
\begin{equation}\label{energy dispersive symbol K}
\LL_K u(x)=\sum_{k\in\Z}\ell_K({\textstyle \frac{\pi k}{L}})u_ke^{\frac{i\pi k}{L}x}.
\end{equation}
This leads us to study the variational structure of periodic solutions to semilinear equations $\LL u=f(u)$ also for general multiplier operators of the form
\begin{equation}\label{energy dispersive symbol}
\LL u(x)=\sum_{k\in\Z}\ell({\textstyle \frac{\pi k}{L}})u_ke^{\frac{i\pi k}{L}x},
\end{equation}
where the given function $\ell:\R\to\R$ is called the symbol of $\LL$.
Clearly, $\LL=\LL_K$ for $\ell=\ell_K$. 

In the case of integro-differential operators, \eqref{definition lK(k)} shows that $\LL_K$ has a positive definite 
symbol~$\ell_K$. In fact, thanks to the bounds \eqref{defi laplacian_K growth1} and \eqref{defi laplacian_K growth2} on $K$, in \Cref{eigenfunctions nonlocal} we will check that 
\begin{equation}\label{estimates l(k) lambda}
\frac{\lambda}{c_s}\,|\xi|^{2s}
\leq\ell_K(\xi)\leq
\frac{\Lambda}{c_s}\,|\xi|^{2s}
\end{equation}
for all $\xi\in\R$, where $c_s$ is given by \eqref{defi laplacian_}.
However, as we will see below, in the Fourier series approach to periodic solutions to $\LL u=f(u)$ 
we can also allow symbols which change sign and which are bounded by a power-like behavior at infinity not necessarily of order less than 2. More precisely, we consider symbols 
$\ell:\R\to\R$ such that
\begin{equation}\label{energy dispersive symbol even}
\ell(\xi)=\ell(-\xi)\quad\text{for all $\xi\in\R$}
\end{equation}
and
\begin{equation}\label{energy dispersive symbol growth}
\limsup_{\xi\uparrow+\infty}|\xi|^{-p}|\ell(\xi)|<+\infty\quad
\text{for some $p\geq0$}.
\end{equation}
Note that the later is a weaker assumption than the second bound in \eqref{estimates l(k) lambda} if $p>2s$.
 This class of symbols contains, for example, the case
$\ell(\xi)=\beta \xi^2-\gamma |\xi|$ associated to the Benjamin equation ---simply take $p=2$ in \eqref{energy dispersive symbol growth}.

Let us start by proving that the Fourier basis is formed by eigenvectors of the integro-differential operator $\LL_K$. 

\begin{lemma}\label{eigenfunctions nonlocal}
Let $K$ satisfy the upper bound \eqref{defi laplacian_K growth1}, and let $L>0$. Then, 
\begin{equation}\label{hamil eq4}
\LL_K  \big(e^{\frac{i\pi k}{L}\cdot}\big)(x)
=\ell_K({\textstyle \frac{\pi k}{L}})e^{\frac{i\pi k}{L}x}
\end{equation}
for all $k\in\Z$ and $x\in\R$, where $\ell_K$ is given by \eqref{definition lK(k)}. Moreover, the second inequality in \eqref{estimates l(k) lambda} holds and, if $K$ also satisfies the lower bound \eqref{defi laplacian_K growth2}, the first inequality in \eqref{estimates l(k) lambda} also holds.
In addition, if $\LL_K=(-\Delta)^s$ then $\ell_K(\xi)=|\xi|^{2s}$ for all $\xi\in\R$.
\end{lemma}
\begin{proof}
Thanks to \eqref{defi laplacian_K growth1}, the integral defining $\ell_K(\xi)$ in \eqref{definition lK(k)} is well defined. 

Since the case $k=0$ is clear, we may assume $k\neq0$. 
Now, by a change of variables we easily see that
\begin{equation}\label{eq: frac lap e_k}
\begin{split}
\LL_K \big(e^{\frac{i\pi k}{L}\cdot}\big)(x)&
=\int_\R d\widetilde z\,\big({e^{\frac{i\pi k}{L}x}-e^{\frac{i\pi k}{L}(x+\overline z)}}\big)K(|\widetilde z|)\\
&= e^{\frac{i\pi k}{L}x}
\frac{L}{\pi|k|}\int_\R dz\,({1-e^{iz}})K\Big(\frac{L|z|}{\pi|k|}\Big)
= \ell_K({\textstyle \frac{\pi k}{L}}) e^{\frac{i\pi k}{L}x},
\end{split}
\end{equation}
where we have used that $\sin(z)$ is and odd function. 
This proves \eqref{hamil eq4}.

We now address the other statements in the lemma. It is well known that 
\begin{equation}\label{limit c_s}
\int_\R dz\,\frac{1-\cos(z)}{|z|^{1+2s}}
=\frac{\sqrt\pi\,\Gamma(1-s)}{s4^s\Gamma(1/2+s)}=\frac{1}{c_s};
\end{equation}
see \cite[Lemma 3.1.3]{BucurValdinoci} and \eqref{defi laplacian_}.
Hence, in case that $\LL_K=(-\Delta)^s$, we have
\begin{equation}
\begin{split}
\ell_K(\xi)=|\xi|^{2s}
c_s\int_\R dz\,\frac{1-\cos(z)}{|z|^{1+2s}}
=|\xi|^{2s}.
\end{split}
\end{equation}
From this, the first and second inequalities in \eqref{estimates l(k) lambda} for a general kernel $K$ follow by the growth estimates \eqref{defi laplacian_K growth2} and \eqref{defi laplacian_K growth1}, respectively. 
\end{proof}

By similar arguments to the previous ones (since we did not require in \eqref{eq: frac lap e_k} that $k$ is an integer), it is easy to see that
if $\FF: L^2(\R)\to L^2(\R)$ denotes the Fourier transform then
$\FF(\LL_K v)=\ell_K\FF(v)$ for all $v\in L^2(\R)$ sufficiently smooth. Thus, $\ell_K$ is indeed the Fourier symbol of $\LL_K$ when acting on regular functions belonging to $L^2(\R)$.

Given a kernel $K$, let us now look at the Fourier description of the Lagrangian $\EE_{\LL_K}$. This will allow us to find a  Lagrangian $\EE_{\LL}$ for general multiplier operators $\LL$. 
First of all, note that if $u$ is a $2L$-periodic $\R$-valued function and $u_k$ are given by \eqref{eq: complex u_k}, 
then $u_{-k}=\overline{u_{k}}$ for all $k\in\Z$, where
$\overline{u_{k}}$ denotes the complex conjugate of $u_{k}\in\C$. In such case, we can also write
\begin{equation}\label{hamil eq6 formula Fourier}
u(x)=\frac{ a_0}{2}+\sum_{k>0}
\Big( a_k\cos\big({\textstyle \frac{\pi k}{L}\,x}\big)
+ b_k\sin\big({\textstyle \frac{\pi k}{L}\,x}\big)\Big),
\end{equation}
where 
\begin{equation}
 a_k:=\frac{1}{L}\intL dx\, u(x)\cos\big({\textstyle \frac{\pi k}{L}\,x}\big),\qquad
 b_k:=\frac{1}{L}\intL dx\, u(x)\sin\big({\textstyle \frac{\pi k}{L}\,x}\big)
\end{equation} 
for all $k\geq0$.
In fact, we have that $2u_k=a_k-ib_k$.
Hence, if $u$ is $\R$-valued, thanks to \Cref{eigenfunctions nonlocal} we see that the Lagrangian in \eqref{energy defi half oper} can be expressed as
\begin{equation}\label{s2 sobol norm statement LK}
\begin{split}
\EE_{\LL_K}(u)
= 2L\sum_{k\in\Z}\ell_K({\textstyle \frac{\pi k}{L}})\frac{1}{2}|u_k|^2-\int_{-L}^L G(u)
=\frac{L}{2}\sum_{k>0}\ell_K({\textstyle \frac{\pi k}{L}})(|a_k|^2+|b_k|^2)
-\int_{-L}^L G(u).
\end{split}
\end{equation}
With this at hand, it is natural to define the Lagrangian
$\EE_{\LL}$ on the Fourier side, for the operator 
\begin{equation}
\LL u(x)=\sum_{k\in\Z}\ell({\textstyle \frac{\pi k}{L}})u_ke^{\frac{i\pi k}{L}x}
\end{equation}
with $\ell$ as in \eqref{energy dispersive symbol even},
and for $2L$-periodic 
$\R$-valued functions $u$, by
\begin{equation}\label{s2 sobol norm statement}
\begin{split}
\EE_{\LL}(u)&:= L\Big(\ell(0)|u_0|^2+2\sum_{k>0}\ell({\textstyle \frac{\pi k}{L}})|u_k|^2\Big)-\int_{-L}^L G(u)\\
&=\frac{L}{4}\Big(\ell(0)|a_0|^2+2\sum_{k>0}\ell({\textstyle \frac{\pi k}{L}})
(|a_k|^2+|b_k|^2)\Big)
-\int_{-L}^L G(u).
\end{split}
\end{equation}
Here we have taken into account that, despite $\ell_K(0)=0$, one may have $\ell(0)\neq0$ for certain multiplier operators $\LL$. Using now that 
$u_{-k}=\overline{u_k}$ and $2u_k=a_k-i b_k$ for 
$\R$-valued functions,
and \eqref{energy dispersive symbol even},
from \eqref{s2 sobol norm statement} we see that actually 
\begin{equation}
\EE_{\LL}(u)=\int_{-L}^L \Big(\frac{1}{2}\,u\,\LL u-G(u)\Big),
\end{equation}
as expected by looking at identity \eqref{energy defi half oper} for $\EE_{\LL_K}$. 
All these computations are formal, but they indeed work if the symbol $\ell$ satisfies \eqref{energy dispersive symbol growth} and we require enough regularity on $u$ (and, therefore, we get fast enough decay of its Fourier coefficients $u_k$ compared to $\ell({\textstyle \frac{\pi k}{L}})$ as $k\uparrow+\infty$) in order to ensure that all the sums above are absolutely convergent; see \Cref{rmk regularity and decay} for more details.
 
To finish, in the following lemma we verify that $\EE_{\LL}$
is a functional whose critical points are solutions to the equation
$\LL u=g(u)$, where $G$ is of class $C^1$ in $\R$ and $g=G'$. This is, therefore, the semilinear extension of \Cref{critical points semilinear periodic} for $\EE_{\LL_K}$, now for multiplier operators on the Fourier side. 

\begin{lemma}\label{critical point energy dispersive}
Let $\LL$ be given by \eqref{energy dispersive symbol} with $\ell$ satisfying \eqref{energy dispersive symbol even} and \eqref{energy dispersive symbol growth}. Let $G$ be of class $C^1$ on~$\R$ and $g=G'$. Assume that $u:\R\to\R$ is a $2L$-periodic function such that
\begin{equation}
\sum_{k\in\Z}\big(1+|\ell({\textstyle \frac{\pi k}{L}})|\big)|u_k|^2
<+\infty
\end{equation}
and $g(u)\in L^1(-L,L)$.
Then,
\begin{equation}
\frac{d}{dt}\biggr\rvert_{t=0}\EE_{\LL}(u+t\psi)
=\int_{-L}^L \big(\LL u-g(u)\big)\psi
\end{equation}
for every $2L$-periodic smooth function $\psi:\R\to\R$. Therefore, if in addition $u$ is a critical point of $\EE_{\LL}$ among $2L$-periodic smooth perturbations, then $\LL u=g(u)$ in $L^1(-L,L)$.
\end{lemma}

\begin{proof}
Let $\psi_k$ be the complex Fourier coefficients of $\psi$ as in \eqref{eq: complex u_k}. By \eqref{energy dispersive symbol even}, we get
\begin{equation}\label{Energy dispersive critical eq1}
\begin{split}
\int_{-L}^L{\psi}\,\LL u
=2L\sum_{k\in\Z}\ell({\textstyle \frac{\pi k}{L}})u_k\psi_{-k}
=2L\sum_{k\in\Z}\ell({\textstyle \frac{\pi k}{L}})u_{-k}\psi_{k}
=\int_{-L}^L u\,\LL \psi.
\end{split}
\end{equation}
Therefore,
\begin{equation}
\begin{split}
\frac{d}{dt}\biggr\rvert_{t=0}\EE_{\LL}(u+t\psi)
=\int_{-L}^L\frac{1}{2}\big({\psi}\,\LL u+u\,\LL \psi\big)
-\int_{-L}^L g(u)\psi
=\int_{-L}^L \big(\LL u-g(u)\big)\psi,
\end{split}
\end{equation}
as desired. The last statement in the lemma follows by density. 

We point out that the assumption \eqref{energy dispersive symbol growth} on $\ell$ guarantees that
$\sum_{k\in\Z}|\ell({\textstyle \frac{\pi k}{L}})||\psi_k|^2<+\infty$ whenever $\psi$ is smooth.
This in turn is used to ensure, using the Cauchy-Schwarz inequality, that the above computations based on Parseval identity make sense.
\end{proof}

\begin{remark}\label{rmk regularity and decay}{\em
Here we analyze which H\"older regularity must be imposed on $u$ in order to ensure the assumption 
\begin{equation}\label{decreas coef 0}
\sum_{k\in\Z}|\ell({\textstyle \frac{\pi k}{L}})||u_k|^2<+\infty
\end{equation}
of \Cref{critical point energy dispersive}.
Clearly, \eqref{decreas coef 0} holds if, for some $\epsilon>0$, $u$ satisfies
\begin{equation}\label{decreas coef 1}
\limsup_{k\to\pm\infty}|k|^{1+\epsilon}|\ell({\textstyle \frac{\pi k}{L}})||u_k|^2<+\infty.
\end{equation}
In particular, if the Fourier multipliers of $\LL$ satisfy \eqref{energy dispersive symbol even} and the growth estimate 
\eqref{energy dispersive symbol growth}, then \eqref{decreas coef 1} holds whenever
\begin{equation}\label{decreas coef 2}
\limsup_{k\uparrow+\infty}|k|^{\frac{1+\epsilon+p}{2}}|u_k|<+\infty.
\end{equation}

Finally, by a simple argument using integration by parts, it is easy to show that  $|u_k|\leq C|k|^{-m}$ whenever $u\in C^m(\R)$ for some $m\in\N$. On the other hand, a classical result states that if
$u\in C^r(\R)$ for some $0<r\leq1$ then $|u_k|\leq C|k|^{-r}$; see \cite[estimate $(4.1)$ in page 45]{Zygmund}. A combination of these two estimates shows that \eqref{decreas coef 2} will hold for some $\epsilon>0$ 
whenever $u\in C^\alpha(\R)$ for some $\alpha>(p+1)/2$.  Hence, in such case \eqref{decreas coef 0} will also hold. In particular, if $\ell=\ell_K$ then by \eqref{estimates l(k) lambda} it suffices to take $\alpha>s+1/2$.
}\end{remark}


\section{Periodic nonlocal P\'olya-Szeg\H o inequality}\label{ss:symmetry}
 
The purpose of this section is to prove \Cref{thm rearrangement}, that is, the periodic symmetric decreasing rearrangement inequality for $[\cdot]_K$, for three different classes of kernels $K$. Recall that in \Cref{section:appendix rearrangement.kernels} we will show that none of these three classes is contained in the other. 

A crucial ingredient of the proof of \Cref{thm rearrangement} is the following Riesz re\-arrange\-ment inequality on the circle found independently in 1976 by Baernstein and Taylor \cite{Baernstein Taylor} and by Friedberg and Luttinger \cite{FriedLutt}. Recall that, for a $2\pi$-periodic function $f$, in \Cref{ss:symmetry_intro} we defined the rearrangements $f^{*}$ and $f^{*\!\eper}$.

\begin{theorem}$($\cite[Theorem 2]{Baernstein Taylor}, \cite[Theorem 1]{FriedLutt}, \cite[Theorem 2]{Burchard Hajaiej}, \cite[Theorem 7.3]{Baernstein book}$)$
\label{thm:gy:sharp Friedberg Luttinger}
Let $f,h,g:\R\to \R$ be three nonnegative $2\pi$-periodic measurable functions. Assume that $g$ is even, as well as nonincreasing in $(0,\pi)$. 

Then, 
\begin{equation}
\label{eq:thm FriedLutt inequality}
    \int_{-\pi}^{\pi}dx\int_{-\pi}^{\pi} dy\,f(x)g(x-y)h(y)
    \leq 
    \int_{-\pi}^{\pi}dx\int_{-\pi}^{\pi}dy\, f^{*}(x)g(x-y)h^{*}(y).
\end{equation}
In addition, if $g$ is decreasing in $(0,\pi)$ 
and the left-hand side of \eqref{eq:thm FriedLutt inequality} is finite,
then
equality holds in \eqref{eq:thm FriedLutt inequality} if and only if at least one of the following conditions holds: 
\begin{itemize}
\item[$(i)$] Either $f$ or $h$ is constant almost everywhere.

\item[$(ii)$] There exists $z\in \R$ such that $f(x)=f^{*\!\eper}(x+z)$ and $h(x)=h^{*\!\eper}(x+z)$ for almost every $x\in\R$.\footnote{This is equivalent to the condition 
$f(\cdot-z)=f^{*}$ and $h(\cdot-z)=h^{*}$ 
a.e. in $(-\pi,\pi)$. Note also that, in case $(i)$, if $f$ is constant then it could happen that $h$ does not agree with any translation  of its periodic rearrangement.}
\end{itemize}
\end{theorem}

Inequality \eqref{eq:thm FriedLutt inequality} was first discovered, independently, in \cite{Baernstein Taylor} and \cite{FriedLutt}. Both references contain more general inequalities: \cite{Baernstein Taylor} deals with the sphere $\mathbb{S}^n$, whereas \cite{FriedLutt} deals with a product of more than three functions. 
The result \cite[Theorem 7.3]{Baernstein book} by Baernstein is also more general than \eqref{eq:thm FriedLutt inequality}, as it deals with a Riesz rearrangement inequality on the sphere $\mathbb{S}^n$ and not only on the circle. Moreover, \cite{Baernstein book} treats more general functions of $f(x)$ and $h(y)$ than simply the product $f(x)h(y)$. 
Inequality \eqref{eq:thm FriedLutt inequality} can be found in \cite{Baernstein Circle}, another work of Baernstein, in a more general form where $g$ is also rearranged.

Instead, the statement in \Cref{thm:gy:sharp Friedberg Luttinger} concerning equality in \eqref{eq:thm FriedLutt inequality} follows from Burchard and Hajaiej \cite[Theorem 2]{Burchard Hajaiej}, who treated the case of equality in $\mathbb{S}^n$ for the first time, even for $n=1$. For this result, we also cite \cite[Theorem 7.3]{Baernstein book} since, being less general than \cite{Burchard Hajaiej}, fits precisely with our statement in \Cref{thm:gy:sharp Friedberg Luttinger}.

We find reference \cite{Baernstein book} to be the simplest one for looking up all assertions of \Cref{thm:gy:sharp Friedberg Luttinger}.

All the above mentioned references, except for \cite{Baernstein Circle}, use the {\em method of polarization} to prove  \Cref{thm:gy:sharp Friedberg Luttinger}. When finishing this article, we have found that the method of polarization has been also used in the following very recent papers on fractional equations. They do not deal, however, with periodic solutions. First, DelaTorre and Parini \cite{delatorre parini polar} and Dieb, Ianni, and Salda\~na \cite{DiebIanniSaldana polar} use the method to prove uniqueness of least energy solutions for fractional equations, while Benedikt, Bobkov, Dhara, and Girg \cite{BenBobDharaGirg Polar} employ it to establish the nonradiality of the second eigenfunction of the fractional Laplacian in a ball. 

\begin{proof}[Proof of \Cref{thm rearrangement}]
Since the periodic Riesz rearrangement inequality \eqref{eq:thm FriedLutt inequality} requires the functions $f$ and $h$ to be nonnegative, we first need to reduce the proof to the case $u\geq 0$.

First, observe that $||u(x)|-|u(y)||\leq|u(x)-u(y)|$ for all $x$ and $y$ in $\R$, and that 
$u^{*\eper}=|u|^{*\eper}$. Thus, to prove part $(a)$ we can assume without loss of generality that $u$ is nonnegative.

Regarding part $(b)$, assume that
$[u^{*\eper}]_K= [u]_K<+\infty.$ We claim that $u$ does not change sign (in the usual ``almost everywhere'' or ``essential'' sense, since we only assume $u\in L^2(-L,L)$). Indeed, otherwise there would exist two subsets 
$U$ and $V$ of $\R$ with positive measure such that 
$||u(x)|-|u(y)||<|u(x)-u(y)|$ for all $x\in U$ and $y\in V$. In addition, due to periodicity, we can assume that $|x-y|<L$ for all $x\in U$ and all $y\in V$ making the sets $U$ and $V$ smaller if necessary.  Let us now use that $K>0$ in $(0,L)$ (actually $K>0$ everywhere in cases $(i)'$ and $(ii)'$) and that $[|u|^{*\eper}]_K\leq[|u|]_K$, since we will have already proved part $(a)$ for nonnegative functions. We deduce that
$[|u|^{*\eper}]_K\leq[|u|]_K<[u]_K=[u^{*\eper}]_K$, leading to a contradiction with the fact that $u^{*\eper}=|u|^{*\eper}$.

To prove part $(b)$ of the theorem, since we now know that $u$ does not change sign, we will later assume $u\geq0$ and prove that $u=u^{*\eper}(\cdot+z)$ for some $z\in\R$. In case that $u\leq0$ we may replace $u$ by $-u$ (since
$(-u)^{*\eper}=u^{*\eper}$ and $[-u]_K=[u]_K$) and then conclude $u=-u^{*\eper}(\cdot+z)$.

Since in one step of the proof ---namely, in \eqref{eq:gy:splitting thanks Ki} below--- we will need the kernel to be integrable around its singularity, we first approximate it by bounded kernels. We will do this in a different way for each of the three classes $(i)$, $(ii)$, and $(iii)$. This first step is not necessary if $K$ is bounded. 

We first deal with the case $(i)$ of convex kernels.
Let $\{t_i\}_{i\in\N}$ be a sequence of positive numbers tending to zero and such that $K$ is differentiable at $t_i$. We replace $K$ in $(0,t_i)$ by the affine function which is tangent to the graph of $K$ at the point $(t_i,K(t_i))$. Since $K$ is convex, we obtain in this way a nondecreasing sequence of bounded and convex kernels $K_i$ in $[0,+\infty)$ converging to $K$ as $t_i\downarrow0$. By monotone convergence, it thus suffices to show the claimed rearrangement inequality for each of the bounded convex kernels $K_i$.

Now we explain the approximation for $K$ as in $(ii)$. In this case we consider, for $i=1,2,3,\ldots,$ the kernels $K_i(t):=K(\sqrt{t^2+1/i^2})$ for $t\geq0$. If $K$ is as in \eqref{kernel Laplace transform} for some nonnegative measure~$\mu$, then 
\begin{equation}\label{kernel Laplace transform_i}
K_i(t)=\int_0^{+\infty}d\mu(r)\,e^{-r/i^2}e^{-t^2r}
\quad\text{for all $t\geq0$}.
\end{equation}
By the comments following \Cref{thm rearrangement}, $\tau>0\mapsto K_i(\tau^{1/2})$ is completely monotonic since it is the Laplace transform of the nonnegative measure 
$e^{-r/i^2}d\mu(r)$.
Also, by monotone convergence we see that $K_i(t)\uparrow K(t)$ for all $t>0$ as $i\uparrow+\infty$. Finally, for any given $i\geq1$, we have $K_i(t)\leq K_i(0)=K(1/i)<+\infty$, and hence $K_i$ are bounded kernels in~$[0,+\infty)$. Once again, by monotone convergence, it thus suffices to show the claimed rearrangement inequality for each of the bounded kernels $K_i$.

In the case $(iii)$, we define $K_i(t):=K(t)$ if $t\geq 1/i$ and $K_i(t)=:K(1/i)$ if $t\in (0,1/i).$ We have that $K_i$ is bounded and still belongs to the class $(iii)$.

Now, in whichever of the three ways that $K_i$ has been defined, set
\begin{equation}\label{defi Ki K convex}
\overline K_i(t):=\sum_{k\in\Z}K_i(|t+2kL|)
\quad \text{for $t\in \R$}.
\end{equation}
Note that $\overline K_i$ is bounded uniformly in $t\in\R$ by the upper bound \eqref{defi laplacian_K growth1} and since $K_i$ is bounded. Using that $u$ is $2L$-periodic, we write 
\begin{equation}
 \label{eq:gy:int of Ki and KiBar} 
\begin{split}
2[u]_{K_i}^2=\intL dx\!\int_\R dy\,|u(x)-u(y)|^2&K_i(|x-y|)
=\intL dx\!\intL dy\,|u(x)-u(y)|^2\overline K_i(x-y).
\end{split}
\end{equation}
Since $\overline K_i$ is integrable around the origin, we can expand $|u(x)-u(y)|^2$ to get
\begin{equation}
\label{eq:gy:splitting thanks Ki}
\begin{split}
\intL dx\intL dy\,|u(x)-u(y)|^2\overline K_i(x-y)
&=\intL dx\,u(x)^2\intL dy\,\overline K_i(x-y)\\
&\quad+\intL dy\, u(y)^2\intL dx\,\overline K_i(x-y)\\
&\quad-2\intL dx\intL dy\, u(x)u(y)\overline K_i(x-y).
\end{split}
\end{equation}
Note that $\overline K_i$ is even with respect to $x=0$, $2L$-periodic, and bounded. Hence, we have that
$$\int_{-L}^Ldy\,\overline K_i(x-y)=\int_{-L}^L\overline K_i<+\infty.$$
In addition, $\intL u^2=\intL (u^{*\eper})^2$. Therefore, the claimed inequality of the theorem will follow if we show that
\begin{equation}\label{rearrange eq4}
\begin{split}
\intL dx\intL dy\, u(x)u(y)\overline K_i(x-y)
\leq\intL dx\intL dy\, u^{*\eper}(x)u^{*\eper}(y)\overline K_i(x-y).
\end{split}
\end{equation}

To prove \eqref{rearrange eq4}, by scaling we may assume that $L=\pi$. According to \Cref{thm:gy:sharp Friedberg Luttinger}, to get \eqref{rearrange eq4} it is enough to show that $\overline K_i$ is a nonnegative and nonincreasing function in $(0,\pi)$. That $\overline K_i\geq 0$ is clear from its definition in terms of $K_i$.  We now show the monotonicity of $\overline K_i$ in two different ways, depending on whether we are in assumption $(i)$ or $(ii)$ of the theorem. Note that in the case $(iii)$ it holds that $\overline  K_i (t)=K_i(|t|)$ for all $t\in (-\pi,\pi),$ and hence the monotonicity of $\overline  K_i$ is an immediate consequence of the assumptions in $(iii)$.

\medskip
\noindent\underline{{\em Proof of the inequality for $K$ as in $(i)$}\,:} 
\smallskip

To show the monotonicity of $\overline{K}_i$ in $(0,\pi)$, we rename by $-k-1$, with $k=0,1,2,\ldots$,  those indices which are negative in the definition of  $\overline{K}_i$. In this way we obtain 
\begin{equation}
\overline{K}_i(t)=\sum_{k\geq0}\big(K_i(2k\pi+t)+K_i(2(k+1)\pi-t)\big)
\quad\text{ for }t\in (0,\pi).
\end{equation}
Thus, if $0<t_1<t_2<\pi$, we have 
\begin{equation}
\label{eq:gy:monoton of Kibar}
  \begin{split}
   \overline{K}_i(t_1)-\overline{K}_i(t_2)=&\sum_{k\geq0}\big(K_i(2k\pi+t_1)+K_i(2(k+1)\pi-t_1)\\
   &
   \qquad -K_i(2k\pi+t_2)-K_i(2(k+1)\pi-t_2)\big).
   \end{split}
\end{equation}

We now observe that if $f$ is a convex function, $a<b<c<d$, and $a+d=b+c,$ then 
\begin{equation}
  \label{eq:strict conv in i}
    f(b)+f(c)\leq f(a)+f(d), \quad\text{  with strict inequality if $f$ is stictly convex in $[a,d]$}.
\end{equation}
This follows from convexity, since for $\lambda=(d-b)/(d-a)\in (0,1)$ we have  that $b=\lambda a+(1-\lambda) d$ and $c=(1-\lambda) a+\lambda d$. 

Using \eqref{eq:strict conv in i} with $a=2k\pi+t_1$, $b=2k\pi+t_2$, $c=2(k+1)\pi-t_2$, and $d=2(k+1)\pi-t_1$, we see that each term in the sum over $k\geq 0$ of \eqref{eq:gy:monoton of Kibar} is nonnegative. Thus, $\overline K_i$ is nonincreasing in $(0,\pi)$. This leads to \eqref{rearrange eq4} and finishes the proof of the inequality stated in the theorem for convex kernels $K$.

\medskip
\noindent\underline{{\em Case of equality for $K$ as in 
$(i)'$}\,:} 
\smallskip

We now assume that $K$ is convex in $(0,+\infty)$ and strictly convex in $(c,+\infty)$ for some $c>0$. Recall that we may assume  $u\geq 0$.

Set $M(t):=K(t)-K_1(t),$ where $K_1$ is the kernel constructed at the beginning of the proof for $i=1$. Recall that $t_1$ is a point of differentiability of $K$ (for which we may assume that $0<t_1<c$) and that $K_1$ is characterized by: $K_1=K$ in $(t_1,+\infty)$, 
$K_1$ is affine in $(0,t_1)$, $K_1$ is convex in 
$(0,+\infty)$, and $K_1\leq K$ in $(0,+\infty)$. Observe that both $K_1$ and $M$ are nonnegative convex functions in $(0,+\infty)$ that satisfy the upper bound \eqref{defi laplacian_K growth1}. We now split the integral  
\begin{equation}\label{split K in K_1,M_bis}
\int_{-L}^Ldx\int_{\R}dy\,|u(x)-u(y)|^2K(|x-y|)=\EE_1(u)+\mathcal{M}(u),
\end{equation}
where
\begin{equation}\label{split K in K_1,M}
\begin{split}
\EE_1(u)&:=\int_{-L}^Ldx\int_{\R}dy\,|u(x)-u(y)|^2K_1(|x-y|),
 \\
\mathcal{M}(u)&:= \int_{-L}^Ldx\,\int_{\R}dy\,|u(x)-u(y)|^2 M(|x-y|).
\end{split}
\end{equation}
By the inequality of  \Cref{thm rearrangement} $(i)$ (applied now to the kernels $K_1$ and $M$) 
we have $\EE_1(u^{*\eper})\leq \EE_1(u)$ and $\mathcal{M}(u^{*\eper})\leq \mathcal{M}(u).$ Combining these two inequalities with the assumption that $[u^{*\eper}]_K= [u]_K<+\infty$, we deduce $\EE_1(u^{*\eper})= \EE_1(u)$ (observe here that all these quantities are finite, 
since $[u^{*\eper}]_K= [u]_K<+\infty$, and $0\leq K_1\leq K$ and $0\leq M\leq K$). 
It now follows from \eqref{eq:gy:int of Ki and KiBar} and \eqref{eq:gy:splitting thanks Ki} (which required the kernel $K_1$ to be integrable) that there must be equality in \eqref{rearrange eq4} for $i=1$, that is, for the kernel $K_1$. 

Now, as before, by scaling we may assume $L=\pi$. From \eqref{eq:gy:monoton of Kibar} and \eqref{eq:strict conv in i} we obtain that $\overline{K}_1$ is decreasing in $(0,L)=(0,\pi)$, since $K_1$ is convex in $(0,+\infty)$ and strictly convex in $(c,+\infty)$ 
(and, hence, for each $k>c/(2\pi)$ the  $k$-th summand  in \eqref{eq:gy:monoton of Kibar} is positive, and all others are nonnegative). As there must be equality in \eqref{rearrange eq4} for $\overline{K}_i=\overline{K}_1$, we deduce from the strict Riesz rearrangement inequality on the circle (\Cref{thm:gy:sharp Friedberg Luttinger} $(i)$ and $(ii)$; which requires the kernel $\overline{K}_1$ to be decreasing in $(0,\pi)$ and also $u\geq 0$) that a translate of $u$ is equal to $u^{*\eper}.$  

\medskip
\noindent\underline{{\em Proof of the inequality for $K$ as in $(ii)$}\,:} 
\smallskip

To check the monotonicity of $\overline K_i$, from \eqref{kernel Laplace transform_i} we see that
\begin{equation}\label{overlineKi_heatkernel}
\begin{split}
\overline K_i(t)=\sum_{k\in\Z}K_i(|t+2kL|)
=\sum_{k\in\Z}\int_0^{+\infty}d\mu(r)\,e^{-r/i^2}
e^{-(t+2kL)^2r}
=\int_0^{+\infty}d\mu(r)\,e^{-r/i^2}\,\Phi(t,r),
\end{split}
\end{equation}
where $\Phi(t,r):=\sum_{k\in \Z}e^{-(t+2kL)^2r}$. Now we claim that, for every $r>0$, the function
$t\mapsto \Phi(t,r)$ is  decreasing in $(0,L)$.
This follows from the fact that the fundamental solution of the heat equation with $2L$-periodic boundary conditions is decreasing in $(0,L)$; we refer to \cite[Appendix~B]{Cabre Csato Mas delaunay} where one can find several references for this result, but also an elementary self-contained proof. Therefore, by \eqref{overlineKi_heatkernel} and since the measure $\mu$ is nonnegative, $\overline K_i$ is nonincreasing in 
$(0,L)$. By the comments right after \eqref{rearrange eq4}, this finishes the proof.

\medskip
\noindent\underline{{\em Case of equality for $K$ as in 
$(ii)'$}\,:}  
\smallskip

The arguments will be similar to the ones above for the case $(i)'$. Recall that we may assume $u\geq0$.
Set $M(t):=K(t)-K_1(t),$ where $K_1(t)=K(\sqrt{t^2+1})$. Hence 
$$
  M(t)=\int_0^{+\infty}d\mu(r)\,(1-e^{-r})e^{-t^2 r} .
$$
Observe that both $K_1(\sqrt{\tau})$ and $M(\sqrt{\tau})$ are completely monotonic functions of $\tau\in(0,+\infty)$
as can be seen by differentiating under the integral sign.
Both also satisfy the upper bound \eqref{defi laplacian_K growth1}. 
We now perform the splitting \eqref{split K in K_1,M_bis} but using the current definitions of $K_1$ and $M$ in \eqref{split K in K_1,M}. From part $(ii)$ of the theorem, we know that $\EE_1(u^{*\eper})\leq \EE_1(u)$ and $\mathcal{M}(u^{*\eper})\leq \mathcal{M}(u).$ 
Now, combining these two inequalities 
with the assumption that $[u^{*\eper}]_K= [u]_K<+\infty$, we deduce $\EE_1(u^{*\eper})= \EE_1(u)$. 
Since $K_1$ is bounded in $[0,+\infty)$, we can argue as in \eqref{eq:gy:int of Ki and KiBar} and \eqref{eq:gy:splitting thanks Ki} to deduce that there must be equality in \eqref{rearrange eq4} for $i=1$. 
As before, by scaling we can also assume $L=\pi$. 
Since we know that $\overline K_1$ is decreasing in 
$(0,\pi)$ by \eqref{overlineKi_heatkernel} (we use here the hypothesis  $K\not\equiv 0$ to ensure that $\mu\not\equiv 0$), from the strict Riesz rearrangement inequality on the circle (see \Cref{thm:gy:sharp Friedberg Luttinger} 
 $(i)$ and $(ii)$) we deduce that a translate of $u$ is equal to $u^{*\eper}.$
 
\medskip
\noindent\underline{{\em Proof of the inequality for $K$ as in 
$(iii)$}\,:} 
\smallskip 

As pointed out after \eqref{rearrange eq4}, after rescaling to have $L=\pi$,  the proof in this case is clear, since $\overline{K}_i=K_i(|\cdot|)$ is nonincreasing in $(0,\pi)$ for all $i\geq 1$. 
 
\medskip
\noindent\underline{{\em Case of equality for $K$ as in 
$(iii)'$}\,:} 
\smallskip

The argument is very similar to the previous ones. From them we see that it suffices to split $K$  as $K=K_0+M$, where $K_0$ and $M$ both satisfy the hypothesis of \Cref{thm rearrangement} $(iii)$, and in addition $K_0$ is bounded and decreasing in $(0,L)$.\footnote{Note that using $K_1$ (i.e., the cutoff of $K$ described in the first step of the proof of the theorem for case $(iii)$) as a candidate for $K_0$ will not work, since $K_1$ is constant near $0$. }

We accomplish this as follows. Define $K_0:[0,+\infty)\to [0,1]$ by
$$
  K_0(t):=1-\frac{1}{K(t)+1}=\frac{K(t)}{K(t)+1}\,.
$$
Note that $K_0$ is bounded, vanishes in $[L,+\infty)$, and is decreasing in $[0,L]$ (since $K$ is decreasing in this interval). It remains to check that $M(t):=K(t)-K_0(t)$ is nonnegative, 
satisfies the upper bound \eqref{defi laplacian_K growth1},
vanishes in $[L,+\infty)$, and is nonincreasing in $(0,L)$. The first three properties are obvious, since $M(t)$ is given by
$$
  M(t)=\frac{K(t)^2}{K(t)+1}\leq K(t).
$$
Concerning the monotonicity of $M$, let $0< t_1<t_2\leq L$. 
Since $K(t_1)\geq K(t_2)$, we have
\begin{align}
  \frac{K^2(t_1)}{K^2(t_2)}\geq \frac{K(t_1)}{K(t_2)}
  \geq \frac{K(t_1)+1}{K(t_2)+1},
\end{align} 
which shows the monotonicity of $M$.
\end{proof}

\section{A strong maximum principle for periodic weak solutions}
\label{section:max principle}



To prove the strict monotonicity in the statements of \Cref{thm: frac Laplacian Intro} and \Cref{coro constrained minim lagrangian2}, we need the strong maximum principle \Cref{theorem:max principle} for periodic solutions. We had never seen such type of result for integro-differential equations in the periodic setting.

Towards the proof, let us try to determine the sign of $\LL_Kv(x_0)$ at some $x_0\in (0,L)$ where $v(x_0)=0$ by splitting the integral over $\R$ in the definition of $\LL_K$  into intervals of length $L$. Recall that, in \Cref{theorem:max principle}, $v$ is odd, $2L$-periodic, and with $v\leq0$ in $(0,L)$. For the first two intervals, the integral of $-v(x)K(|x_0-x|)$ over $(-L,0)\cup(0,L)$ will give a positive sign, since $v$ is odd and the negative mass of $v$ is closer to $x_0$ than the positive one. Instead, by the same reason, the integral over $(-2L,-L)\cup(L,2L)$ will give a negative sign, and so on, leading to an alternating sequence. There are thus competing terms and it is not clear why 
$\LL_Kv(x_0)$ should be positive. If $K(\sqrt{\cdot})$ is completely monotonic we bypass this difficulty  using the Laplace transform, 
as in the proof of \Cref{thm rearrangement},  and obtain a sign for $\LL_Kv(x_0)$. Whereas if $K$ is convex as in $(i)'$, a simpler direct proof is indeed possible, by comparing groups of four intervals of length $L$ (and not only two). The details go as follows.

\begin{proof}[Proof of \Cref{theorem:max principle}.]
 We provide first a detailed proof for the case when $K(\sqrt{\cdot})$ is completely monotonic. At the end of the proof we show why the same proof works for the slightly simpler cases $(i)'$ and $(iii)'$, with the obvious adaptations.
 
For $K(\sqrt{\cdot})$ completely monotonic, we have split the proof in two cases for the sake of clarity of the ideas, although the second proof is valid for both cases.

\medskip
\noindent\underline{{\em Case $K(\sqrt{\cdot})$ completely monotonic and $s<1/2$}\,:} 
\smallskip

Recall that if $s<1/2$, then the principal value in the definition of $\LL_K v(x_0)$ is not necessary, since $x\mapsto |v(x_0)-v(x)||x_0-x|^{-1-2s}$ belongs to $L^1(\R)$; recall that $v\in C^{2s+\gamma}(\R)$. This justifies the forthcoming calculations.

We prove the result by contradiction. Assume that $v$ is not identically zero and that there exists $x_0\in (0,L)$ such that $v(x_0)=0$.  
Then, using the  expression for $K$ in \eqref{kernel Laplace transform} based on the Laplace transform, we have
\begin{equation}\label{eq:Appendix maxprinc G}
\begin{split}
  \LL_K v(x_0)&=\int_\R dx\,
(v(x_0)-v(x))K(|x_0-x|)
  =
  -\int_0^{+\infty}d\mu(r)\,\int_\R dx\, v(x)e^{-|x_0-x|^2r}
  \\
  &=
  -\int_0^{+\infty}d\mu(r)\,\sum_{k\in\Z}\int_{-L}^{L} dx\, v(x)e^{-|x_0-x+2kL|^2r}
  \\
  &=-\int_0^{+\infty}d\mu(r)\int_{-L}^{L} dx\, v(x)\Phi(x_0-x,r),
\end{split}
\end{equation}
where $\Phi(t,r):=\sum_{k\in\Z}e^{-(t+2kL)^2 r}.$

We know that the function
$t\mapsto \Phi(t,r)$ is even, $2L$-periodic, and decreasing in $(0,L)$, as explained in the proof of \Cref{thm rearrangement} $(ii)$; see the comment right after \eqref{overlineKi_heatkernel}.  Splitting the integration domain as $(-L,L)=(-L,0)\cup (0,L)$ and using that $v$ is odd, one obtains that
\begin{equation}\label{eq:maxprinc symmmetrize}
  \int_{-L}^{L} dx\, v(x)\Phi(x_0-x,r)
  =\int_0^L dx\, v(x)\left(\Phi(x_0-x,r)-\Phi(x_0+x,r)\right).
\end{equation}
Since $\Phi(\cdot,r)$ is even, $2L$-periodic, and decreasing in $(0,L)$ it holds that 
$$
  \Phi(x_0-x,r)=\Phi(|x_0-x|,r)>\Phi(x_0+x,r)
$$
for all $x_0\in (0,L)$ and all $x\in(0,L)$. This  follows from the inequalities $|x_0-x|<x_0+x$ and $|x_0-x|<2L-(x_0+x)$, which mean that $|x_0-x|$ is closer to $0$ than $x_0+x$ to $0$ and to $2L$. 

Finally, using that $v\leq 0$ in $(0,L)$ and $v\not\equiv 0$ we obtain that \eqref{eq:maxprinc symmmetrize} is negative and thus, by \eqref{eq:Appendix maxprinc G}, that $\LL_K v(x_0)>0.$ This contradicts the inequality $\LL_K v(x_0)- c(x_0)v(x_0)\leq 0$, since we are assuming that $v(x_0)=0$.

\medskip
\noindent\underline{{\em Case $K(\sqrt{\cdot})$ completely monotonic and $1/2\leq s<1$}\,:} 
\smallskip

The proof follows that of the case $s<1/2$, with the only difference that we must to take into account the principal value in the definition of $\LL_K$. We again assume that $v(x_0)=0$ for some $x_0\in (0,L).$ In order to have more symmetric expressions in what follows, we observe, since $v$ is bounded, that 
$$
  \lim_{\epsilon\downarrow 0} 
  \int_{A_{\epsilon}}
  dx\, \frac{v(x_0)-v(x)}{|x_0-x|^{1+2s}}=0, 
  \quad\text{ where } 
  A_{\epsilon}:=\bigcup_{k\in \Z\setminus\{0\}}\left((x_0-\epsilon,x_0+\epsilon)+2kL\right),
$$
and
$$
  \lim_{\epsilon\downarrow 0} 
  \int_{D_{\epsilon}}
  dx\, \frac{v(x_0)-v(x)}{|x_0-x|^{1+2s}}=0, 
  \quad\text{ where } 
  D_{\epsilon}:=\bigcup_{k\in \Z}\left((-x_0-\epsilon,-x_0+\epsilon)+2kL\right).
$$
Hence,  proceeding as in \eqref{eq:Appendix maxprinc G} and \eqref{eq:maxprinc symmmetrize} we obtain (using also the upper bound \eqref{defi laplacian_K growth1})
\begin{equation}
\begin{split}
  \LL_K v(x_0)&=\lim_{\epsilon\downarrow 0}\int_{\R\setminus(x_0-\epsilon,x_0+\epsilon)} dx\,
(v(x_0)-v(x))K(|x_0-x|)
\\
&=\lim_{\epsilon\downarrow 0}\int_{\R\setminus\left(A_\epsilon\cup D_\epsilon\cup (x_0-\epsilon,x_0+\epsilon)\right)} dx\,
(v(x_0)-v(x))K(|x_0-x|)
\\
  &=-\lim_{\epsilon\downarrow 0}\int_0^{+\infty}d\mu(r)\,\int_{(-L,L)\setminus \left((x_0-\epsilon,x_0+\epsilon)\cup (-x_0-\epsilon,-x_0+\epsilon)\right)} dx\, v(x)\Phi(x_0-x,r)
  \\
  &=
  \lim_{\epsilon\downarrow 0}\int_0^{+\infty}d\mu(r)\,\int_{(0,L)
  \setminus(x_0-\epsilon,x_0+\epsilon)}
  dx\, (-v(x))\left(\Phi(x_0-x,r)-\Phi(x_0+x,r)\right).
\end{split}
\end{equation}
We assume again, to reach a contradiction, that there exists $z_0\in (0,L)$ such that $v(z_0)<0.$ By what we have seen in the previous case $s<1/2$ concerning the function $\Phi$, we know that the integrand satisfies
$$
-v(x)\left(\Phi(x_0-x,r)-\Phi(x_0+x,r)\right)\geq 0
\quad\text{ for all }x\in (0,L).
$$ 
Hence, $\LL_K v(x_0)$ is a limit of a nondecreasing sequence in $\epsilon$, a sequence which, moreover, is positive for $\epsilon$ small enough (precisely when $z_0\notin [x_0-\epsilon,x_0+\epsilon]$). This shows that $\LL_K v(x_0)>0$, which leads to the contradiction. 

\medskip
\noindent\underline{{\em Case $K$ convex or $K$ as in  $(iii)'$}\,:} 
\smallskip

In this case \eqref{eq:Appendix maxprinc G} must be replaced by
$$
  \LL_K v(x_0)=-\int_{-L}^L dx\,v(x)\overline{K}(x_0-x),
  \quad\text{ where } \overline{K}(t):=\sum_{k\in\Z}K(|t+2kL|).
$$
Note that $\overline{K}$ has the same properties as $\Phi(\cdot,r)$: it is even, $2L$-periodic, and decreasing in $(0,L).$ The third property is obvious for $K$ as in $(iii)'$. Instead, for $K$ as in $(i)'$, it follows exactly arguing as in \eqref{eq:gy:monoton of Kibar}  and \eqref{eq:strict conv in i} (with $K_i$ and $\pi$ replaced by $K$ and $L$), and as in the last paragraph of the proof of \Cref{thm rearrangement} $(b)$ for $(i)'$. Hence, the entire argument of the proof in case $(ii)'$ applies to the current cases $(i)'$ and $(iii)'$, with the only modification that $\Phi$ is to be replaced by $\overline{K}$ and the integral in $d\mu(r)$ is to be removed.
\end{proof}

\section{Proofs of \Cref{thm: frac Laplacian Intro} and \Cref{coro constrained minim lagrangian2}}
\label{section:proof of thm 1.1 and corollary constr. min}

In this section we prove the symmetry and monotonicity results for  constrained minimizers, \Cref{coro constrained minim lagrangian2} and \Cref{thm: frac Laplacian Intro} ---the latter being essentially a special case of \Cref{coro constrained minim lagrangian2}.

\begin{proof}[Proof of \Cref{coro constrained minim lagrangian2}]
Recall that $u^{*\eper}\chi_{(-L,L)}$ is the Schwarz rearrangement of the absolute value $|u|\chi_{(-L,L)}$. Thus it will be useful to
 first assume that $u$ is nonnegative. With $u\geq0$, it follows that
$$
  \int_{-L}^L G(u)=\int_{-L}^L G(|u|)=\int_{-L}^L G(u^{*\eper}).
$$
The last equality follows from the fact that $u^{*\eper}$ is a rearrangement of $|u|$ and by standard properties of rearrangements; see for instance \cite[Section 3.3]{Lieb Loss}. For the same reason the constraint $\int_{-L}^L \widetilde{G}(v)=c$ is satisfied by $u^{*\eper}$ too.
Thus, $u^{*\eper}\in L^\infty(\R)$ is an admissible competitor.

Since $u$ is a constrained minimizer of $\EE_{\LL_K}$, we deduce that $[u]_K\leq [u^{*\eper}]_K$. But the reversed inequality also holds by \Cref{thm rearrangement} $(a)$. We conclude that there is equality, and by \Cref{thm rearrangement} $(b)$, that $u=\pm u^{*\eper}(\cdot+z)$ for some $z\in\R$. Since $u\geq0$, we must have $u= u^{*\eper}(\cdot+z)$. From this, the first conclusion of the corollary (i.e., after a translation, $u$ is even and nonincreasing in $(0,L)$) follows.

Note that, when assuming $u\geq 0$, we do not need the condition  $u\in L^{\infty}(\R)$ in the  previous argument in case we minimize among a class of functions which are not bounded, such as for instance $L^2(-L,L)$.

Let now $u$ be arbitrary (not necessarily nonnegative). In view of the assumption $u\in L^{\infty}(\R)$ there exists a constant $d$ such that $u+d\geq 0$. Define now
$$
  G^d(t):=G(t-d),\quad\text{ and }\quad \EE_{\LL_K}^d(v):=
  \frac{1}{4}[v]_K^2-\int_{-L}^L G^d(v),
$$
and $\widetilde{G}^d(t):=\widetilde{G}(t-d).$ It is easy to check that if $u$ is a constrained minimizer of $ \EE_{\LL_K}$ then $u+d$ is a constrained minimizer of $\EE_{\LL_K}^d$ under the constraint $\int_{-L}^L \widetilde{G}^d(v)=c$. Applying now the corollary to the nonnegative constrained minimizer  $u+d$ we obtain that 
$$
  u= (u+d)^{*\eper}(\cdot +z)-d\quad\text{ for some }z\in \R.
$$
Thus $u$ has the claimed form, i.e., $u$ is even and nonincreasing in $(0,L)$.

Let us now prove the second statement of the corollary. 
Note that, since 
$u\in L^\infty(\R)$ and $G$ and $\widetilde{G}$ belong to $C^{2+\delta}(\R)$ for some $\delta>0$, both $G(u)$ and $g(u)$ belong to $L^\infty(\R)$. This justifies the forthcoming arguments. Also observe that for a constrained minimizer $u$ the semi-norm $[u]_K$ must be finite, since $G(u)$ is bounded. 

First, since $\EE_{\LL_K}(u)\leq \EE_{\LL_K}(v)$ for all $2L$-periodic functions $v\in L^\infty(\R)$, $u$ is a critical point of $\EE_{\LL_K}$ among $2L$-periodic smooth perturbations satisfying the constraint and, thus, by \Cref{critical points semilinear periodic} (and the paragraph following it), 
$\LL_K u =g(u)+\lambda \widetilde{g}(u)$ in $\R$ in the periodic weak sense for some $\lambda\in\R$. Thanks to \eqref{fl eq0} in \Cref{fl l1}, this last equation also holds in the periodic distributional sense. 

We can now apply the regularity \Cref{fthm:Holder reg periodic} $(ii)$ (with $\beta=1+\delta$) to the equation 
\begin{equation}
\label{eq:gy:c1beta_right_side}
  \LL_K u=h(u)\quad\text{in $\R,\quad$ where }\quad h:=g+\lambda\widetilde{g}   
  \in C^{1+\delta}(\R).
\end{equation}
 Therefore we obtain  that $u\in C^{1+\nu}(\R)$ for some $\nu>2s$. Hence $u'\in C^{\nu}(\R)$ and we can differentiate the equation to obtain
$$
  \LL_K u'=h'(u) u'\quad\text{ in }\R.
$$
In view of the evenness and monotonicity result of the corollary we know that $u'$ is odd and $u'\leq 0$ in $(0,L)$. Applying now \Cref{theorem:max principle} to $v=u'$, we obtain that $u'<0$ in $(0,L)$, unless $u$ is constant.   
\end{proof}

\begin{remark}\label{remark: gen const}
{\em The statement of \Cref{coro constrained minim lagrangian2} dealing with the monotonicity (but not the strict monotonicity) remains true for more general constraints than those of the form $\int_{-L}^L \widetilde{G}(u)=c.$ Indeed, any constraint which is preserved by the symmetric decreasing rearrangement in $(-L,L)$ can be assumed, as long as  $G$ is even (or equivalently, depends  only on the absolute value of its argument). Under such assumptions the proof remains identical, with the only difference that the first step ---assuming $u$ to be nonnegative--- remains superfluous. 

An example for such a constraint is the quasi-norm in the Lorentz space $L^{p,q}(-L,L)$, which is given for 
$1\leq p<+\infty$ and $1\leq q\leq+\infty$ by
\begin{equation}
\|u\|_{L^{p,q}(-L,L)}:=
\begin{cases}
\Big(p{\displaystyle \int_0^{+\infty}}\!dt\,
t^{q-1} \big|\{x\in(-L,L):\,|u(x)|>t\}\big|^{q/p}\Big)^{1/q},
&1\leq q<+\infty,\\
\sup_{t>0}t\big|\{x\in(-L,L):\,|u(x)|>t\}\big|^{1/p}, &q=+\infty.
\end{cases}
\end{equation}
Note that 
$\|u\|_{L^{p,q}(-L,L)}$ cannot be expressed as $\intL \widetilde{G}({u})$ when $p\neq q$, but it is invariant under the symmetric decreasing rearrangement since it is defined in terms of level sets of $|u|$.}
\end{remark}

We will now prove \Cref{thm: frac Laplacian Intro}. It follows from \Cref{coro constrained minim lagrangian2} and the results of  \Cref{ss energy semilinear frac lap,ss_Calpha_reg}. Since it deals with the fractional Laplacian, let us make the following observation.

\begin{remark}
\label{remark:corrv15 remark 6.2}
{\em  
For both the kernel $K$ of the fractional Laplacian and for $K$ as in \eqref{eq:intro Del kernel}, the function $K(\tau^{1/2})$ is completely monotonic. This has been easily checked right after \Cref{def compl mon rearrangement}. Alternatively, let us check that, for both kernels, the equivalent condition \eqref{kernel Laplace transform}, with $\mu\geq 0$, regarding the Laplace transform, also holds. For this purpose, we apply the Laplace transform to the function $r\mapsto r^{\gamma-1}$, which amounts to the equality
\begin{equation}\label{Kernel frac Lapl compl mon}
   w^{-\gamma}
  =\frac{1}{\Gamma(\gamma)}\int_0^{+\infty}dr\,
  r^{\gamma-1}e^{-wr},
\end{equation}
where $\Gamma(\gamma):=\int_0^{+\infty}dr\,r^{\gamma-1}e^{-r}$ is the Gamma function. The identity \eqref{Kernel frac Lapl compl mon} follows simply by a change of variables $r\mapsto wr$ in the definition of the function $\Gamma$.
To prove the completely monotonic property for both kernels, we use the equivalent condition \eqref{kernel Laplace transform}, which reads as
\begin{equation}
K(t)=\int_0^{+\infty}d\mu(r)\,e^{-t^2r}\quad \text{for all $t>0$,}
\end{equation} 
for some nonnegative measure $\mu$. We conclude the claim from \eqref{Kernel frac Lapl compl mon} applied to, respectively, $w=t^2$ and $\gamma=(1+2s)/2$, and $w=t^2+a^2$ and $\gamma=(n+s)/2$. In particular we obtain that \eqref{kernel Laplace transform} is satisfied for the measures, respectively,
$$
   d\mu(r)=\frac{1}{\Gamma(\frac{1+2s}{2})} r^{s-1/2} dr,\qquad \text{and} \qquad
   d\mu(r)=\frac{1}{\Gamma(\frac{n+s}{2})} r^{(n+s)/2-1} e^{-a^2 r} dr.
$$
}
\end{remark}

\begin{proof}[Proof of \Cref{thm: frac Laplacian Intro}]

The statements $(a)$ and $(c)$ are a special case of \Cref{coro constrained minim lagrangian2}, applied 
with $K(t)=c_s t^{-1-2s}$. This kernel satisfies both $(i)'$ and 
$(ii)'$ in the assumptions of \Cref{coro constrained minim lagrangian2}, as pointed out in \Cref{remark:corrv15 remark 6.2}.

The proof of $(b)$ follows the same line as the one in \Cref{coro constrained minim lagrangian2}. The difference is that $g$ and $\widetilde{g}$ belong now to $C^{\delta}(\R)$ only (instead of $C^{1+\delta}(\R)$), and hence  
$(-\Delta)^su =h(u)$ in $\R$ in the periodic distributional sense for some $h\in C^{\delta}(\R)$.  We can now apply the regularity \Cref{fthm:Holder reg periodic}~$(i)$ or~$(ii)$, depending on the value of $\delta$ (called $\beta$ there). Observe that in both cases $(i)$ and $(ii)$ we deduce that $u\in C^{\alpha}(\R)$ for some $\alpha>2s$. Then, using \eqref{fl eq0_full integration}, we obtain that $(-\Delta)^su =h(u)$ in~$\R$ in the classical sense.  
\end{proof}


\section{$L^\infty(\R)$ estimates for periodic weak solutions}\label{s:Linfty_estimates}

In this section we give an $L^{\infty}$ estimate for periodic weak solutions to subcritical semilinear equations.
More precisely, under standard subcritical-type assumptions on the nonlinearity $f$, we show that every periodic weak solution to the semilinear equation $\LL_K u=f(x,u)$ in $\R$ is bounded. Here, the kernel $K$ is only assumed to be even and satisfy the lower bound \eqref{defi laplacian_K growth2}. 
To state the result, for $s<1/2$, let us denote by
\begin{equation}\label{frac sob exponent}
2_s^*:=\frac{2}{1-2s}=\frac{1+2s}{1-2s}+1
\end{equation}
the fractional Sobolev exponent.

\begin{proposition}\label{L infty estimate}
Let $K$ satisfy the lower bound \eqref{defi laplacian_K growth2} for some $0<s<1$ and some $\lambda>0$. Let $u:\R\to\R$ be a  $2L$-periodic function with
$u\in L^2(-L,L)$ and $[u]_K<+\infty$. 
\begin{itemize}
\item[$(i)$] Assume that $1/2<s<1$. Then, 
$\|u\|_{L^\infty(\R)}\leq C(\|u\|_{L^2(-L,L)}+[u]_K)<+\infty$ for some constant $C$ depending only on $L$, $s$, and 
$\lambda$.
\item[$(ii)$] Assume that $0<s\leq1/2$ and that $u$ is a  $2L$-periodic weak solution to
\begin{equation}\label{eq: L infty estimate eq1}
\LL_K u=f(x,u)\quad\text{in $\R$},
\end{equation}
where $f:\R^2\to\R$ is $2L$-periodic in the first variable and  satisfies
\begin{equation}\label{eq: growth estimate regularity}
|f(x,t)|\leq C_0(1+|t|^p)\quad\text{for all $(x,t)\in(-L,L)\times\R$,}
\end{equation}
for some constant $C_0$, some exponent $1\leq p<2_s^*-1$ if $0<s<1/2$, and some $1\leq p<+\infty$ if $s=1/2$. 
Then, $\|u\|_{L^\infty(\R)}\leq C$ for some constant $C$ depending only on $L$, $s$, $\lambda$, $C_0$, $p$, $\|u\|_{L^2(-L,L)}$, and $[u]_K$.
\end{itemize}
\end{proposition}

The proof of \Cref{L infty estimate} $(ii)$ will follow the classical argument of Brezis and Kato \cite{BK}; see also \cite[B.3 Lemma]{Struwe}. This argument has been already adapted to the fractional framework; see, possibly among others, \cite[Proposition 5.1.1]{ValBook} for a nonperiodic semilinear problem and \cite[Lemma 2.3]{BGQ} for a periodic linear problem, both treating only the operator $(-\Delta)^s$. 
We carry this out here for the operator $\LL_K$ putting, in addition, special care on tracking what the constant~$C$ in \Cref{L infty estimate} depends on.

To carry out the proof, we will use the  auxiliary \Cref{appendix l1}, which plays the role of the chain rule in the local case. More precisely, in the case of the Laplacian, one starts by testing the equation against
$u^{2r+1}$. Then, from the chain rule, one uses
\begin{equation}
\frac{2r+1}{(r+1)^2}|\nabla(u^{r+1})|^2
=(2r+1)u^{2r}|\nabla u|^2
=\nabla u\cdot\nabla(u^{2r+1}).
\end{equation}
It is precisely to adapt this step to the nonlocal operator $\LL_K$ where we make use of the following simple lemma. In its statement, the left hand side of the inequality plays the role of $|\nabla(u^{r+1})|^2$ in the local framework, while the right hand side plays the role of $\nabla u\cdot\nabla(u^{2r+1})$ ---this later is the one which links \eqref{eq: L infty estimate eq1} to its weak formulation.

\begin{lemma}\label{appendix l1}
For every $a$ and $b$ in $\R$, and every $M\geq 0$ and $r\geq0$, it holds
\begin{equation}
\begin{split}
\big(a\min\{|a|,M\}^r-b\min\{|b|,M\}^r\big)^2
\leq 2(r+2)(a-b)\big(a\min\{|a|,M\}^{2r}-b\min\{|b|,M\}^{2r}\big).
\end{split}
\end{equation}
\end{lemma}

\begin{proof}
Assume first that $ab\leq0$. Then,
\begin{equation}
\begin{split}
&\big(a\min\{|a|,M\}^r-b\min\{|b|,M\}^r\big)^2\\
&\qquad=a^2\min\{|a|,M\}^{2r}+b^2\min\{|b|,M\}^{2r}
-2ab\min\{|a|,M\}^r\min\{|b|,M\}^r\\
&\qquad\leq
a^2\min\{|a|,M\}^{2r}+b^2\min\{|b|,M\}^{2r}
-ab\big(\min\{|a|,M\}^{2r}+\min\{|b|,M\}^{2r}\big)\\
&\qquad=(a-b)\big(a\min\{|a|,M\}^{2r}-b\min\{|b|,M\}^{2r}\big),
\end{split}
\end{equation}
and the lemma follows in this case. 

Therefore, from now on we can assume that $ab>0$. Note also that the case where both $a$ and $b$ are negative numbers follows easily once we know the inequality for all positive $a$ and $b$. In addition, if $a=b$ the statement is clear, and hence we can assume $a\neq b$ and, by symmetry, we can further assume that $$0<b<a.$$ 

At this point, we split the problem in three possible situations, as follows.

\medskip
\noindent\underline{{\em Case $M<b<a$}\,:} 
Then,
\begin{equation}
\begin{split}
\big(a\min\{|a|,M\}^r-b\min\{|b|,M\}^r\big)^2
&=(M^r a-M^r b)^2=(a-b)(M^{2r}a-M^{2r}b)\\
&=(a-b)\big(a\min\{|a|,M\}^{2r}-b\min\{|b|,M\}^{2r}\big).
\end{split}
\end{equation}

\medskip
\noindent\underline{{\em Case $0<b<a\leq M$}\,:}
By the Cauchy-Schwarz inequality we see that
\begin{equation}\label{appendix eq1}
\begin{split}
\big(a\min\{|a|,M\}^r&-b\min\{|b|,M\}^r\big)^2
=(a^{r+1}-b^{r+1})^2
=\Big(\int_b^adt\,(r+1)t^r\Big)^2\\
&\leq(r+1)^2(a-b)\int_b^adt\,t^{2r}
=\frac{(r+1)^2}{2r+1}(a-b)(a^{2r+1}-b^{2r+1})\\
&\leq (r+1)(a-b)\big(a\min\{|a|,M\}^{2r}-b\min\{|b|,M\}^{2r}\big).
\end{split}
\end{equation}

\medskip
\noindent\underline{{\em Case $0<b\leq M<a$}\,:}
Then,
\begin{equation}\label{appendix eq2}
\begin{split}
\big(a\min\{|a|,M\}^r&-b\min\{|b|,M\}^r\big)^2
=(M^r a-b^{r+1})^2\\
&\leq 2\Big(M^r a-(M^{2r}a)^{\frac{r+1}{2r+1}}\Big)^2
+2\Big((M^{2r}a)^{\frac{r+1}{2r+1}}-b^{r+1}\Big)^2.
\end{split}
\end{equation}
Now, on the one hand, note that $M^r a\geq(M^{2r}a)^{\frac{r+1}{2r+1}}\geq M^{r+1}$. Thus
\begin{equation}\label{appendix eq3}
\begin{split}
\Big(M^r a-(M^{2r}a)^{\frac{r+1}{2r+1}}\Big)^2
&\leq(M^r a-M^{r+1})^2=(a-M)(M^{2r}a-M^{2r+1})\\
&\leq(a-b)\big(a\min\{|a|,M\}^{2r}-b\min\{|b|,M\}^{2r}\big).
\end{split}
\end{equation}
On the other hand, by the Cauchy-Schwarz inequality and using that $M<a$, we get
\begin{equation}
\begin{split}
\Big((&M^{2r}a)^{\frac{r+1}{2r+1}}-b^{r+1}\Big)^2
=\Big(\int_b^{(M^{2r}a)^{\frac{1}{2r+1}}}\!\!\!dt\,(r+1)t^r\Big)^2\\
&\leq(r+1)^2\big((M^{2r}a)^{\frac{1}{2r+1}}-b\big)\int_b^{(M^{2r}a)^{\frac{1}{2r+1}}}\!\!\!dt\,t^{2r}
=\frac{(r+1)^2}{2r+1}\big((M^{2r}a)^{\frac{1}{2r+1}}-b\big)(M^{2r}a-b^{2r+1})\\
&\leq (r+1)(a-b)\big(a\min\{|a|,M\}^{2r}-b\min\{|b|,M\}^{2r}\big).
\end{split}
\end{equation}
Combining this with \eqref{appendix eq2} and \eqref{appendix eq3} we see that the lemma also holds for $0<b\leq M<a$.
\end{proof}

We now prove the $L^{\infty}$ estimate for periodic weak solutions.

\begin{proof}[Proof of \Cref{L infty estimate}]
In the following, we denote
\begin{equation}
\|u\|_{W^{s,2}(-L,L)}:=\Big(\intL |u|^2+\intL\!dx\intL\!dy\,
\frac{|u(x)-u(y)|^2}{|x-y|^{1+2s}}\Big)^{1/2}.
\end{equation}
Clearly, the lower bound \eqref{defi laplacian_K growth2} gives
\begin{equation}\label{easy control sobolev Hs}
\|u\|_{W^{s,2}(-L,L)}^2\leq \|u\|_{L^2(-L,L)}^2+\frac{2}{\lambda}[u]_K^2<+\infty.
\end{equation}

The proof of $(i)$ is straightforward. Since $1/2<s<1$, \cite[Theorem 8.2]{Valdinoci1} and \eqref{easy control sobolev Hs} yield
\begin{equation}
\|u\|_{L^\infty(-L,L)}\leq C\|u\|_{W^{s,2}(-L,L)}
\leq C\big(\|u\|_{L^2(-L,L)}+\sqrt{2\lambda^{-1}}[u]_K\big)
\end{equation}
for some constant $C$ depending only on $L$ and $s$.

We now address the proof of $(ii)$ for $0<s<1/2$. Given nonnegative constants $M$ and $r$, and $t\in\R$, set
\begin{equation}
  \varphi_{r,M}(t):=t\min\{|t|,M\}^r.
\end{equation}
It is not hard to show that $\varphi_{r,M}(u)$ is a 
$2L$-periodic function with $\varphi_{r,M}(u)\in L^2(-L,L)$ and $[\varphi_{r,M}(u)]_K<+\infty$ for all $r\geq0$ and $M\geq0$. 
Using \Cref{appendix l1},
taking $\varphi_{2r,M}(u)$ as a test function in the weak formulation of \eqref{eq: L infty estimate eq1} given by \eqref{truly weak form eq} with $h=f(x,u)$, and using the subcritical growth \eqref{eq: growth estimate regularity}, we see that
\begin{equation}\label{appendix eq5}
\begin{split}
[\varphi_{r,M}(u)]_K^2
&\leq 2(r+2)\langle u, \varphi_{2r,M}(u)\rangle_K
=2(r+2)\intL f(x,u)\, \varphi_{2r,M}(u)\\
&\leq C_1(r+1)\intL (1+|u|^p) |\varphi_{2r,M}(u)|,
\end{split}
\end{equation}
where $C_1=4C_0$.
Adding $\|\varphi_{r,M}(u)\|_{L^2(-L,L)}^2$ on both sides of \eqref{appendix eq5}, using \eqref{easy control sobolev Hs},  applying the fractional Sobolev inequality \cite[Theorem 6.7]{Valdinoci1} to the left-hand side, and using that 
\begin{equation}\label{eq: appendix eq6_ident_aux}
|u|^p|\varphi_{2r,M}(u)|=|u|^{p-1}\varphi_{r,M}(u)^2
\end{equation} 
on the right-hand side, we deduce 
\begin{equation}\label{eq: appendix eq6}
\Big(\intL |\varphi_{r,M}(u)|^{2_s^*}\Big)^{2/2_s^*}
\leq C_2(r+1)\Big(\intL|\varphi_{2r,M}(u)|+\intL 
|u|^{p-1}\varphi_{r,M}(u)^2+\intL \varphi_{r,M}(u)^2\Big)
\end{equation}
for some constant $C_2$ depending only on $L$, $\lambda$, $s$, and $C_0$.

Since we are assuming that
$p<2_s^*-1$, for every $\epsilon>0$ (to be chosen later on) there exists a constant $C_\epsilon>0$ (depending only on
$\epsilon$, $s$, and $p$) such that
\begin{equation}
t^{p-1}+1\leq\epsilon t^{2_s^*-2}+\frac{C_\epsilon}{t}\quad\text{for all }t>0.
\end{equation}
In particular, from \eqref{eq: appendix eq6}, the definition of $\varphi_{r,M}$, and \eqref{eq: appendix eq6_ident_aux}, we deduce that
\begin{equation}\label{eq: appendix eq6'}
 \begin{split}
   \Big(\intL |\varphi_{r,M}(u)|^{2_s^*}\Big)^{2/2_s^*}
   &\leq 
    C_2(r+1)\Big((1+C_\epsilon)\intL|\varphi_{2r,M}(u)|
   +\epsilon \intL |u|^{2_s^*-2}\varphi_{r,M}(u)^2\Big).
 \end{split}
\end{equation}
Now, by H\"older's inequality,
\begin{equation}\label{eq: appendix eq7}
   \intL |u|^{2_s^*-2}\varphi_{r,M}(u)^2
   \leq
   \Big(\intL |u|^{2_s^*}\Big)^{\frac{2_s^*-2}{2_s^*}}
   \Big(\intL |\varphi_{r,M}(u)|^{2_s^*}\Big)^{2/2_s^*},
\end{equation}
and thus \eqref{eq: appendix eq6'} finally yields
\begin{equation}\label{eq: appendix eq7'}
\begin{split}
\Big(\intL |\varphi_{r,M}(u)|^{2_s^*}\Big)^{2/2_s^*}
&\leq C_2(r+1)(1+C_\epsilon)\intL|\varphi_{2r,M}(u)|\\
&\quad+C_2(r+1)\epsilon \|u\|_{L^{2_s^*}(-L,L)}^{2_s^*-2}
\Big(\intL |\varphi_{r,M}(u)|^{2_s^*}\Big)^{2/2_s^*}.
\end{split}
\end{equation}
Recall that $\|u\|_{L^2(-L,L)}$ and $[u]_K$ are finite. Therefore, the fractional Sobolev inequality and \eqref{easy control sobolev Hs} show that
$\|u\|_{L^{2_s^*}(-L,L)}\leq C\|u\|_{W^{s,2}(-L,L)}<+\infty$. 

Let us now choose $r_1>0$ such that $2r_1+1=2_s^*$, and also choose $\epsilon>0$ small enough depending only on $L$, $\lambda$, $C_0$, $s$, and $\|u\|_{W^{s,2}(-L,L)}$, such that
\begin{equation}\label{eq: appendix eq7''}
C_2(r_1+1)\epsilon \|u\|_{L^{2_s^*}(-L,L)}^{2_s^*-2}\leq\frac{1}{2}.
\end{equation}
Then,
from \eqref{eq: appendix eq7'} we get that
\begin{equation}\label{eq: appendix eq7''bisi}
\begin{split}
\Big(\intL |\varphi_{r_1,M}(u)|^{2_s^*}\Big)^{2/2_s^*}
&\leq C_{3}\intL|\varphi_{2r_1,M}(u)|
\end{split}
\end{equation}
for some constant $C_3$ depending only on $L$, $\lambda$, $C_0$, $s$, $p$, and $\|u\|_{W^{s,2}(-L,L)}$. Now, letting $M\uparrow+\infty$ and using monotone convergence in \eqref{eq: appendix eq7''bisi}, in view of the definition of $r_1$ we conclude that
\begin{equation}\label{eq: appendix eq11}
\begin{split}
\||u|^{r_1+1}\|_{L^{2_s^*}(-L,L)}^{2}
=\Big(\intL |u|^{2_s^*(r_1+1)}\Big)^{2/2_s^*}
&\leq C_{3}\intL |u|^{2r_1+1}=C_3\|u\|_{L^{2_s^*}(-L,L)}^{2_s^*}\leq C_4
\end{split}
\end{equation}
for some constant $C_4$ depending only on $L$, $\lambda$, $C_0$, $s$, $p$, and $\|u\|_{W^{s,2}(-L,L)}$.

We now come back to \eqref{eq: appendix eq6'}.  We stress that  $\epsilon$ has already been chosen according to \eqref{eq: appendix eq7''}. By letting $M\uparrow+\infty$ in \eqref{eq: appendix eq6'} and using monotone convergence, we deduce that
\begin{equation}\label{eq: appendix eq8}
\Big(\intL |u|^{2_s^*(r+1)}\Big)^{2/2_s^*}
\leq C(r+1)\Big(\intL |u|^{2r+1}
+\intL |u|^{2r+2_s^*}\Big)
\end{equation}
for all $r\geq0$, where the constant $C$ depends only on $L$, $\lambda$, $C_0$, $s$, $p$, and $\|u\|_{W^{s,2}(-L,L)}$. Since $2_s^*>1$, we clearly have
\begin{equation}\label{eq: appendix eq9}
t^{2r+1}\leq 1+t^{2r+2_s^*}\quad\text{for all }t\geq0.
\end{equation}
From \eqref{eq: appendix eq8} and \eqref{eq: appendix eq9}
we get that
\begin{equation}\label{eq: appendix eq10}
\Big(1+\intL |u|^{2_s^*(r+1)}\Big)^{\frac{1}{2_s^*r}}
\leq \big(C(r+1)\big)^{\frac{1}{2r}}\Big(1+\intL |u|^{2r+2_s^*}\Big)^{\frac{1}{2r}},
\end{equation}
where $C$ depends only on $L$, $\lambda$, $C_0$, $s$, $p$, and $\|u\|_{W^{s,2}(-L,L)}$. 

We now define
\begin{equation}
r_{m+1}:=\frac{2_s^*}{2}\,r_m=\Big(\frac{2_s^*}{2}\Big)^m r_1
\quad\text{for all }m\geq1
\end{equation}
($r_1$ has already been defined by the relation $2r_1+1=2_s^*$).
Clearly, $2r_{m+1}+2_s^*=2_s^*(r_m+1)$, and $r_m\uparrow+\infty$ as $m\uparrow+\infty$. We also set
\begin{equation}
A_m:=\Big(1+\intL |u|^{2_s^*(r_m+1)}\Big)^{\frac{1}{2_s^*r_m}}
\quad\text{for all }m\geq1.
\end{equation}
Combining \eqref{eq: appendix eq10} for $r=r_{m+1}$ with
\eqref{eq: appendix eq11} we deduce that
\begin{equation}
\begin{split}
A_{m+1}&\leq \Big(C\Big(\frac{2_s^*}{2}\Big)^m r_1+C\Big)
^{(\frac{2}{2_s^*})^m\frac{1}{2r_1}}A_m
\leq\cdots\leq A_1\prod_{j=1}^m
\Big(C\Big(\frac{2_s^*}{2}\Big)^j r_1+C\Big)
^{(\frac{2}{2_s^*})^j\frac{1}{2r_1}}\\
&=A_1\exp\Big\{\frac{1}{2r_1}\sum_{j=1}^m
{\Big(\frac{2}{2_s^*}\Big)^j}
\log\Big(C\Big(\frac{2_s^*}{2}\Big)^j r_1+C\Big)\Big\}
\leq C_5
\end{split}
\end{equation}
for some constant $C_5\geq1$  depending only on $L$, $\lambda$, $C_0$, $s$, $p$, and $\|u\|_{W^{s,2}(-L,L)}$. Here we also used that $2_s^*>2$ to see that the sum above is uniformly bounded in $m\geq1$. Therefore,
\begin{equation}
\|u\|_{L^{2_s^*(r_m+1)}(-L,L)}\leq C_5^{\frac{r_m}{r_m+1}}\leq C_5
\end{equation}
for all $m\geq1$. Since this bound is uniform in $m$, letting $m\uparrow+\infty$ (and, thus, $r_m\uparrow+\infty$) we get $\|u\|_{L^{\infty}(-L,L)}\leq C_5$. By periodicity, $\|u\|_{L^{\infty}(\R)}\leq C_5<+\infty$, as desired.

Let us finally address the proof of $(ii)$ for $s=1/2$. In this case we are assuming \eqref{eq: growth estimate regularity} for some $1\leq p<+\infty$. Since $p$ is finite, we can take $0<s'<1/2$ such that $p<\frac{1+2s'}{1-2s'}$, which yields $p<2^*_{s'}-1$ by \eqref{frac sob exponent}. The idea now is to carry out the proof made for the case $0<s<1/2$ but replacing $s$ by $s'$ everywhere. To do so, we only need to bound 
$\|u\|_{W^{s',2}(-L,L)}$ by $\|u\|_{L^2(-L,L)}$ and $[u]_K$. Note that \eqref{defi laplacian_K growth2} for $s=1/2$ gives $K(t)\geq \lambda (2L)^{2s'-1}t^{-1-2s'}$ for all $0<t\leq2L$. Hence,
\begin{equation}\label{easy control sobolev Hs 2}
\|u\|_{W^{s',2}(-L,L)}^2\leq \|u\|_{L^2(-L,L)}^2+\frac{1}{\lambda}(2L)^{1-2s'}[u]_K^2,
\end{equation}
as desired.
Now, going back to the beginning of the proof, from the weak formulation of \eqref{eq: L infty estimate eq1} we obtain \eqref{appendix eq5}. This is the only point where we used that $u$ is a periodic weak solution. This means that, from \eqref{frac sob exponent} up to the end of the proof, we simply need to replace $s$ by $s'$, and \eqref{easy control sobolev Hs} by \eqref{easy control sobolev Hs 2}, everywhere to conclude 
$\|u\|_{L^{\infty}(\R)}<+\infty$.
\end{proof}

\begin{remark} {\em Let $0<s<1/2.$
Consider a sequence of semilinear equations of the form
\begin{equation}\label{bdd.critical.not_unif}
\LL_{K_j}u_j=f_j(x,u_j)\quad\text{in $\R$,}
\end{equation}
where $f_j$ is $2L$-periodic in the first variable.
From \Cref{L infty estimate} we see that, in the subcritical case $1\leq p<2_s^*-1$,
if 
\begin{itemize}
\item $|f_j(x,t)|\leq C_0(1+|t|^p)$ for all $j$ and all $(x,t)\in(-L,L)\times[0,+\infty)$, 
\item $K_j$ satisfies \eqref{defi laplacian_K growth2} uniformly in $j$, and 
\item $u_j$ is a periodic weak solution to 
$\LL_{K_j}u_j=f_j(x,u_j)$ in $\R$ such that 
$\|u_j\|_{L^2(-L,L)}$ and $[u_j]_{K_j}$ are uniformly bounded in $j$ (hence, $\|u_j\|_{W^{s,2}(-L,L)}$ too),
\end{itemize} 
then $\|u_j\|_{L^\infty(\R)}$ are uniformly bounded in $j$. 
We took a special care to track the dependence of the constants appearing in the proof of 
\Cref{L infty estimate}~$(ii)$ precisely to get this type of uniform estimates. It can be useful, for instance, when one uses a monotone iteration argument to construct solutions to semilinear integro-differential equations.}
\end{remark}


\section{$C^\alpha(\R)$ estimates for bounded periodic distributional solutions}
\label{ss_Calpha_reg}

The aim of this section is to prove  \Cref{fthm:Holder reg periodic}. 

Before providing the proof, let us first comment on the optimality of the H\"older exponents appearing in \Cref{fthm:Holder reg periodic}.
The~$\epsilon$ appearing in the statement of \Cref{fthm:Holder reg periodic} $(ii)$ is due to the fact that $\beta+2s$ may be an integer. In particular, if $\beta+2s$ is not an integer, we indeed get $u\in C^{\beta+2s}(\R)$. On the contrary, the~$\epsilon$ appearing in the statement of \Cref{fthm:Holder reg periodic} $(i)$ has a different nature. It comes from the fact that the iteration method of the proof yields 
$u\in C^{2s\beta_{k}}(\R)$ for all $k\geq 1$, where 
$\beta_k$ is a sequence of exponents such that 
$\beta_k\uparrow1/(1-\beta)$ but $\beta_k<1/(1-\beta)$ for all $k$.

Currently we do not know whether \Cref{fthm:Holder reg periodic} $(i)$ is sharp or if it holds true with $\epsilon=0$ as well. But what we know for sure is that the exponent $2s/(1-\beta)$ cannot be improved, as shown in our forthcoming work \cite{Csato Mas Holder opt}.

We have seen that
\Cref{prop:gy:weak per and nonper sol} allows us to apply the known regularity results for distributional Dirichlet solutions to the class of periodic distributional (and in particular also to weak) solutions. 
Thus, let us recall the following regularity result of \cite{DRSV}. Be aware that what we call here distributional solution is called weak solution in \cite{DRSV}.

\begin{theorem} $($\cite[Theorem 3.8]{DRSV}$)$\label{fl Lemma bis}
Let $K$ satisfy \eqref{defi laplacian_K growth1} and \eqref{defi laplacian_K growth2}, and let $u\in L^\infty(\R)$ be a distributional (Dirichlet)
solution to $\LL_K u=h$ in $I=(-L,L)$, i.e., a solution satisfying \eqref{weak dirichlet sol defi} for all $\varphi\in C_c^{\infty}(I)$. 

Then, the following holds:
\begin{itemize}
\item[$(i)$] Assume that $h\in L^\infty(I)$. If $s\neq 1/2$, then $u\in C^{2s}(\overline I/2)$ and
\begin{equation}
\|u\|_{C^{2s}(\overline I/2)}
\leq C\big(\|u\|_{L^\infty(\R)}+\|h\|_{L^\infty(I)}\big).
\end{equation}
If $s=1/2$, then for every $\epsilon>0$, $u\in C^{2s-\epsilon}(\overline I/2)$ and
\begin{equation}
\|u\|_{C^{2s-\epsilon}(\overline I/2)}
\leq C\big(\|u\|_{L^\infty(\R)}+\|h\|_{L^\infty(I)}\big).
\end{equation}

The constants $C$ depend only on $L$, $s$, 
$\lambda$, and $\Lambda$ (and also on $\epsilon$ in the case $s=1/2$).
\item[$(ii)$] 
Assume that $h\in C^{\alpha}(\overline I)$ and $u\in C^\alpha(\R)$ for some $\alpha>0$. If $\alpha+2s$ is not an integer, then $u\in C^{\alpha+2s}(\overline I/2)$ and
\begin{equation}
\|u\|_{C^{\alpha+2s}(\overline I/2)}
\leq C\big(\|u\|_{C^{\alpha}(\R)}+\|h\|_{C^{\alpha}(\overline I)}\big).
\end{equation}
If $\alpha+2s$ is an integer, then for every $\epsilon>0$, $u\in C^{\alpha+2s-\epsilon}(\overline I/2)$ and
\begin{equation}
\|u\|_{C^{\alpha+2s-\epsilon}(\overline I/2)}
\leq C\big(\|u\|_{C^{\alpha}(\R)}+\|h\|_{C^{\alpha}(\overline I)}\big).
\end{equation}

The constants $C$ depend only on $L$, $s$, $\alpha$, 
$\lambda$, and $\Lambda$ (and also on $\epsilon$ in the case that $\alpha+2s$ is an integer).
\end{itemize}
\end{theorem}

The last case of $(ii)$ in \Cref{fl Lemma bis} ($\alpha+2s$ is an integer) is not stated explicitly in \cite[Theorem 3.8]{DRSV}, but it follows  from the first case of $(ii)$ by replacing $\alpha$ by $\alpha-\epsilon$.

As we showed in \Cref{prop:gy:weak per and nonper sol}, if we assume that $h$ is $2L$-periodic, the conclusions of the \Cref{fl Lemma bis} also hold for $2L$-periodic distributional solutions to $\LL_K u=h$ in $\R$, replacing $I$ (and $I/2$) by $\R$ in $(i)$ and $(ii)$. More precisely, since $\LL_Ku=h$ in the whole real line and $u$ is $2L$-periodic, one can apply the theorem to the sets $I/2+a\subset I+a$ for different values of $a\in\R$ (to cover a neighborhood of $\overline I$) and obtain an estimate on the H\"older norm of $u$ in all of~$\R$. Thus, we can use \Cref{fl Lemma bis} to prove  \Cref{fthm:Holder reg periodic} on  periodic solutions.

\begin{proof}[Proof of \Cref{fthm:Holder reg periodic}.]
We will prove the theorem by distinguishing several cases, but we will treat all of them essentially using a same bootstrap argument based on \Cref{fl Lemma bis}. 
We will also make use of the following simple facts regarding the composition of H\"older continuous functions.

Let $f\in C^{\beta}(\R)$ and $v\in C^{\sigma}(\R)$ be a $2L$-periodic function. Then, there exists some constant $C_{\beta,\sigma}>0$ depending only on 
$\|f\|_{C^{\beta}(\R)}$ and $\|v\|_{C^{\sigma}(\R)}$ such that the following holds:
\begin{align}
&\text{If $0\leq \beta,\sigma<1$, then $f(v)\in C^{\beta\sigma}(\R)$ and 
$\|f(v)\|_{C^{\beta\sigma}(\R)}\leq C_{\beta,\sigma}$.}\label{gyula:H1}\\
&\text{If $\max\{\beta,\sigma\}\geq 1$, then $f(v)\in C^{\min\{\beta,\sigma\}}(\R)$ and $\|f(v)\|_{C^{\min\{\beta,\sigma\}}(\R)}\leq C_{\beta,\sigma}$.}\label{gyula:H2}
\end{align}
The proof of these statements follows, for instance, from 
\cite[Theorem 16.31]{CDK} and the periodicity of $v$; \cite[Theorem 16.31]{CDK} does not cover the case $\beta\geq 1$ and $\sigma<1,$ but this case follows very easily from the fact that $f$ is Lipschitz. As already mentioned right above this proof, the periodicity allows to extend the estimates of the H\"older norm of the composition $f(v)$ from bounded intervals to the whole real line.

Assume $s\neq1/2$ (the case $s=1/2$ will be explained at the end of the proof). Since $f(u)\in L^\infty(\R)$, an application of \Cref{fl Lemma bis} $(i)$ gives
\begin{equation}
\|u\|_{C^{2s}(\overline I/2)}
\leq C\big(\|u\|_{L^\infty(\R)}+\|f(u)\|_{L^\infty(I)}\big),
\end{equation}
where $I=(-L,L)$.
Since $u$ is $2L$-periodic, this argument applies to all translates of $u$ (since they are also periodic solutions), and thus we actually deduce that 
\begin{equation}\label{reg estimate SR f(u)}
\|u\|_{C^{2s}(\R)}
\leq C\big(\|u\|_{L^\infty(\R)}+\|f(u)\|_{L^\infty(\R)}\big).
\end{equation}

To simplify the exposition, we first assume in the second and fourth case below 
that $\beta$ is not an integer, and that all the H\"older exponents appearing in the forthcoming arguments never coincide with an integer when using \Cref{fl Lemma bis}. Then, at the end of the proof we will explain what should be modified in case that some H\"older exponent eventually hits an integer when using \Cref{fl Lemma bis}.

\medskip
\noindent\underline{{\em Case $2s\leq1-\beta$}\,:} 
\smallskip

Since $s$ and $\beta$ are positive, we infer from $2s\leq1-\beta$ that $0<\beta<1$ and $0<s<1/2$. Set 
\begin{equation}\label{beta_k iteration regularity}
\beta_k:=\sum_{j=0}^k \beta^j\quad \text{for $k=1,2,3,\ldots$}
\end{equation} 
Note that $\beta_k\uparrow1/(1-\beta)$ as $k\uparrow+\infty$. 
Since we are assuming $2s/(1-\beta)\leq1$, we have $2s\beta_k<1$ for all $k\geq1$. 

From \eqref{reg estimate SR f(u)} and the fact that $\beta<1$ we have $u\in C^{2s}(\R)\subset C^{2s\beta}(\R)$. Since 
$2s<1$ and $\beta<1$, \eqref{gyula:H1} yields 
$f(u)\in C^{2s\beta}(\R)$. Therefore, \Cref{fl Lemma bis}~$(ii)$ gives $u\in C^{2s(\beta+1)}(\R)=C^{2s\beta_1}(\R)$.
Since $2s\beta_1<1$, from \eqref{gyula:H1} once again we obtain 
$f(u)\in C^{2s\beta_1\beta}(\R)$. 
Also, since $\beta<1$, we have $u\in C^{2s\beta_1\beta}(\R)$. 
Applying \Cref{fl Lemma bis} $(ii)$  we get 
$u\in C^{2s\beta_2}(\R)$ with $\beta_2=\beta_1\beta+1
=\beta^2+\beta+1$. Iterating this argument, and recalling that $2s\beta_k<1$ for all $k\geq1$, we see that
$u\in C^{2s\beta_{k}}(\R)$ for all $k\geq 1$.  
Using that $\beta_k\uparrow1/(1-\beta)$, this yields 
$u\in C^{\frac{2s}{1-\beta}-\epsilon}(\R)$ for all $\epsilon>0$, as desired. 

\medskip

\noindent\underline{{\em Case $2s>1-\beta$, with $0<\beta<1$, $0<s<1/2$}\,:} 
\smallskip

Arguing as in the case $2s\leq1-\beta$, we see that 
$u\in C^{2s\beta_{k}}(\R)$ whenever $2s\beta_{k-1}<1$, with the understanding that $\beta_0:=\beta$. Note that $2s\beta_0<1$ and hence we can start the iteration. Recall that $\beta_k\uparrow1/(1-\beta)$ and, since we are now assuming that $2s>1-\beta$, we have $2s/(1-\beta)>1$. Hence, there exists a unique $k_0\geq1$ such that $2s\beta_{k_0}\geq1$ and $2s\beta_{k_0-1}<1$. Now, since $u\in C^{2s\beta_{k_0}}(\R)$ and $2s\beta_{k_0}\geq1$, we have $u\in C^{1}(\R)$. Thus, $f(u)\in C^{\beta}$ by \eqref{gyula:H2}.  Applying \Cref{fl Lemma bis} $(ii)$ we deduce that $u\in C^{\beta+2s-\epsilon}(\R)$ for all $\epsilon>0$, as desired. 
   
\medskip

\noindent\underline{{\em Case $2s>1-\beta$, with $0<\beta<1$, $1/2<s<1$}\,:} 
\smallskip

Since $\beta<1<2s$, from \eqref{reg estimate SR f(u)} we get 
$u\in C^{1}(\R)$. Thus,
$f(u)\in C^{\beta}(\R)$ by \eqref{gyula:H2}. Applying \Cref{fl Lemma bis} $(ii)$ we end up with $u\in  C^{\beta+2s-\epsilon}(\R)$ for all $\epsilon>0$.

\medskip

\noindent\underline{{\em Case $\beta>1$ (and hence $2s>1-\beta$) and $s\neq1/2$}\,:} 
\smallskip

Let $j_0\geq1$ be an integer such that $2s(j_0-1)<\beta\leq2sj_0$. From \Cref{fl Lemma bis} with $\alpha=2sj$ and the periodicity of $u$ we see that 
\begin{equation}\label{reg estimate SR f(u)2}
\|u\|_{C^{2sj}(\R)}
\leq C\big(\|u\|_{C^{2s(j-1)}(\R)}+\|f(u)\|_{C^{2s(j-1)}(\R)}\big)
\end{equation}
whenever the right-hand side is finite ---recall that we are assuming that $2sj\notin\Z$. Here, \eqref{reg estimate SR f(u)2} for $j=1$ must be interpreted as \eqref{reg estimate SR f(u)}. Observe that if 
$u\in C^{2sj}(\R)$ and $\beta\geq1$, then $f(u)\in C^{\min\{\beta,2sj\}}(\R)$ by \eqref{gyula:H2}. Therefore, \eqref{reg estimate SR f(u)2} can be iterated for $j=1,\ldots,j_0$ to get 
$u\in C^{2sj_0}(\R)\subset C^\beta(\R)$, which yields $f(u)\in C^\beta(\R)$ by \eqref{gyula:H2}. Now, \Cref{fl Lemma bis} leads to $u\in C^{\beta+2s-\epsilon}(\R)$ for all $\epsilon>0$, as desired.
\medskip

\noindent\underline{{\em Case $s=1/2$, or $\beta$ is an integer, or if we hit an integer (in the 2nd and 4th case above)\,:}}
\smallskip

Assume first that $s\neq1/2$. During the iteration processes explained in the previous cases, if we know that $u$ and $f(u)$ belong to $C^{\alpha}(\R)$ for some $\alpha<\beta$ and it turns out that $\alpha+2s$ hits an integer 
(clearly, this situation cannot happen in the first and third case), we simply diminish a bit the exponent $\alpha$ in order to avoid that integer and we continue the iteration.

Finally, the case $s=1/2$ is treated analogously but, instead of \eqref{reg estimate SR f(u)}, using that 
\begin{equation}
\|u\|_{C^{2s-\epsilon}(\R)}
\leq C\big(\|u\|_{L^\infty(\R)}+\|f(u)\|_{L^\infty(\R)}\big)
\end{equation}
for any $\epsilon>0$ arbitrarily small, which follows from  \Cref{fl Lemma bis} $(i)$ and periodicity.
\end{proof}

\begin{remark}
\label{gy:remark:nonperiodic reg Holder}{\em 
With the same iteration method used here one can show that the analogue of \Cref{fthm:Holder reg periodic} in the nonperiodic scenario also holds for regular enough kernels $K$ satisfying \eqref{defi laplacian_K growth1} and \eqref{defi laplacian_K growth2}.  Note that from the first iteration one would only get a bound on the H\"older regularity of $u$ in 
$\overline I/2$, but not in the whole real line. Hence, one cannot use \Cref{fl Lemma bis} $(ii)$ for the next iteration since its right-hand side involves $\|u\|_{C^\alpha(\R)}$. However, for sufficiently regular kernels there is an improvement of \Cref{fl Lemma bis} $(ii).$ More precisely, if $[K]_{C^{\alpha}(\R^n\setminus B_r)}\leq \Lambda r^{-n-2s-\alpha}$ for all $r>0$ then
$$
  \|u\|_{C^{\alpha+2s}(\overline I/2)}
  \leq C\big(\|u\|_{L^{\infty}(\R)}+\|h\|_{C^{\alpha}(\overline I)}\big)
  \quad\text{if $    \alpha+2s$ is not an integer};
$$
see \cite[Theorem 3.8 $(c)$]{DRSV}. 

Therefore, for sufficiently regular kernels $K$ one can obtain the following result: 
{\em given $f\in C^\beta(\R)$ for some $\beta>0$, every $u\in L^\infty(\R)$ distributional solution to $\LL_K u=f(u)$ in $I=(-L,L)$ satisfies
\begin{itemize}
\item[$(i)$] $u\in C^{\frac{2s}{1-\beta}-\epsilon}(\overline I/2)$ for all $\epsilon>0$ if $2s<1-\beta$, 
\item[$(ii)$] $u\in C^{\beta+2s-\epsilon}(\overline I/2)$ for all $\epsilon>0$ if $2s\geq 1-\beta$.
\end{itemize}}
In the periodic scenario this difficulty does not appear thanks to the periodicity of $u$, even for nonregular kernels.
}
\end{remark}


\appendix


\section{The half-Laplacian on $\Sph^1$}
\label{section:appendix half.harm}

We shall now explain the relation between the integro-differential expressions for the half-Laplacian on $\R$ for periodic functions and the half-Laplacian on $\Sph^1$. 
Any $2\pi$-periodic function~$u$ on the real line can be identified with a function $u_{\Sph^1}$ in the unit circle by $u_{\Sph^1}(e^{ix})=u(x)$ for all 
$x\in\R$. With this at hand, we will see below that 
$(-\Delta)^{1/2} u$ corresponds to the Dirichlet-to-Neumann map for the harmonic extension of $u_{\Sph^1}$ to the unit disk. 
This map 
also has an integro-differential expression on $\Sph^1$ for  $u_{\Sph^1}$. It is noticeable that the expression that one gets for 
$u_{\Sph^1}$ on $\Sph^1$ is the same as the one for $u$ on 
$\R$, namely \eqref{defi laplacian_} with $s=1/2$, except that $\R$ must be replaced by $\Sph^1$ and 
$u$ by $u_{\Sph^1}$, but the kernel remains the same. The purpose of this section is to give the details of all these observations.

By definition,
\begin{equation}\label{defi laplacian.1/2}
(-\Delta)^{1/2} u(x):=\frac{1}{\pi}\,\text{P.V.}\!\int_\R dy\,
\frac{u(x)-u(y)}{|x-y|^{2}}
\end{equation} 
for all $u:\R\to\R$ whenever the integral and the limit make sense. Assume that $u$ is $2\pi$-periodic. Then, using the Fourier series expansion
$u(x)=\sum_{k\in\Z}u_ke^{ikx}$, where 
$u_k:=\frac{1}{2\pi}\intp dx\, u(x)e^{-ikx}$,
we know form \Cref{eigenfunctions nonlocal} that
\begin{equation}
(-\Delta)^{1/2} u(x)=\sum_{k\in\Z}|k|u_ke^{ikx}.
\end{equation}
To the $2\pi$-periodic function $u$, let us associate the function $u_{\Sph^1}:\Sph^1\subset\C\to\R$ given, for $p=e^{ix}$, by 
\begin{equation}
u_{\Sph^1}(p)=u_{\Sph^1}(e^{ix}):=u(x)
=\sum_{k\in\Z}u_ke^{ikx}
=\sum_{k\in\Z}u_k(e^{ix})^k
=\sum_{k\in\Z}u_kp^k.
\end{equation}

The function $u_{\Sph^1}$ can be naturally extended to a harmonic function in the unit disk $u_{\mathbb D}:\mathbb D\subset\C\to\R$ by writing $z=rp=re^{ix}\in\C$, where $r\geq0$ and $p=e^{ix}\in\Sph^1$, and 
\begin{equation}\label{def:U.ext}
\begin{split}
u_{\mathbb D}(z)&=u_{\mathbb D}(rp):=\sum_{k\in\Z}u_kr^{|k|}p^k
=\sum_{k\in\Z}u_kr^{|k|}(e^{ix})^k\\
&=\sum_{k\geq0}u_k(re^{ix})^k+\sum_{k>0}u_{-k}(re^{-ix})^k
=\sum_{k\geq0}u_kz^k+\sum_{k>0}u_{-k}\overline{z}^k.
\end{split}
\end{equation}
The fact that $\Delta u_{\mathbb D}=0$ in $\mathbb D$ follows directly from the right-hand side of \eqref{def:U.ext} and the fact that $\Delta=4\partial_z\partial_{\overline z}=4\partial_{\overline z}\partial_z$. It is also clear that $u_{\mathbb D}=u_{\Sph^1}$ on $\Sph^1$. Now, one can use this extension into the unit disk to interpret $(-\Delta)^{1/2} u$ as a Dirichlet-to-Neumann map on $\Sph^1$, that is,
\begin{equation}\label{def:U.ext.normal}
\begin{split}
(\partial_ru_{\mathbb D})(p)=\frac{\partial}{\partial r}\Big|_{r=1}u_{\mathbb D}(z)&
=\frac{\partial}{\partial r}\Big|_{r=1}
\sum_{k\in\Z}u_kr^{|k|}p^k
=\sum_{k\in\Z}|k|u_ke^{ikx}=(-\Delta)^{1/2} u(x).
\end{split}
\end{equation}

This Dirichlet-to-Neumann map 
$u_{\Sph^1}(p)\mapsto (\partial_ru_{\mathbb D})(p)$ 
on $\Sph^1$ can now be used to get an integro-differential expression on $\Sph^1$ for $(-\Delta)^{1/2}u_{\Sph^1}$. 
The following lemma shows that, as we claimed at the beginning of this section, the formula that one gets for 
$(-\Delta)^{1/2}u_{\Sph^1}$ on $\Sph^1$ is the same as \eqref{defi laplacian.1/2} for $(-\Delta)^{1/2}u$ on $\R$ replacing $\R$ by $\Sph^1$ and 
$u$ by $u_{\Sph^1}$, but leaving the same kernel.

\begin{lemma}\label{lm1:half.harm}
Given 
$u_{\Sph^1}\in C^{\alpha}(\Sph^1)$ for some $\alpha>1$, let $u_{\mathbb D}:\mathbb D\to\R$ be the harmonic extension of $u_{\Sph^1}$ to the unit disk. Then, 
\begin{equation}
(\partial_ru_{\mathbb D})(p)=\frac{1}{\pi}\,\operatorname{P.\!V.}\!
\int_{\Sph^1}dq\,\frac{u_{\Sph^1}(p)-u_{\Sph^1}(q)}{|p-q|^{2}}
=:(-\Delta)^{1/2}u_{\Sph^1}(p) 
\end{equation}
for all $p\in\Sph^1$, where the second equality must be understood as a definition of an operator $(-\Delta)^{1/2}$ on $\Sph^1$.
\end{lemma}

\begin{proof}
It is well known that one can write
$u_{\mathbb D}(z)=\int_{\Sph^1}dq\,P(z,q)u_{\Sph^1}(q)$
for all $z\in\mathbb D$, where $P$ denotes the Poisson kernel on the unit disk. It is given by
\begin{equation}
P(z,q):=\frac{1-|z|^2}{2\pi|z-q|^2}\quad 
\text{for all $z\in\mathbb D$ and $q\in\Sph^1$}.
\end{equation}
The fact that a harmonic function which is constant on $\Sph^1$ must be constant in $\mathbb D$ yields that 
$\int_{\Sph^1}dq\,P(z,q)=1$ for all $z\in\mathbb D$. Therefore, for every $p\in\Sph^1$ 
and every $\epsilon>0$
 we have
\begin{equation}\label{DtN.eq1}
\begin{split}
(\partial_ru_{\mathbb D})(p)
&=\lim_{\delta\downarrow0}
\frac{u_{\mathbb D}(p)-u_{\mathbb D}((1-\delta)p)}{\delta}
=\lim_{\delta\downarrow0}\int_{\Sph^1}dq\,(u_{\Sph^1}(p)-u_{\Sph^1}(q))\frac{1}{\delta}P((1-\delta)p,q)\\
&=\lim_{\delta\downarrow0}\Big(\int_{\Sph^1\cap\{|p-q|>\epsilon\}}+\int_{\Sph^1\cap\{|p-q|\leq\epsilon\}}\Big)dq\,(u_{\Sph^1}(p)-u_{\Sph^1}(q))\frac{1}{\delta}P((1-\delta)p,q)\\
&=:A_{\epsilon}+B_{\epsilon},
\end{split}
\end{equation}
where $A_{\epsilon}$ denotes the limit as $\delta\downarrow 0$ of the integral over $\Sph^1\cap\{|p-q|>\epsilon\}$, and $B_\epsilon$ the limit of the other integral.

To compute $A_{\epsilon}$ we use that $p\neq q$. Then, since $|p|=1$, we see that 
\begin{equation}\label{DtN.eq2}
\begin{split}
\lim_{\delta\downarrow0}\frac{1}{\delta}P((1-\delta)p,q)
=\lim_{\delta\downarrow0}\frac{1-|(1-\delta)p|^2}{2\pi\delta|(1-\delta)p-q|^2}
=\lim_{\delta\downarrow0}\Big(1-\frac{\delta}{2}\Big)\frac{1}{\pi|(1-\delta) p-q|^2}
=\frac{1}{\pi|p-q|^2}.
\end{split}
\end{equation}
By dominated convergence, we conclude that 
\begin{equation}\label{DtN.eq4}
\begin{split}
\lim_{\epsilon\downarrow0}A_{\epsilon}&=\lim_{\epsilon\downarrow0}
\int_{\Sph^1\cap\{|p-q|>\epsilon\}}dq\,(u_{\Sph^1}(p)-u_{\Sph^1}(q))\lim_{\delta\downarrow0}\frac{1}{\delta}P((1-\delta)p,q)\\
&=\frac{1}{\pi}\lim_{\epsilon\downarrow0}
\int_{\Sph^1\cap\{|p-q|>\epsilon\}}dq\,
\frac{u_{\Sph^1}(p)-u_{\Sph^1}(q)}{|p-q|^2}.
\end{split}
\end{equation}

Let us now compute $B_{\epsilon}$. To this end, we can assume without loss of generality that $1<\alpha\leq2$. By the symmetry of 
$\Sph^1\cap\{|p-q|\leq\epsilon\}$ with respect to the line 
$\R p$,
\begin{equation}\label{DtN.eq.aux.3}
\begin{split}
\int_{\Sph^1\cap\{|p-q|\leq\epsilon\}}dq\,
\frac{u_{\Sph^1}(p)-u_{\Sph^1}(q)}{|(1-\delta)p-q|^2}
=\frac{1}{2}\int_{\Sph^1\cap\{|p-q|\leq\epsilon\}}dq\,
\frac{2u_{\Sph^1}(p)-u_{\Sph^1}(q)-u_{\Sph^1}(q_s)}
{|(1-\delta)p-q|^2},
\end{split}
\end{equation}
where $q_s\in \Sph^1$ denotes the reflected point to $q$ with respect to $\R p$. On the one hand, since $\Sph^1$ has a tangent line at $p$, we see that the points $q$, $p$, and $q_s$ tend to be aligned as $q$ approaches~$p$. Indeed, 
$|p-q-(q_s-p)|=O(|p-q|^2)$. It is not hard to show that this, together with the fact that $u_{\Sph^1}\in C^{\alpha}(\Sph^1)$ for some $1<\alpha\leq2$, yields that 
\begin{equation}\label{DtN.eq.aux.1}
|2u_{\Sph^1}(p)-u_{\Sph^1}(q)-u_{\Sph^1}(q_s)|
=O(|p-q|^{\alpha})\quad\text{as $q$ tends to $p$.}
\end{equation} 
On the other hand, for every $0\leq \delta\leq 1$ and $q\in\Sph^1$ it holds that $\delta\leq |(1-\delta)p-q|$, and we therefore obtain that
\begin{equation}\label{DtN.eq.aux.2}
|p-q|\leq |p-(1-\delta)p|+|(1-\delta)p-q| =\delta+|(1-\delta)p-q|
\leq 2 |(1-\delta)p-q|.
\end{equation}
Finally, since
$\frac{1}{\delta}P((1-\delta)p,q)=\frac{1-{\delta}/{2}}{\pi|(1-\delta) p-q|^2}$, combining \eqref{DtN.eq.aux.3}, \eqref{DtN.eq.aux.1}, and \eqref{DtN.eq.aux.2} we get
\begin{equation}\label{gy:Bepsilon}
|B_{\epsilon}|\leq C\int_{\Sph^1\cap\{|p-q|\leq\epsilon\}}
dq\,|p-q|^{\alpha-2}
\leq C \epsilon^{\alpha-1}
\end{equation}
for some constant $C>0$ independent of $\epsilon$.
Thus, we conclude that 
$\lim_{\epsilon\downarrow 0}B_{\epsilon}=0$, 
and the lemma follows from this, \eqref{DtN.eq4}, and \eqref{DtN.eq1}.
\end{proof}

Combining the previous results, we deduce the following identities at the level of energies.

\begin{corollary}
Given $u:\R\to\R$ $2\pi$-periodic, let 
$u_{\Sph^1}:\Sph^1\to\R$ be defined by $u_{\Sph^1}(e^{ix}):=u(x)$ for all $x\in\R$, and let $u_{\mathbb D}:\mathbb D\to\R$ be the harmonic extension of $u_{\Sph^1}$ to the unit disk. 

Then,
\begin{equation}
\begin{split}
\EE_{(-\Delta)^{1/2}}(u)&=
\frac{1}{4\pi}\int_{-\pi}^\pi dx\int_{\R}dy\,
\frac{|u(x)-u(y)|^2}{|x-y|^{2}}
=\frac{1}{2}\int_{\mathbb D}|\nabla u_{\mathbb D}|^2\\
&=
\frac{1}{4\pi}
\int_{\Sph^1}dp\int_{\Sph^1}dq\,
\frac{|u_{\Sph^1}(p)-u_{\Sph^1}(q)|^2}{|p-q|^{2}},
\end{split}
\end{equation}
where $\EE_{(-\Delta)^{1/2}}$   
is the energy defined in \eqref{lagrangian rearrange mod u} with $G=0$.
\end{corollary}

\begin{proof}
Using \Cref{fl l1}, that $u(x)=u_{\mathbb D}(e^{ix})$, and \eqref{def:U.ext.normal} we have that
\begin{equation}\label{coro:half.harm.eq1}
\frac{1}{2\pi}\int_{-\pi}^\pi dx\int_{\R}dy\,
\frac{|u(x)-u(y)|^2}{|x-y|^{2}}
=\int_{-\pi}^\pi u\, (-\Delta)^{1/2}u
=\int_{\Sph^1} u_{\mathbb D}\, \partial_ru_{\mathbb D}.
\end{equation}
Now, the divergence theorem and the fact that $u_{\mathbb D}$ is harmonic shows that the right-hand side of \eqref{coro:half.harm.eq1} equals 
$\int_{\mathbb D}|\nabla u_{\mathbb D}|^2$.

Finally, using \Cref{lm1:half.harm} and since 
$u_{\mathbb D}=u_{\Sph^1}$ on $\Sph^1$, a standard symmetrization argument (as the one in 
\eqref{fl eq1} in the proof of 
 \Cref{fl l1}) shows that 
\begin{equation}
\int_{\Sph^1} u_{\mathbb D}\, \partial_ru_{\mathbb D}
=
\frac{1}{2\pi}
\int_{\Sph^1}dp\int_{\Sph^1}dq\,
\frac{|u_{\Sph^1}(p)-u_{\Sph^1}(q)|^2}{|p-q|^{2}},
\end{equation}
as claimed.
\end{proof}

\begin{remark}{\em 
We follow the notation above, 
$u_{\Sph^1}(e^{ix})=u(x)$. We finish this section by proving that 
$(-\Delta)^{1/2} u(x)=(-\Delta)^{1/2}u_{\Sph^1}(e^{ix})$ for all $x\in\R$ without using the extension argument above via $u_{\mathbb D}$, but directly at the level of integro-differential expressions. 

Note that, by the periodicity of $u$, we have
\begin{equation}\label{hhm:rmk.eq}
(-\Delta)^{1/2} u(x)=\frac{1}{\pi}\int_\R dy\,
\frac{u(x)-u(y)}{|x-y|^{2}}
=\frac{1}{\pi}\intp dy\,(u(x)-u(y))\sum_{k\in\Z}
\frac{1}{|x-y+2k\pi|^{2}}.
\end{equation}
Recall that 
\begin{equation}
  s^{-\gamma}=\frac{1}{\Gamma(\gamma)}\int_0^{+\infty}dt\,
  t^{\gamma-1}e^{-st}\quad\text{for all $s>0$.}
\end{equation}
Hence, assuming that $0< x-y< 2\pi$ (the case $-2\pi< x-y<0$ is analogous), we have
\begin{equation}
\begin{split}
\sum_{k\in\Z}\frac{1}{|x-y+2k\pi|^{2}}
&=\sum_{k\in\Z}\int_0^{+\infty}dt\,te^{-|x-y+2k\pi|t}\\
&=\int_0^{+\infty}dt\,te^{-(x-y)t}\sum_{k\geq0}(e^{-2\pi t})^k
+\int_0^{+\infty}dt\,te^{(x-y)t}\sum_{k<0}(e^{2\pi t})^{k}\\
&=\int_0^{+\infty}dt\,t\frac{e^{-(x-y)t}}{1-e^{-2\pi t}}
+\int_0^{+\infty}dt\,t\frac{e^{(x-y)t}e^{-2\pi t}}{1-e^{-2\pi t}}\\
&=\int_0^{+\infty}dt\,t\frac{e^{-(x-y-\pi)t}+e^{(x-y-\pi)t}}
{e^{\pi t}-e^{-\pi t}}
=\int_0^{+\infty}dt\,t\frac{\cosh((\pi-x+y)t)}{\sinh(\pi t)}.
\end{split}
\end{equation}

For every $0<r<2\pi$ it holds that (see \cite[Number 8 in Section 3.524 on page 377]{Table of Integrals})
$$\int_0^{+\infty}dt\,t\frac{\cosh((\pi-r)t)}{\sinh(\pi t)}
=\frac{1}{2-2\cos(r)},$$ 
and since $|e^{ix}-e^{iy}|^{2}=2-2\cos(x-y)$, we deduce that 
$$\sum_{k\in\Z}\frac{1}{|x-y+2k\pi|^{2}}=\frac{1}{|e^{ix}-e^{iy}|^{2}}.$$
Plugging this identity in \eqref{hhm:rmk.eq}, we conclude that
\begin{equation}
(-\Delta)^{1/2} u(x)
=\frac{1}{\pi}\intp dy\,
\frac{u_{\Sph^1}(e^{ix})-u_{\Sph^1}(e^{iy})}{|e^{ix}-e^{iy}|^{2}}
=\frac{1}{\pi}\int_{\Sph^1} dq\,
\frac{u_{\Sph^1}(e^{ix})-u_{\Sph^1}(q)}{|e^{ix}-q|^{2}}
=(-\Delta)^{1/2}u_{\Sph^1}(e^{ix}). 
\end{equation}
}\end{remark}

\section{On the class of kernels considered in \Cref{thm rearrangement} }
\label{section:appendix rearrangement.kernels}

In this section we shall discuss the three conditions in the rearrangement  \Cref{thm rearrangement}  imposed on the kernel $K$ by exhibiting several examples. In particular we show that the theorem is false in general for nonincreasing kernels, which is the most natural class which contains all three  conditions. We also show that neither of the three classes is contained in one of the others.   

Let us start by exhibiting a simple example which shows that the statement in \Cref{thm rearrangement}, namely, that $[u^{*\eper}]_K\leq [u]_K$ for every $2L$-periodic function 
$u:\R\to\R$ such that $u\in L^2(-L,L)$, does not hold for a nonincreasing kernel $K$. Take, for example, $K_{\epsilon}$ to be the characteristic function of $[0,L+\epsilon]$ for some $0<\epsilon<L$, and define $\overline{K}_{\epsilon}(t):=\sum_{k\in\Z}K_{\epsilon}(|t+2kL|)$ as in the proof of \Cref{thm rearrangement}. Since $K_{\epsilon}(|t|)=\chi_{[-L-\epsilon,L+\epsilon]}(t)$ for all $t\in\R$, we have that $\overline{K}_{\epsilon}$ is the even $2L$-periodic function given in $[-L,L]$ by 
$$
  \overline{K}_{\epsilon}(t)=
  \left\{\begin{array}{l}
  1\quad\text{ if }t\in(-L+\epsilon,L-\epsilon),
\\
  2\quad\text{ if }t\in[-L,-L+\epsilon]\cup [L-\epsilon,L].
  \end{array}\right.
$$
Note that in particular $\overline{K}_{\epsilon}$ is 
smaller in $(0,L-\epsilon)$ than in $(L-\epsilon,L).$ 
This last observation shows that our proof of  \Cref{thm rearrangement} cannot work for such a kernel, since the Riesz rearrangement inequality of \Cref{thm:gy:sharp Friedberg Luttinger} does not apply. In fact, we now show that the conclusion given in \Cref{thm rearrangement} is indeed false for the kernel
$K_{\epsilon}=\chi_{[0,L+\epsilon]}$. 

Indeed, by using \eqref{gy:expr of K by periodicity} for $[u]_{K_{\epsilon}}$, expanding the square $|u(x)-u(y)|^2$ as $u^2(x)+u^2(y)-2u(x)u(y)$, and using the fact that rearrangement does not modify the double integrals containing $u^2(x)$ and $u^2(y)$ (assuming $u$ to be in $L^2(-L,L)$ for the integrals to be finite) we see that $[u^{*\eper}]_{K_{\epsilon}}\leq [u]_{K_{\epsilon}}$ is equivalent to the inequality
\begin{equation}\label{eq:gy:example ineq K}
  \int_{-L}^L dx\int_{-L}^{L}dy\,
  u^{*\eper}(x)u^{*\eper}(y)\overline{K}_{\epsilon}(x-y)
  \geq
  \int_{-L}^L dx\int_{-L}^{L}dy\,
  u(x)u(y)\overline{K}_{\epsilon}(x-y).
\end{equation}
We now show that this inequality does not hold for some functions $u\in L^2(-L,L)$. For instance, take
$$
  u=a\chi_{(-\frac{L}{2}-\delta,-\frac{L}{2}+\delta)}+
  b\chi_{(\frac{L}{2}-\delta,\frac{L}{2}+\delta)},
$$
where $a$ and $b$ are two positive numbers and $0<\delta<L/2$.
If we chose $\delta$ sufficiently small compared to $\epsilon$, then 
$$
    \int_{-L}^L dx\int_{-L}^{L}dy\,
  u(x)u(y)\overline{K}_{\epsilon}(x-y)=
  4\delta^2\big(a^2\overline{K}_{\epsilon}(0)
  +b^2\overline{K}_{\epsilon}(0)+2ab\overline{K}_{\epsilon}(L)\big).
$$
The function $u^{*\eper}$ restricted to $(-L,L)$ is equal to   
$$
  u^{*\eper}=\max\{a,b\}\chi_{(-{\delta},{\delta})}
  +\min\{a,b\}\chi_{(-2\delta,-{\delta}]\cup[{\delta},2\delta)}
$$
 which, for $\delta$ small enough, leads to
$$
    \int_{-L}^L dx\int_{-L}^{L}dy\,
  u(x)^{*\eper}u^{*\eper}(y)\overline{K}_{\epsilon}(x-y)=
  4\delta^2\big(a^2\overline{K}_{\epsilon}(0)
  +b^2\overline{K}_{\epsilon}(0)+2ab\overline{K}_{\epsilon}(0)\big).
$$
Therefore, since $\overline K_{\epsilon}(L)=2>1=\overline K_{\epsilon}(0)$, \eqref{eq:gy:example ineq K} does not hold for this function $u$. Indeed, we have shown that for every choice of positive numbers $a$ and $b$ it holds that  
$[u^{*\eper}]_{K_{\epsilon}}> [u]_{K_{\epsilon}}$.

Let us now comment on the class of kernels $K$ considered in the statement of \Cref{thm rearrangement}. It is simple to show that the collection of kernels satisfying $(iii)$ does not contain and is not contained in the collection of kernels described neither by $(i)$ nor by $(ii)$.

We now show that the collection of kernels satisfying $(i)$ does not contain and is not contained in the collection of kernels described by $(ii)$. On the one hand, the kernel $K(t)=(t^2+1)^{-(1+2s)/2}$ satisfies the growth condition \eqref{defi laplacian_K growth1} with $\Lambda=1$, and belongs to the class $(ii)$ as we have seen right after  \Cref{def compl mon rearrangement} or in \Cref{remark:corrv15 remark 6.2}.
However, $K$ is strictly concave for $t$ close to the origin. Therefore, $K$ is as in $(ii)$ but it does not satisfy $(i)$, that is, $K$ is not convex.\footnote{Instead, recall that the function $\tau\mapsto K(\tau^{1/2})$ is convex, as follows from \Cref{def compl mon rearrangement}.}

On the other hand, if a kernel $K$ is as in $(ii)$ then it is infinitely differentiable in $(0,+\infty)$. This follows form \Cref{def compl mon rearrangement} or by differentiating \eqref{kernel Laplace transform}. 
At the same time, there exist convex (and even strictly convex) kernels which are not infinitely differentiable in $(0,+\infty)$. These kernels will fulfill $(i)$ but not $(ii)$. Nevertheless, the lack of regularity is not the only property that prevents a convex kernel to belong to the class 
$(ii)$. In \Cref{example i vs ii} we provide an example of a smooth kernel that fulfills $(i)$ but not~$(ii)$.

We finally mention that there exist kernels that satisfy both $(i)$ and $(ii)$, such as the kernel of the fractional Laplacian (see \Cref{remark:corrv15 remark 6.2}).

\begin{lemma}\label{example i vs ii}
Given $0<s<1$, set 
$$K(t)=\int_{t^2}^{+\infty}\!dr\int_r^{+\infty}\!dx\,\frac{2+\sin x}{x^{s+5/2}}\quad\text{for all }t>0.$$
Then, $K$ is infinitely differentiable in $(0,+\infty)$, satisfies \eqref{defi laplacian_K growth1} and \eqref{defi laplacian_K growth2}, is strictly convex in $(0,+\infty)$, but $\tau>0\mapsto K(\tau^{1/2})$ is not a completely monotonic function.
\end{lemma}

\begin{proof}
We first show that $K$ satisfies \eqref{defi laplacian_K growth1} and \eqref{defi laplacian_K growth2}. Since 
$1\leq2+\sin x\leq3$ for all $x>0$, 
\begin{equation}
\frac{t^{-1-2s}}{(s+\frac{3}{2})(s+\frac{1}{2})}=\int_{t^2}^{+\infty}\!dr\int_r^{+\infty}\!\frac{dx}{x^{s+5/2}}
\leq K(t)\leq3\int_{t^2}^{+\infty}\!dr\int_r^{+\infty}\!\frac{dx}{x^{s+5/2}}
=3\frac{t^{-1-2s}}{(s+\frac{3}{2})(s+\frac{1}{2})}.
\end{equation}
Moreover, since $x\mapsto x^{-s-5/2}(2+\sin x)$ is infinitely differentiable in $(0,+\infty)$, so is $K$. 

Let us now check that $K$ is strictly convex in $(0,+\infty)$. Set 
\begin{equation}
f(\tau):=\int_{\tau}^{+\infty}\!dr\int_r^{+\infty}\!dx\,\frac{2+\sin x}{x^{s+5/2}}\quad\text{for all }\tau>0.
\end{equation}
Then $K(t)=f(t^2)$, which yields
$K''(t)=4t^2f''(t^2)+2f'(t^2)$. Note that 
\begin{equation}
f'(\tau)=-\int_{\tau}^{+\infty}\!dx\,\frac{2+\sin x}{x^{s+5/2}}\quad \text{and}\quad
f''(\tau)=\frac{2+\sin \tau}{\tau^{s+5/2}}.
\end{equation}
Therefore,
\begin{equation}
\begin{split}
K''(t)&=4t^2f''(t^2)+2f'(t^2)
=4\frac{2+\sin (t^2)}{t^{2s+3}}
-2\int_{t^2}^{+\infty}\!dx\,\frac{2+\sin x}{x^{s+5/2}}\\
&\geq\frac{4}{t^{2s+3}}-2
\int_{t^2}^{+\infty}\!\frac{3\,dx}{x^{s+5/2}}
=\frac{4}{t^{2s+3}}-\frac{6t^{-2s-3}}{s+\frac{3}{2}}
=4t^{-2s-3}\Big(1-\frac{3}{2s+3}\Big)>0
\end{split}
\end{equation}
for all $t>0$. That is, $K$ is strictly convex in $(0,+\infty)$. 

Finally, it remains to check that $\tau\mapsto K(\tau^{1/2})=f(\tau)$ is not a completely monotonic function. If it was, then we would have
$f'''(\tau)\leq0$ for all $\tau>0$. However, 
\begin{equation}
f'''(\tau)=\frac{\cos \tau}{\tau^{s+5/2}}-\Big(s+\frac{5}{2}\Big)\frac{2+\sin \tau}{\tau^{s+7/2}}
=\frac{1}{\tau^{s+5/2}}\Big(\!\cos \tau-\Big(s+\frac{5}{2}\Big)\frac{2+\sin \tau}{\tau}\Big),
\end{equation}
which changes sign infinitely many times for $\tau$ big enough.
\end{proof}

\section{Stable periodic solutions are constant}
\label{section:appendix:minimizers const}

We show in this appendix that stable periodic solutions of $\LL_Ku=f(u)$ in $\R$  must be constant. Here the corresponding variational problem has no constraint. That is, stability means that the second variation of
$$
\EE_{\LL_K}:=\frac{1}{2}[u]_K^2-\int_{-L}^L F(u),
$$ 
at the $2L$-periodic solution $u$, is nonnegative
definite when acting on $2L$-periodic functions. This condition is stated in an equivalent way in \eqref{eq:gy:appen:D2}.
This result applies to all kernels having the standard growth bounds.

\begin{lemma}
\label{lemma:minimizers constant}
Let $K$ satisfy the bounds \eqref{defi laplacian_K growth1} and \eqref{defi laplacian_K growth2},  $L>0$, and $f=F'\in C^{1+\beta}(\R)$ for some $\beta>0$.

If $u\in L^\infty(\R)$ is a $2L$-periodic stable solution of $\LL_Ku=f(u)$ in $\R$, then $u$ is constant. 
\end{lemma}

\begin{proof}
As in the proof of \Cref{coro constrained minim lagrangian2} or of \Cref{thm: frac Laplacian Intro}, one can see that $u'\in C^{\nu}(\R)$ for some $\nu>2s$ and that $u'$ satisfies the linearized equation $\LL_Ku'=f'(u)u'$ in $\R$. In view of the stability of $u$, the second variation of $\EE_{\LL_K}$ at $u$ is nonnegative, i.e, 
\begin{equation}\label{eq:gy:appen:D2}
  D^2\EE_{\LL_K}(u)(\xi,\xi):=
  \frac{1}{2}\int_{-L}^L\!dx\int_{\R}dy\, {|\xi(x)-\xi(y)|^2}K(|x-y|)
- \int_{-L}^{L} f'(u)\xi^2\geq 0
\end{equation}
for all $2L$-periodic functions $\xi\in C^{2s+\epsilon}(\R)$ and for all $\epsilon>0$.

Multiplying by $u'$ the linearized equation satisfied by $u'$ and integrating in $(-L,L)$ we deduce from  the integration by parts formula \Cref{fl l1} (applied with $u$ and
$\psi$ both replaced by $u’$, a $2L$-periodic function of class $C^{\nu}(\R)$ with $\nu>2s$) that $D^2\EE_{\LL_K}(u)(u',u')=0.$ 
Since $||u'(x)|-|u'(y)||\leq |u'(x)-u'(y)|$ it follows that
$D^2\EE_{\LL_K}(u)(|u'|,|u'|)\leq D^2\EE_{\LL_K}(u)(u',u')=0,$ 
which in turn shows that 
$D^2\EE_{\LL_K}(u)(|u'|,|u'|)$ vanishes too, by \eqref{eq:gy:appen:D2} used with $\xi=|u'|$.
 
Thus the function 
$$
   w(t):=D^2\EE_{\LL_K}(u)(|u'|+t\eta,|u'|+t\eta)
$$
vanishes at $t=0$ and, in view of \eqref{eq:gy:appen:D2}, is nonnegative for all $t\in \R$ and all $2L$-periodic functions $\eta$. It follows that $w'(0)=0$. Computing $w'(0)$, integrating by parts, and using that $\eta$ is arbitrary, we obtain that $|u'|$ is a weak solution of 
\begin{equation}
    \LL_K |u'|-f'(u)|u'|=0 \text{ in }\R.
\end{equation}
Thus, by our regularity result \Cref{fthm:Holder reg periodic}, it follows that $|u'|$ is of classe $C^{\omega}$ for some $\omega>2s$. Therefore $\LL_K|u'|$ makes sense pointwise and $|u'|$ is a pointwise solution of the equation.

Now, by the well known strong maximum principle for the operator $\LL_K$ in $\R$, and since $|u'|\geq 0$ in $\R$, we must have that either $|u'|\equiv 0$ or $|u'|>0$ in $\R$ (we are not using here the strong maximum principle of \Cref{section:max principle} for periodic solutions, whose proof is harder). Now $|u'|>0$ in $\R$ is not possible, since $u$ is periodic. Thus $|u'|\equiv 0$ and hence $u$ is constant in $\R$. 
\end{proof}


\begin{thebibliography}{AHMTT}

\bibitem{Alvinya} M.\! Alviny\`a, {\em Delaunay cylinders with constant non-local mean curvature}, Master's Thesis, UPC (2017), \href{https://upcommons.upc.edu/handle/2117/104594?locale-attribute=en}{https://upcommons.upc.edu/handle/2117/104594?locale-attribute=en}. 

\bibitem{AmickToland} C.\! J.\! Amick, J.\! F.\! Toland, {\em Uniqueness and related analytic properties for the Benjamin-Ono equation ---a nonlinear Neumann problem in the plane}, Acta Math. 167 (1991), 107--126. doi:10.1007/BF02392447.

\bibitem{Baernstein Circle} A.\! Baernstein, \textit{Convolution and rearrangement on the circle}, Complex Variables, 12 (1989), 33--37.

\bibitem{Baernstein book} A.\! Baernstein, \textit{Symmetrization in analysis}, with David Drasin and Richard S.\! Laugesen, New Mathematical Monographs 36, Cambridge University Press, Cambridge, 2019. 

\bibitem{Baernstein Taylor} A.\! Baernstein, B.\! A.\! Taylor, \textit{Spherical rearrangements, subharmonic functions and $*$-functions in $n$-space}, Duke Math. J., 27 (1976), 233--249.

\bibitem{BGQ} B.\! Barrios, J.\! Garc\'ia-Meli\'an, A.\! Quaas, {\em Periodic solutions for the one-dimensional fractional Laplacian}, J. Diff. Equations, 267(9) (2019), 5258-5289, \href{https://doi.org/10.1016/j.jde.2019.05.031}{https://doi.org/10.1016/j.jde.2019.05.031}.

\bibitem{BenBobDharaGirg Polar} J.\! Benedikt, V.\! Bobkov, R.\! N.\! Dhara, P.\! Girg, {\em Nonradiality of second eigenfunctions of the fractional Laplacian in a ball}, Proc. Amer. Math. Soc. 150 (2022), no.12, 5335--5348.

\bibitem{Benj} T.\! B.\! Benjamin, {\em Internal waves of permanent form in fluids of great depth}, J. Fluid Mech., 29 (3), (1967) 559--592. doi:10.1017/S002211206700103X.

\bibitem{BBM} T.\! B.\! Benjamin, J.\! L.\! Bona, J.\! J.\! Mahony, {\em Model equations for long waves in nonlinear dispersive systems}, Phil. Trans. Royal Soc. London, Ser. A, 272 (1972), 47--78.

\bibitem{Ber_monotonic} S.\! Bernstein, {\em Sur les fonctions absolument monotones}, Acta Math. 52 (1929), 1--66.

\bibitem{BK} H.\! Brezis, T.\! Kato, {\em Remarks on the Schr\"odinger operator with singular complex potentials}, J. Math. Pures Appl. 58 (1979), 137--151.

\bibitem{BruellPei} G.\! Bruell, L.\! Pei, {\em Symmetry of periodic traveling waves for nonlocal dispersive equations}, SIAM J. Math. Anal. 55 (2023), no.1, 486--507.

\bibitem{BucurValdinoci} C.\! Bucur, E.\! Valdinoci, {\em Nonlocal Diffusion and Applications}, Springer International Publishing (2016), Print ISBN 978-3-319-28738-6.  
DOI: 10.1007/978-3-319-28739-3.

\bibitem{Burchard Hajaiej} A.\! Burchard, H.\! Hajaiej, {\em Rearrangement inequalities for functionals with monotone integrands}, J. Funct. Anal. 233 (2006), no. 3, 499--527.

\bibitem{Cabre Csato Mas delaunay} X.\! Cabr{\'e}, G.\! Csat\'o, A.\! Mas, {\em Existence and symmetry of periodic nonlocal-CMC surfaces via variational methods}, J. Reine Angew. Math., 804 (2023), 11--40, doi:10.1515/crelle-2023-0057.

\bibitem{Cabre Mas Hamilt} X.\! Cabr{\'e}, A.\! Mas, {\em Periodic solutions to integro-differential equations:
Hamiltonian structure}, forthcoming.

\bibitem{Cabre Mas Sola AC-BO} X.\! Cabr{\'e}, A.\! Mas, J.\! Sol\`a-Morales, {\em Periodic solutions to integro-differential equations: Allen-Cahn and Benjamin-Ono nonlinearities},  forthcoming.

\bibitem{Cabre Sire A} X.\! Cabr{\'e}, Y.\! Sire, {\em Nonlinear equations for fractional Laplacians, I: Regularity, maximum principles, and Hamiltonian estimates}, Ann. Inst. H. Poincaré C Anal. Non Linéaire 31 (2014), 23--53. 

\bibitem{Cabre Sola bdry reaction} X.\! Cabr\'e,  J.\! Sol\`a-Morales,  {\em Layer solutions in a half-space for boundary reactions}, Comm. Pure Appl. Math. 58 (2005), no.12, 1678--1732.

\bibitem{CB} H.\! Chen, J.\! Bona, {\em Periodic Traveling-wave Solutions of Nonlinear Dispersive Evolution Equations}, Discrete Cont.  Dynamical Systems Series A
33(11-12) (2013), 4841--4873.

\bibitem{Chen Li Ou}
W.\! Chen, C.\! Li, B.\! Ou, {\em Classification of solutions for an integral equation}, Comm. Pure Appl. Math. 59 (2006), no. 3, 330--343.

\bibitem{Claassen Johnson} K.\! Claassen, M.\! Johnson, {\em
Nondegeneracy and stability of antiperiodic bound states for fractional nonlinear Schrödinger equations}, J. Differential Equations 266 (2019), no.9, 5664--5712.

\bibitem{Cui Wang 2021} Y.\! Cui, Z.\! Wang, {\em Multiple periodic solutions of a class of fractional Laplacian equations}, Adv. Nonlinear Stud. 21 (2021), no. 1, 41--56.

\bibitem{CDK} G.\! Csat\'o, B.\! Dacorogna, O.\! Kneuss, {\em The pullback equation for differential forms}, Birkhauser/Springer, New York (2012). 

\bibitem{Csato Mas Holder opt} G.\! Csat\'o, A.\! Mas, {\em Examples of optimal Hölder regularity in semilinear equations involving the fractional Laplacian}, forthcoming.  

\bibitem{Davila et alt Delaunay} J.\! D\'avila, M.\! del Pino, S.\! Dipierro, E.\! Valdinoci, {\em Nonlocal Delaunay surfaces}, Nonlinear Analysis: Theory, Methods and Applications 137 (2016), 357--380.

\bibitem{DDGW} A.\! DelaTorre, M.\! del Pino, M.\! M.\! Gonz\'alez, J.\! Wei,
{\em Delaunay-type singular solutions for the fractional Yamabe problem}, Math. Ann. 369 (2017), 597--626. 

\bibitem{delatorre parini polar} A.\! DelaTorre, E.\! Parini, {\em Uniqueness of least energy solutions of the fractional Lane-Emden equation in the ball}, preprint (2023), \href{https://arxiv.org/abs/2310.02228}{https://arxiv.org/abs/2310.02228}.

\bibitem{DiebIanniSaldana polar} A.\! Dieb, I.\! Ianni, A.\! Salda\~na, {\em Uniqueness and nondegeneracy of least-energy solutions to fractional Dirichlet problems}, preprint (2023), \href{https://arxiv.org/abs/2310.01214}{https://arxiv.org/abs/2310.01214}.

\bibitem{ValBook} S.\! Dipierro, M.\! Medina, E.\! Valdinoci,
{\em Fractional Elliptic Problems with Critical Growth in the Whole of $\R^n$}, Lecture Notes SNS, 15 (2017), Scuola Normale Superiore Pisa,
ISBN 978-88-7642-601-8.

\bibitem{DRSV} S.\! Dipierro, X.\! Ros-Oton, J.\! Serra, E.\! Valdinoci, {\em Non-symmetric stable operators: regularity theory and integration by parts}, Advances in Math., (2022). \href{https://doi.org/10.1016/j.aim.2022.108321}{https://doi.org/10.1016/j.aim.2022.108321}. 

\bibitem{Valdinoci1} E.\! Di Nezza, G.\! Palatucci, E.\! Valdinoci, {\em Hitchhiker's guide to the fractional Sobolev spaces}, Bull. Scie. Math. 136-5 (2012) 521--573.

\bibitem{DG} Z. Du, C. Gui, {\em Further study on periodic solutions of elliptic equations with fractional Laplacian}, Nonlinear Anal. 193 (2020), 111417, 16 pp.

\bibitem{Fall Weth 2023} M.\!  Fall, T.\! Weth, {\em Uniqueness and nondegeneracy of solutions to $(-\Delta)^su+u=u^p$ in $\R^N$ and in balls}, preprint (2023), \href{https://arxiv.org/abs/2310.10577}{https://arxiv.org/abs/2310.10577}.

\bibitem{Frank Lenzmann 1} R.\! Frank, E.\! Lenzmann, {\em
Uniqueness of non-linear ground states for fractional Laplacians in $\R$}, Acta Math. 210 (2013), no.2, 261--318.

\bibitem{Frank Lenzmann Silvestre} R.\! Frank, E.\! Lenzmann, L.\! Silvestre, {\em Uniqueness of radial solutions for the fractional Laplacian}, Comm. Pure Appl. Math. 69 (2016), no.9, 1671--1726.

\bibitem{FriedLutt} R.\! Friedberg, J.\! M.\! Luttinger, {\em Rearrangement inequality for periodic functions}, Arch. Rat. Mech. Anal. 61 (1976), 35--44.

\bibitem{Table of Integrals}I.\! S.\! Gradshteyn, I.\! M.\! Ryzhik, {\em Table of integrals, series, and products}, Seventh Edition, Academic Press (Elsevier), 2007. 

\bibitem{GZD} C.\! Gui, J.\! Zhang, Z. Du, {\em Periodic solutions of a semilinear elliptic equation with a fractional Laplacian}, J. Fixed Point Theory Appl. (2017) 19: 363. https://doi.org/10.1007/s11784-016-0357-1.

\bibitem{Lax}P.\! Lax, {\em Functional analysis}, Pure Appl. Math., Wiley-Interscience, New York, 2002. 

\bibitem{Li Wang} D.\! Li, K.\! Wang, {\em Symmetric radial decreasing rearrangement can increase the fractional Gagliardo norm in domains}, Commun. Contemp. Math. 21 (2019), no. 7, 1850059, 9 pp.

\bibitem{Lieb Loss} E.\! Lieb, M.\! Loss, \textit{Analysis}, Second Edition,  Graduate Studies in Mathemtics 14, American Math. Soc, 2010.

\bibitem{Nab} F.\! R.\! N.\! Nabarro, {\em Dislocations in a simple cubic lattice}, Proc. Phys. Soc. 59 (1947), 256--272.

\bibitem{Ono}  H.\! Ono, {\em Algebraic Solitary Waves in Stratified Fluids}, J. Phys. Soc. Japan 39 (1975), 1082--1091.

\bibitem{RS} X.\! Ros-Oton, J.\! Serra, {\em Regularity theory for general stable operators}, J. Differential Equations 260 (2016), 8675--8715.

\bibitem{S1} J.\! Serra, {\em $C^{\sigma+\alpha}$ regularity for concave nonlocal fully nonlinear elliptic equations with rough kernels}, Calc. Var. Partial Differential Equations 54 (2015), 3571--3601.

\bibitem{Shahgholian 2015}H.\! Shahgholian,
{\em Regularity issues for semilinear PDE-s (a narrative approach)}, Algebra i Analiz 27 (2015), no. 3, 311--325.

\bibitem{Struwe} M.\! Struwe, {\em Variational methods. Applications to Nonlinear Partial Differential Equations and Hamiltonian Systems}, second edition: Ergebnisse Math. 34, Springer (1996).

\bibitem{Toland} J.\! F.\! Toland, {\em The Peierls-Nabarro and Benjamin-Ono Equations}, Journal of Functional Analysis 145(1) (1997), 136--150, https://doi.org/10.1006/jfan.1996.3016.

\bibitem{Wid_monotonic} D.V.\! Widder, {\em The Laplace transform}, Princeton: Princeton University Press; London, H. Milford, Oxford University Press, (1941).

\bibitem{Zygmund} A.\! Zygmund, {\em Trigonometric Series}, Third Edition: Volumes I \& II combined (Cambridge Mathematical Library), Cambridge: Cambridge University Press, (2002). 

\end{thebibliography}
\end{document}